\documentclass[12pt, a4paper]{amsart}
\usepackage[dvips]{epsfig}
\usepackage{amsmath}
\usepackage{amssymb}
\usepackage{amsfonts}
\usepackage{dsfont}
\usepackage{bbm}
\usepackage{amsthm}
\usepackage{amsbsy}
\usepackage{amsgen}
\usepackage{amscd}
\usepackage{amsopn}
\usepackage{amstext}
\usepackage{amsxtra}
\usepackage[utf8]{inputenc}
\usepackage{enumerate}
\usepackage{hyperref}
\usepackage{lineno}
\usepackage{marginnote}
\usepackage{verbatim}
\usepackage{calrsfs}
\usepackage{times, graphicx}
\usepackage{layout}
\usepackage{subcaption}
\usepackage{eso-pic}
\usepackage{lastpage}
\DeclareMathAlphabet{\pazocal}{OMS}{zplm}{m}{n}
\usepackage[foot]{amsaddr}

  \evensidemargin
\oddsidemargin
\numberwithin{equation}{section}

\makeatletter
\newcommand*\rel@kern[1]{\kern#1\dimexpr\macc@kerna}
\newcommand*\widebar[1]{%
  \begingroup
  \def\mathaccent##1##2{%
    \rel@kern{0.8}%
    \overline{\rel@kern{-0.8}\macc@nucleus\rel@kern{0.2}}%
    \rel@kern{-0.2}%
  }%
  \macc@depth\@ne
  \let\math@bgroup\@empty \let\math@egroup\macc@set@skewchar
  \mathsurround\z@ \frozen@everymath{\mathgroup\macc@group\relax}%
  \macc@set@skewchar\relax
  \let\mathaccentV\macc@nested@a
  \macc@nested@a\relax111{#1}%
  \endgroup
}
\makeatother

\newtheorem{theorem}{Theorem}[section]
\newtheorem{proposition}[theorem]{Proposition}
\newtheorem{lemma}[theorem]{Lemma}
\newtheorem{corollary}[theorem]{Corollary}

\theoremstyle{definition}
\newtheorem{definition}[theorem]{Definition}

\newtheorem{remark}[theorem]{Remark}

\newcommand{\vol}{\operatorname{Vol}}

\newcommand{\Dc}{\mathcal{D}}
\newcommand{\Dpc}{\pazocal{D}}
\newcommand{\Gc}{\mathcal{G}}

\newcommand{\Sc}{\pazocal{S}}

\newcommand{\Ec}{\pazocal{E}}
\newcommand{\tr}{\operatorname{tr}}

\newcommand{\Rc}{\mathcal{R}}

\newcommand{\Qc}{\pazocal{Q}}
\newcommand{\Mcc}{\mathcal{M}}
\newcommand{\Hc}{\mathcal{H}}

\newcommand{\meas}{\operatorname{meas}}

\newcommand {\E} {\mathbb{E}}
\newcommand {\M} {\pazocal{M}}
\newcommand {\R} {\mathbb{R}}

\newcommand {\Z} {\mathbb{Z}}

\newcommand {\Cc} {\pazocal{C}}
\newcommand {\Lc} {\mathcal{L}}

\newcommand {\C} {\mathbb{C}}

\newcommand {\Zc} {\mathcal{Z}}

\newcommand {\Ac} {\pazocal{A}}

\newcommand {\Ic} {\mathcal{I}}

\newcommand {\Tb} {\mathbb{T}}
\newcommand {\var} {\operatorname{Var}}
\newcommand {\len} {\operatorname{len}}

 \allowbreak

\begin{document}

\title[Boundary effect and spectral semi-correlations]{Boundary effect on the nodal length for Arithmetic Random Waves, and spectral semi-correlations}

\author{Valentina Cammarota\textsuperscript{1}}
\email{valentina.cammarota@uniroma1.it}
\address{\textsuperscript{1}Department of Statistics, Sapienza University of Rome}
\author{Oleksiy Klurman\textsuperscript{2}}
\email{lklurman@gmail.com}
\address{\textsuperscript{2}Department of Mathematics, KTH Royal Institute of Technology, Stockholm}
\author{Igor Wigman\textsuperscript{3}}
\email{igor.wigman@kcl.ac.uk}
\address{\textsuperscript{3}Department of Mathematics, King's College London}
\dedicatory{Dedicated to the memory of Jean Bourgain}

\date{\today}

\begin{abstract}

We test M. Berry's ansatz on nodal deficiency in presence of boundary.
The square billiard is studied, where the high spectral degeneracies allow for the introduction of a Gaussian ensemble of
random Laplace eigenfunctions (``boundary-adapted arithmetic random waves").
As a result of a precise asymptotic analysis, two terms in the asymptotic expansion
of the expected nodal length are derived, in the high energy limit along a generic sequence of energy levels.
It is found that the precise nodal deficiency or surplus of the nodal length
depends on arithmetic properties of the energy levels, in an explicit way.

To obtain the said results we apply the Kac-Rice method for computing the expected
nodal length of a Gaussian random field. Such an application uncovers major obstacles, e.g.
the occurrence of ``bad" subdomains, that, one hopes, contribute insignificantly to the nodal length.
Fortunately, we were able to reduce
this contribution to a number theoretic question of counting the ``spectral semi-correlations",
a concept joining the likes of ``spectral correlations" and ``spectral quasi-correlations" in having
impact on the nodal length for arithmetic dynamical systems.

This work rests on several breakthrough techniques of J. Bourgain, whose interest in the subject helped shaping it to high extent,
and whose fundamental work on spectral correlations, joint with E. Bombieri, has had a crucial impact
on the field.

\end{abstract}
	
\maketitle

\section{Introduction}

\subsection{Nodal length of Laplace eigenfunctions}

Let $f:\M\rightarrow\R$ be a smooth function on a compact smooth Riemannian surface $(\M,g)$,
with or without boundary, with no critical zeros. The zero set of
$f$, called the {\em nodal line} is a {\em smooth} curve with no self-intersections; it is an important qualitative descriptor of $f$.
We are interested in the geometry of the nodal lines of Laplace eigenfunctions on $\M$, in the high energy limit, i.e. the solutions
$\{(\varphi_{j},\lambda_{j})\}_{j\ge 1}$ of the Schr\"{o}dinger equation
\begin{equation}
\label{eq:schrodinger}
\Delta \varphi_{j}+\lambda_{j}\varphi_{j}=0,
\end{equation}
satisfying either Dirichlet or Neumann boundary conditions,
where $\Delta=\operatorname{div}\circ \nabla$ is the Laplace-Beltrami (Laplacian) operator on $\M$, and $\lambda_{j}\ge 0$ are
the energy levels (or simply the energies). It is well-known
that the spectrum of $\Delta$ is purely discrete, i.e. there exists a complete orthonormal system $\{\varphi_{j}\}_{j\ge 1}$
(orthonormal basis),
spanning the whole of $L^{2}(\M)$, so that all the spectral multiplicities are finite, and $\lambda_{j}\rightarrow\infty$
as $j\rightarrow\infty$ being the high energy limit.

Much of the focus in the study of the nodal lines of Laplace eigenfunctions has been turned to the study of the
{\em nodal length}, i.e. the total length $\Lc(\varphi_{j})$ of the curve $\varphi_{j}^{-1}(0)$ on $\M$, as $j\rightarrow\infty$.
Yau's conjecture ~\cite{Yau} asserts that the nodal length is commensurable with $\sqrt{\lambda_{j}}$, i.e.
\begin{equation}
\label{eq:Yau conj}
c_{\M}\cdot \sqrt{\lambda_{j}} \le \Lc(\varphi_{j}) \le C_{\M}\cdot \sqrt{\lambda_{j}}
\end{equation}
for some positive constants $c_{\M},C_{\M}>0$. Yau's conjecture was proven ~\cite{Bruning,Bruning-Gromes,DF} for $\M$ analytic, and more recently
the optimal lower bound ~\cite{Log Lower} and polynomial upper bound ~\cite{Log Upper} were established for the more general, smooth, case
(see also ~\cite{Log Mal}).

\subsection{Nodal length for random fields}

\label{sec:Nod length RWM}

One way to obtain stronger (or more precise) results than \eqref{eq:Yau conj} is to study the nodal length $\Lc(f)$ of {\em random} functions $f$,
an approach that has been actively pursued, in particular in the recent few years. As a concrete direction of research within the indicated scope,
one may take a Gaussian random field $f:\R^{2}\rightarrow\R$ (or $f:\R^{d}\rightarrow\R$, $d\ge 2$)
and study the distribution of the nodal length $\Lc(f;R)$ of $f$ restricted to
$B(R)\subseteq \R^{2}$, the radius-$R$ centred ball, $R\rightarrow\infty$. For\footnote{From this point we will tacitly assume that
all the involved random fields are sufficiently smooth and are satisfying some non-degeneracy assumptions.}
$f:\R^{2}\rightarrow\R$ {\em stationary} a
straightforward application of the standard Kac-Rice formula yields a precise expression
\begin{equation}
\label{eq:exp len stat}
\E[\Lc(f;R)]= c\cdot \vol(B(R)),
\end{equation}
whereas a significantly heavier machinery involving perturbation theory (and asymptotic analysis of the $2$-point correlation
function) yields an asymptotic expression for the variance
$$\var(\Lc(f;R)),$$ as $R\rightarrow\infty$. One may go further by applying the Wiener chaos decomposition on $\Lc(f;R)$ to
obtain ~\cite{KratzLeon} a limit law for the distribution of $$\frac{\Lc(f;R)}{\sqrt{\var(\Lc(f;R))}}.$$
Alternatively to working with a fixed random field restricted to expanding balls, one may fix a compact surface $\M$, consider a Gaussian ensemble of random functions on $\M$, i.e. a sequence $f_{n}:\M\rightarrow\R$ of Gaussian random fields indexed by $\M$, and study the asymptotic distribution of the nodal length of $f_{n}$, that is the total length $\Lc(f_{n})$
of $f_{n}^{-1}(0)$, as $n\rightarrow \infty$; in some natural examples (to be discussed
below) $f_{n}$ possesses a natural scaling with $n$.

\vspace{2mm}

Berry ~\cite{Berry 1977} suggested that for $\M$ generic {\em chaotic},
there exists a (non-rigorous) {\em link} between the (deterministic) eigenfunctions
$\varphi_{j}$ as in \eqref{eq:schrodinger}, and the restriction of monochromatic isotropic random wave $g$ (a particular random field on $\R^{2}$ to be defined immediately below), to $B(R)$ with $R\approx \sqrt{\lambda_{j}}$; this vague relation, usually referred to
as ``Berry's Random Wave Model" (RWM), agreed in a wide community, is subject to many numerical tests with overwhelmingly positive outcomes.
In particular, the study of the nodal structures of $g$ restricted to $B(R)$ as $R\rightarrow\infty$
facilitates our understanding of the nodal structures of $\varphi_{j}$ in the high energy limit. Berry's monochromatic isotropic
random wave $g$ is uniquely
defined as the centred Gaussian random field on $\R^{2}$ with covariance function $$r_{g}(x-y)=r_{g}(x,y):= \E[g(x)\cdot g(y)]=J_{0}(|x-y|),$$ $x,y\in\R^{2}$,
whose Fourier transform on $\R^{2}$ is the arc length of the unit circle (meaning that the monochromatic waves are propagating uniformly in all directions).
Since $r_{g}$ depends only on the Euclidean distance $|x-y|$, the law of $g$ is invariant under
all translations $g(\cdot)\mapsto g(\cdot+z)$, $z\in\R^{2}$, and rotations $g(\cdot)\mapsto g(o\,\cdot)$, $o\in O(2)$ (i.e. $g$ is stationary isotropic); it has applications in the study of ocean waves propagating ~\cite{LH1,LH2}.

Consistent to the above \eqref{eq:exp len stat},
the expected nodal length for this stationary model is easily found to be
\begin{equation}
\label{eq:exp len RWM}
\E[\Lc(g;R)] = c_{0}\cdot \vol(B(R)),
\end{equation}
with $c_{0}>0$ explicitly evaluated, via a straightforward application of the Kac-Rice formula.
Berry ~\cite[Formula (28)]{Berry 2002} further found that the variance is {\em logarithmic}, i.e. satisfying the asymptotic law
\begin{equation}
\label{eq:Berry var log}
\var(\Lc(g;R))= c_{1}\cdot R^{2}\log{R}+O_{R\rightarrow\infty}(R^{2}),
\end{equation}
much smaller than one would expect, due to an ``obscure cancellation" (see also ~\cite{wig}).

\vspace{2mm}

The central objective of this manuscript is investigating the effect of nontrivial boundary on the nodal structures
of Laplace eigenfunctions, first and foremost on the nodal length, either in the vicinity of the boundary, or globally.
Berry argued that, since the nodal line is perpendicular to the boundary ~\cite{Cheng} (except for intersection points
with higher degree vanishing),
its presence should impact its length negatively compared to \eqref{eq:exp len RWM}, he referred to as {\em ``nodal deficiency"}.
He backed this ansatz by a precise evaluation of the secondary term around the boundary for the ``boundary-adapted random waves", a Gaussian random field constrained to satisfy the boundary conditions, Dirichlet or Neumann, on an infinite straight line.

It was concluded that, bearing in mind that the primary term in his asymptotic expansion of nodal length for this boundary-adapted case is consistent to \eqref{eq:exp len RWM}, whereas the secondary term was, a large number of wavelengths away from the boundary, negative (identical between Dirichlet and Neumann boundary conditions), with total contribution in absolute value larger
than the length fluctuations in \eqref{eq:Berry var log} ({\em possibly}
extending to the boundary-adapted case), based on one sample only,
one should be able to detect the deficiency of the total nodal length of Laplace eigenfunctions for surfaces with boundary compared to the boundary-less case. Gnutzmann and Lois ~\cite{GL} supported Berry's deficiency ansatz by
performing a mean nodal volume calculation
for $\M$ cuboid of arbitrarily high dimension, a dynamical system with separation of variables, while averaging w.r.t. energy levels (rather than w.r.t. a Gaussian ensemble).

\subsection{Arithmetic Random Waves}
\label{sec:KKW ARW}

The ``usual" Arithmetic Random Waves are random toral Laplace eigenfunctions. Let
\begin{equation}
\label{eq:S=2 squares}
S=\{a^{2}+b^{2}:a,b\in\Z\}
\end{equation}
be the set of all integers expressible as sum of two squares, $n\in S$, and
\begin{equation}
\label{eq:Ec lattice pnts}
\Ec_{n}=\{ \mu=(\mu_{1},\mu_{2})\in\Z^{2}:\mu_{1}^{2}+\mu_{2}^{2}=n \}
\end{equation}
the lattice points set lying on the centred
radius-$\sqrt{n}$ circle.
It is well-known that $n\in S$, if and only if the prime decomposition of $n$ is of the form
\begin{equation}
\label{eq:n in S factor primes}
n=2^{a}\cdot \prod\limits_{j=1}^{s}p_{j}^{e_{j}} \prod\limits_{k=1}^{r}q_{k}^{2h_{k}},
\end{equation}
for some nonnegative integers $a$, $\{e_{j}\}_{j\le s}$, $\{h_{k}\}_{k\le r}$, and $p_{j}\equiv 1\mod 4$
and $q\equiv 3 \mod 4$ are primes.

By a classical result due to E. Landau ~\cite{Landau} the sequence $S\subseteq \Z$ is thin, i.e.
of asymptotic density $0$, and, what is stronger,
\begin{equation*}
\frac{|\{n\le X:\:n\in S\}|}{X} \: {{\underset{X\to\infty}\sim}} \: c_{0}\cdot \frac{X}{\sqrt{\log{X}}},
\end{equation*}
with some semi-explicit constant $c_{0}>0$.
Any function $g_{n}:\Tb^{2}\rightarrow\R$ on the torus $\Tb^{2}=\R^{2}/\Z^{2}$ of the
form
\begin{equation}
\label{eq:fn eigfunc tor ARW}
g_{n}(x)= \frac{1}{\sqrt{N_{n}}}\sum\limits_{\mu\in\Ec_{n}}a_{\mu}e(\langle x,\mu \rangle),
\end{equation}
where $a_{\mu}\in\C$ are some constants satisfying the condition
\begin{equation}
\label{eq:a(-mu)=conj(a(mu))}
a_{-\mu}=\overline{a_{\mu}}
\end{equation}
and $N_{n}=|\Ec_{n}|$
is the size of the lattice points set $\Ec_{n}$ (equivalently, $N_{n}=r_{2}(n)$, the number of ways to express $n$ as a sum of two squares),
is a real-valued Laplace eigenfunction with eigenvalue
\begin{equation}
\label{eq:lambda'n def}
\lambda'_{n}=4\pi^{2}n.
\end{equation}
The convenience pre-factor $\frac{1}{\sqrt{N_{n}}}$ on the r.h.s. of \eqref{eq:fn eigfunc tor ARW} has no bearing on the nodal set of $g_{n}$, and will be understood below. Conversely, any real-valued Laplace eigenfunction on $\Tb^{2}$ is necessarily of the form \eqref{eq:fn eigfunc tor ARW} for some $n\in S$.

\vspace{2mm}

The linear space of functions \eqref{eq:fn eigfunc tor ARW} may be endowed with a probability measure by making the coefficients $\{a_{\mu}\}_{\mu\in\Ec_{n}}$ i.i.d. standard complex valued Gaussian random variables, save for the condition \eqref{eq:a(-mu)=conj(a(mu))} to ensure the $g_{n}$ are real-valued; this model is called ``Arithmetic Random Waves".
Alternatively and equivalently, Arithmetic Random Waves is the Gaussian ensemble of centred stationary random fields with the covariance functions
\begin{equation}
\label{eq:pn covar ARW def}
p_{n}(x)=\E[g_{n}(x)\cdot g_{n}(0)] =\frac{1}{N_{n}}\sum\limits_{\mu\in\Ec_{n}} \cos(2\pi\langle x,\mu\rangle);
\end{equation}
its (random) nodal length $\Zc_{n}=\Lc(g_{n})$ on $\Tb^{2}$ is our etalon, representing the boundary-less cases for comparison against the appearance of nontrivial boundary. As it is the case with stationary random fields, it is easy to evaluate its expected nodal length to be
\begin{equation}
\label{eq:exp nodal len torus}
\E[\Zc_{n}]=\frac{\sqrt{\lambda'_{n}}}{2\sqrt{2}\sqrt{n}}=\frac{\pi}{\sqrt{2}}\cdot \sqrt{n}.
\end{equation}
Rudnick and Wigman ~\cite{RW08} gave the useful upper bound
\begin{equation}
\label{eq:var(Lc)<<1/sqrt(N) RW}
\var(\Zc_{n}) \ll_{N_{n}\rightarrow\infty} \frac{n}{\sqrt{N_{n}}},
\end{equation}
showing, in particular, that the distribution of $\frac{\Zc_{n}}{\sqrt{n}}$ concentrates around the constant $\frac{\pi}{\sqrt{2}}$.

Krishnapur-Kurlberg-Wigman ~\cite{KKW} further resolved the question of the true asymptotic behaviour of the variance on the l.h.s. of
\eqref{eq:var(Lc)<<1/sqrt(N) RW}, requiring the following background in the two squares problem.
For every $n$ we define the atomic probability measure
\begin{equation}
\label{eq:nun discrete prob}
\nu_{n}=\frac{1}{N_{n}}\sum\limits_{\mu\in\Ec_{n}}\delta_{\mu/\sqrt{n}}
\end{equation}
on the unit circle $\Sc^{1}$, supported on the angles of $\Sc^{1}$ corresponding to points of $\Ec_{n}$. It is known that
for a ``generic" sequence $\{n\}\subseteq S$ the angles $\{\mu/\sqrt{n}\}_{\mu\in\Ec_{n}}$ equidistribute on $\Sc^{1}$,
i.e. there exists a relative density $1$ sequence $\{n\}\subseteq S$, so that
\begin{equation}
\label{eq:nun=>dt/2pi}
\nu_{n}\Rightarrow \frac{d\theta}{2\pi},
\end{equation}
with $`\Rightarrow'$ standing for the weak-$*$ convergence of probability measures on $\Sc^{1}$, and in particular $N_{n}\rightarrow\infty$.
However, there exist ~\cite{Cil,KKW,KuWi} other {\em attainable} measures, i.e. weak-$*$ partial limits of the
sequence $\{\nu_{n}\}_{n\in S}$, and even under the (generic) constraint $N_{n}\rightarrow\infty$, the accumulation set of the sequence
$\{\widehat{\nu_n}(4)\}$ of the $4$th Fourier coefficients of $\nu_{n}$ (being the first nontrivial Fourier coefficient)
is the whole of $[-1,1]$. The said work
~\cite{KKW} established the {\em precise} asymptotic relation
\begin{equation}
\label{eq:var asymp ARW KKW}
\var(\Zc_{n})\, {{\underset{N_{n}\to\infty}\sim}}\, \frac{\pi^{2}}{128} \left(1+\widehat{\nu_{n}}(4)^{2}\right)\cdot  \frac{n}{N_{n}^{2}},
\end{equation}
and, bearing in mind that $1+\widehat{\nu_{n}}(4)^{2}$ is bounded away from both $0$ and infinity, in particular,
it shows that the fluctuations around the mean of $\Zc_{n}$ are of the order of magnitude
\begin{equation}
\label{eq:ARW fluct sqrt(n)/N}
\Zc_{n}-\E[\Zc_{n}]\approx \frac{\sqrt{n}}{N_{n}},
\end{equation}
important below, just like \eqref{eq:Berry var log}, due to an unexpected cancellation (``arithmetic Berry's cancellation").
Later a non-universal limit theorem for $\Zc_{n}$ was obtained ~\cite{MRW}.

\subsection{Boundary-adapted Arithmetic Random Waves}

The boundary-adapted Arithmetic Random Waves are random Laplace eigenfunctions on the unit square $\Qc=[0,1]^{2}$,
subject to Dirichlet boundary conditions. Let $S$ be the as above \eqref{eq:S=2 squares}, $n\in S$, and
$\mu=(\mu_{1},\mu_{2})\in \Ec_{n}$
a lattice point with $\Ec_{n}$ given by \eqref{eq:Ec lattice pnts}. Any function of the form $\Qc\rightarrow\R$
\begin{equation}
\label{eq:x->sin(l1x1)sin(l2x2)}
x=(x_{1},x_{2})\mapsto \sin(\pi \mu_{1}x_{1})\cdot \sin(\pi \mu_{2}x_{2})
\end{equation}
is a Laplace eigenfunction with eigenvalue
\begin{equation}
\label{eq:lambdan def}
\lambda_{n}=\pi^{2}n,
\end{equation}
satisfying the Dirichlet boundary conditions on $\Qc$ (cf. \eqref{eq:lambda'n def}).
However, given $\mu=(\mu_{1},\mu_{2})\in\Ec_{n}$ and
$\mu'=(\pm\mu_{1},\pm\mu_{2})\in\Ec_{n}$, the resulting maps as in \eqref{eq:x->sin(l1x1)sin(l2x2)}
differ at most by a sign. Therefore, to avoid redundancies, we introduce the equivalence relation on $\Ec$:
$(\mu_{1},\mu_{2})\sim (\mu_{1}',\mu_{2}')$ if $\mu_{1}=\pm \mu_{1}'$ and
$\mu_{2}=\pm\mu_{2}'$.

The general form of a Laplace eigenfunction on $\Qc$ satisfying Dirichlet boundary conditions assumes the form
\begin{equation}
\label{eq:BAARW fn def}
f_{n}(x) = \frac{4}{\sqrt{N_{n}}}\sum\limits_{\mu\in\Ec_{n}/\sim}a_{\mu}\sin(\pi\mu_{1}x_{1})\sin(\pi\mu_{2}x_{2}),
\end{equation}
for some $n\in S$, and we endow this linear space with a Gaussian probability measure by making the $\{a_{\mu}\}_{\mu\in\Ec_{n}/\sim}$
i.i.d. standard (real) Gaussian random variables. If either $\mu_{1}=0$ or $\mu_{2}=0$, then the corresponding
summand in \eqref{eq:BAARW fn def} vanish, so that we are allowed to assume that both $\mu_{1}\ne 0$ and $\mu_{2}\ne 0$.
We then call the random function \eqref{eq:BAARW fn def} equipped with the said Gaussian probability measure
``boundary-adapted Arithmetic Random Waves", much like the Arithmetic Random Waves \eqref{eq:fn eigfunc tor ARW}.
Alternatively (and equivalently), the boundary-adapted Arithmetic Random Waves is the ensemble of Gaussian centred random fields indexed by
$x\in \Qc$,
with covariance functions
\begin{equation}
\label{eq:rn covar def}
r_{n}(x,y)=\E[f_{n}(x)\cdot f_{n}(y)] = \frac{16}{N_{n}}\sum\limits_{\mu\in\Ec_{n}/\sim}\sin(\pi \mu_{1}x_{1})\sin(\pi\mu_{2}x_{2})
\sin(\pi \mu_{1}y_{1})\sin(\pi\mu_{2}y_{2}),
\end{equation}
$n\in S$, $x=(x_{1},x_{2}),y=(y_{1},y_{2})\in\Qc$. The $f_{n}$ are not stationary, even though, around generic points in $\Qc$, far away from the boundary,
$f_{n}$ asymptotically tends to stationarity ~\cite{Cann}, understood in suitable regime, after suitable re-scaling.
The main interest of this manuscript is the expected nodal length of $f_{n}$, and its comparison to \eqref{eq:exp nodal len torus}, from this point on tacitly assuming $N_{n}\rightarrow\infty$.

\begin{theorem}
\label{thm:main thm asymp exp}
Let $$\Lc_{n}=\Lc(f_{n})=\len(f_{n}^{-1}(0))$$ be the total nodal length of $f_{n}$ on $\Qc$, and recall the notation
\eqref{eq:nun discrete prob} and \eqref{eq:lambdan def}.
There exists a subsequence of energy levels $S'\subseteq S$ satisfying the following properties:

\begin{enumerate}[a.]

\item The sequence $S'$ is of relative asymptotic density $1$ within $S$.

\item The set of accumulation points of the sequence of numbers $\{\widehat{\nu_{n}}(4)\}_{n\in S'}$ is $[-1,1]$.

\item Along $n\in S'$ we have $N_{n}\rightarrow\infty$, and
\begin{equation}
\label{eq:exp main asympt}
\E[\Lc_{n}] = \frac{\sqrt{\lambda_{n}}}{2\sqrt{2}}\cdot \left( 1-   \frac{1+4\widehat{\nu_{n}}(4)}{16}\cdot \frac{1}{N_{n}} +
o_{N_{n}\rightarrow\infty}\left( \frac{1}{N_{n}} \right)    \right).
\end{equation}

\end{enumerate}

\end{theorem}

The asymptotics \eqref{eq:exp main asympt} is expressed in terms of $\lambda_{n}$ rather than in terms of $n$ in a way that the leading
term on the r.h.s. of \eqref{eq:exp main asympt} agrees with \eqref{eq:exp nodal len torus} explicitly, for there is a discrepancy factor of $2$ otherwise, due to the discrepancy between \eqref{eq:lambda'n def} and \eqref{eq:lambdan def}. The boundary effect is then
encapsulated within the second, {\em correction}, term
\begin{equation}
\label{eq:2nd correction term}
\Cc_{n}:= - \frac{1+4\widehat{\nu_{n}}(4)}{16}\cdot \frac{1}{N_{n}} .
\end{equation}
On one hand the asymptotics \eqref{eq:exp main asympt} shows that, since, outside a thin set of $n\in S$, we have the convergence \eqref{eq:nun=>dt/2pi}
of $\nu_{n}$ to the uniform measure on $\Sc^{1}$, for such a sequence of $n$ the correction term is asymptotic to
\begin{equation*}
\Cc_{n}\sim -\frac{1}{16 N_{n}},
\end{equation*}
confirming Berry's ansatz on the nodal deficiency.
On the other hand, bearing in mind that, by property (b) of the sequence in Theorem \ref{thm:main thm asymp exp},
the accumulation set of $\{\widehat{\nu_{n}}(4)\}_{n\in S'}$ is the whole of $[-1,1]$,
without exclusions from $S'$, $$\Cc_{n}\cdot N_{n}$$ fluctuates infinitely in the asymmetric interval
\begin{equation*}
\Cc_{n}\cdot N_{n} \in \left[ -\frac{5}{16}, \frac{3}{16}  \right].
\end{equation*}

The maximal nodal deficiency (resp. maximal nodal surplus) in \eqref{eq:2nd correction term} is uniquely attained by the Cilleruelo measure
$$\frac{1}{4} \left( \delta_{\pm 1}+ \delta_{\pm i}\right)$$
(resp. its tilt by $\pi/4$), consistent with our interpretation of the amplified horizontal and vertical wave propagation
for Cilleruelo sequences (resp. their $\pi/4$-tilt for the tilted Cilleruelo), in light of Berry's rationale of nodal deficiency occurring
as a result of nodal lines perpendicular to the boundary. Finally, we notice that,
judging by the analogous quantity for the Arithmetic Random Waves \eqref{eq:ARW fluct sqrt(n)/N}, and applying M. Berry's reasoning, explained in \S\ref{sec:Nod length RWM}, we expect the fluctuations of $\Lc_{n}$ to be of the same order of magnitude $\approx \frac{\sqrt{n}}{N_{n}}$
as $\Cc_{n}$ (a by-product of the aforementioned ``miraculous" cancellation). That means that, unlike the situation in \cite{Berry 2002}, one cannot detect the nodal surplus or deficiency judging by the total nodal length based on one sample only.
However, this could be mended by taking more samples, or, likely, by restricting the sample to the vicinity of the boundary.

\vspace{2mm}

The main conclusion \eqref{eq:exp main asympt} of Theorem \ref{thm:main thm asymp exp} is valid for ``generic" $n\in S'\subseteq S$ only,
rather than for the whole sequence $n\in S$, though, importantly, this generic family is sufficiently rich so that to exhibit
a variety of different asymptotic biases of the correction term \eqref{eq:2nd correction term}.
Below Theorem \ref{thm:main thr expl semi-corr} will be stated, a version
of Theorem \ref{thm:main thm asymp exp} with an explicit control over the error term in
\eqref{eq:exp main asympt}, valid for the whole sequence $n\in S$ of energy level, expressed in
terms of the so-called ``spectral semi-correlations", defined in \S\ref{sec:semi-corr} (see Definition \ref{def:semi-corr}).
Our failure to unrestrict the statement of Theorem \ref{thm:main thm asymp exp} for the whole sequence $n\in S$ is then a
by-product of Theorem \ref{thm:semi-correlations Lyosha} below asserting a bound for the semi-correlations for a generic sequence
of energy levels. {\bf We do believe that \eqref{eq:exp main asympt} holds for $n\in S$, with no further restriction.}

\subsection{Spectral semi-correlations}
\label{sec:semi-corr}

Let
\begin{equation}
\label{eq:l=2k}
l=2k
\end{equation}
be an even number, with $k\ge 1$ an integer. The length-$l$ spectral correlation set ~\cite{KKW} is the set
\begin{equation}
\label{eq:prec corr def}
\Rc_{l}(n) = \left\{(\mu^{1},\ldots,\mu^{l})\in\Ec_{n}^{l}:\: \sum\limits_{j=1}^{l}\mu^{j}=0\right\}
\end{equation}
of all $l$-tuples of lattice points in $\Ec_{n}$ whose sum vanishes; by an elementary congruence obstruction modulo $2$,
for $l$ odd the corresponding correlation sets are all empty. The size of $\Rc_{l}(n)$ is directly related to the $l$-th moment
of the covariance function \eqref{eq:pn covar ARW def} corresponding to the Arithmetic Random Waves:
\begin{equation*}
\int\limits_{\Tb^{2}}p_{n}(x)^{l}dx = \frac{1}{N_{n}^{l}}|\Rc_{l}(n)|,
\end{equation*}
and bounding $|\Rc_{6}(n)|$ was a key ingredient for bounding the remainder while proving \eqref{eq:var asymp ARW KKW} in ~\cite{KKW}.
Since, for $k\ge 2$, fixing $\mu^{1},\ldots, \ldots,  \mu^{l-2}$ so that $$\sum\limits_{j=1}^{l-2}\mu^{j}\ne 0,$$
the relation $$\sum\limits_{j=1}^{l}\mu^{j}= 0$$
determines the remaining two lattice points $\mu^{l-1}$ and $\mu^{l}$ up to permutation, it is readily seen that for every $l\ge 4$,
$$|\Rc_{l}(n)|=O\left(N_{n}^{l-2}\right).$$

Let
\begin{equation*}
\Dc_{l}(n) = \left\{\pi (\mu^{1},-\mu^{1},\ldots, \mu^{k},-\mu^{k}) : \: \pi\in S_{l}\right\}
\end{equation*}
be the set of all $l$-tuples cancelling out in pairs, where $S_{l}$ is the symmetric group permuting the $l$-tuples,
of size
\begin{equation*}
|\Dc_{l}(n)| \sim c_{k}\cdot N_{n}^{k}.
\end{equation*}
Evidently, for every $l$ and $n\in S$, we have the inclusion
$$\Dc_{l}(n)\subseteq \Rc_{l}(n).$$ Hence, in particular
\begin{equation}
\label{eq:Corr>>N^k}
|\Rc_{l}(n)| \gg N_{n}^{k},
\end{equation}
recalling \eqref{eq:l=2k}. Bombieri and Bourgain ~\cite{BB}
showed that
\begin{equation}
\label{eq:BB length 6 uncond full}
|\Rc_{6}(n)| \ll N_{n}^{7/2},
\end{equation}
and established the striking inequality
\begin{equation*}
|\Rc_{6}(n)\setminus\Dc_{6}(n)| \ll N_{n}^{l/2-\gamma},
\end{equation*}
for some $\gamma>0$, valid for density-$1$ sequence of $n\in S$, or, alternatively, conditionally for the full sequence $S$,
so that, in particular, for these $n$, the optimal inequality
\begin{equation}
\label{eq:Corr<<N^k}
|\Rc_{l}(n)| \ll N_{n}^{k}
\end{equation}
holds (recall \eqref{eq:l=2k}).

\vspace{2mm}

In ~\cite{BMW}, the notion of spectral {\em quasi-correlations}, was instrumental for studying the analogue of $\Zc_{n}$ for the Arithmetic Random Waves \eqref{eq:fn eigfunc tor ARW}, restricted to domains decreasing with $n$ above Planck scale, e.g discs with radius $n^{-1/2+\delta}$
(``Shrinking balls").
For $l$ as above and $\epsilon>0$, a length-$l$ quasi correlation
is an $l$-tuple $(\mu^{1},\ldots,\mu^{l})$ of points in $\Ec_{n}$ so that $$0<\left\|\sum\limits_{j=1}^{l}\mu^{j}\right\| < n^{1/2-\epsilon}.$$
It was shown ~\cite[Theorem 1.4]{BMW} that for $n$ generic and $l$ arbitrary even number, the quasi-correlation set is {\em empty}.

\vspace{2mm}

In this manuscript we introduce a new concept, of {\em semi-correlations}, instrumental within the proof of
Theorem \ref{thm:main thm asymp exp}, as it will allow us to control the problematic {\em singular} set in $\Qc$, see
Corollary \ref{cor:sing contr << semi-corr}, and, we believe, of independent interest on its own right.

\begin{definition}[Semi-correlations]
\label{def:semi-corr}

For $l=2k$, $n\in S$, the length-$l$ semi-correlation set is the collection
\begin{equation}
\label{eq:M semi-corr def}
\Mcc_{l}(n) = \left\{(\mu^{1},\ldots \mu^{l})\in\Ec_{n}^{l}: \: \sum\limits_{j=1}^{l} \mu^{j}_{1} = 0 \right\}
\end{equation}
of all $l$-tuples of lattice points in $\Ec_{n}$ with the first coordinate summing up to $0$.
\end{definition}

From the above definition, it is evident that $$\Dc_{l}(n)\subseteq \Rc_{l}(n)\subseteq \Mcc_{l}(n),$$
so that, in particular, $$|\Mcc_{l}(n)|\gg N_{n}^{k},$$ cf. \eqref{eq:Corr>>N^k}. Quite remarkably, the following {\em optimal}
upper bound holds for the semi-correlation set size, albeit for a generic sequence only.

\begin{theorem}[Bound for the number of semi-correlations.]
\label{thm:semi-correlations Lyosha}
For every $l=2k\ge 4$ even integer, there exists a sequence $S'=S'(l)\subseteq S$ satisfying the following properties.

\begin{enumerate}[a.]

\item The sequence $S'\subseteq S$ is of relative asymptotic density $1$.

\item The set of accumulation points of the sequence of numbers $\{\widehat{\nu_{n}}(4)\}_{n\in S'}$ is
the whole of $[-1,1]$.

\item Along $n\in S'$ we have $N_{n}\rightarrow\infty$ and
\begin{equation}
\label{eq:semi-corr<N^k}
|\Mcc_{l}(n)| = O\left(N_{n}^{k}\right).
\end{equation}

\end{enumerate}

\end{theorem}

By using a standard diagonal argument it is possible to choose a density $1$ sequence $S'\subseteq S$, satisfying \eqref{eq:semi-corr<N^k}
for {\em all} $l\ge 4$ even (with constant involved in the `O'-notation depending on $l$).
Theorem \ref{thm:semi-correlations Lyosha} is stronger compared to the upper bound \eqref{eq:Corr<<N^k}
for the spectral correlations,
due to Bombieri-Bourgain, also valid for a density one sequence of $n\in S$. In addition to claiming the
upper bound for the semi-correlations rather, significantly weaker than correlation, at also asserts the richness of the postulated
sequence in terms of the angular distribution of $\Ec_{n}$, expressed in terms of the Fourier coefficients $\{\widehat{\nu_{n}}(4)\}_{n\in S'}$.
The following result is
a version of Theorem \ref{thm:main thm asymp exp}, with an explicit control over the error term in \eqref{eq:exp main asympt}, expressed in terms of
the spectral semi-correlations. After a significant amount of effort put into, {\bf we still do not know whether \eqref{eq:semi-corr<N^k} holds for all $l$ even, along $n\in S$ with no further restriction}, and believe this question to be of sufficiently high interest, both for applications of the type of Theorem \ref{thm:main thm asymp exp}, and intrinsic, to be addressed in the future.

\begin{theorem}[Explicit unrestricted version of Theorem \ref{thm:main thm asymp exp}]
\label{thm:main thr expl semi-corr}
For every $l\ge 4$ even we have
\begin{equation}
\label{eq:exp len lead semi-corr}
\E[\Lc_{n}] = \frac{\sqrt{\lambda_{n}}}{2\sqrt{2}}\cdot \left( 1-   \frac{1+4\widehat{\nu_{n}}(4)}{16}\cdot \frac{1}{N_{n}} +
O_{N_{n}\rightarrow\infty}\left( \frac{1}{N_{n}^{2}}+N_{n}^{2-l}| \Mcc_{l}(n)| \right)    \right).
\end{equation}

\end{theorem}

Theorem \ref{thm:main thr expl semi-corr} implies Theorem \ref{thm:main thm asymp exp} at once, by working with the sequence
resulting from an application of Theorem \ref{thm:semi-correlations Lyosha} on $l=8$ (say), and from this point on we will only care to prove
Theorem \ref{thm:main thr expl semi-corr} (and Theorem \ref{thm:semi-correlations Lyosha}).

\subsection*{Acknowledgements}

The research leading to these results has received funding from the European Research Council under the European Union's Seventh Framework Programme (FP7/2007-2013), ERC grant agreement n$^{\text{o}}$ 335141 (V.C and I.W.).
We are grateful to Sven Gnutzmann for raising the question of nodal deficiency in arithmetic context as ours, and to Andrew Granville,
P\"{a}r Kurlberg, Domenico Marinucci, Ze\'{e}v Rudnick and Mikhail Sodin for the very inspiring and fruitful conversations concerning subjects relevant to this manuscript.

\section{Outline of the paper}

\subsection{Outline of the proof of Theorem \ref{thm:main thr expl semi-corr}}

\subsubsection{The Kac-Rice formula}

The Kac-Rice formula is a meta-theorem allowing one to evaluate the $(d-1)$-volume of the zero set of a random field $F:\R^{d}\rightarrow\R$,
for $F$ satisfying some smoothness and non-degeneracy conditions.
For $F:\R^{d}\rightarrow\R$, a sufficiently smooth centred Gaussian random field, we define
\begin{equation*}
K_{1}(x):= \frac{1}{(2\pi)^{1/2}\sqrt{\var(F(x))}}\cdot \E[|\nabla F(x)|\big| F(x)=0]
\end{equation*}
the zero density (first intensity) of $F$. Then the Kac-Rice formula asserts that for some suitable class of random fields $F$
and $\overline{\Dpc}\subseteq \R^{d}$ a compact closed subdomain of $\R^{d}$, one has the equality
\begin{equation}
\label{eq:Kac Rice meta}
\E[\vol_{d-1}(F^{-1}(0)\cap \overline{\Dpc})] = \int\limits_{\overline{\Dpc}}K_{1}(x)dx.
\end{equation}

We would like to apply \eqref{eq:Kac Rice meta} on the random fields $f_{n}$ in \eqref{eq:BAARW fn def} to evaluate the expectation on
the l.h.s. of \eqref{eq:exp main asympt}. Unfortunately, for some $n$, the aforementioned
non-degeneracy conditions fail decisively for some points of $\Qc$. For these cases, an {\em approximate} version of Kac-Rice was
developed ~\cite[Proposition 1.3]{RW2014} (in a slightly different context of evaluating the variance),
so that rather than holding precisely, \eqref{eq:Kac Rice meta} would hold {\em approximately}, still yielding
the asymptotic law for the evaluated expectation. Nevertheless, for this particular case, by excising some neighbourhoods of
the problematic {\em degenerate} set, consisting of a union of a {\em grid} and finitely many {\em isolated} points,
and by applying the Monotone Convergence Theorem, we will be able to deduce that
\eqref{eq:Kac Rice meta} holds {\em precisely}, save for the length of the said {\em deterministic} {\em grid}
contained in the nodal set of $f_{n}$ for some $n\in S$ of a particular form.

Let $n\in S$ be of the form \eqref{eq:n in S factor primes}, and denote the associated number
\begin{equation}
\label{eq:Qn grid numb def}
Q_{n} = 2^{\lfloor a/2\rfloor}\prod\limits_{k=1}^{r}q_{k}^{h_{k}},
\end{equation}
so that, in particular, $Q_{n}^{2} | n$. We will establish later (cf. Lemma \ref{lem:degenerate set grid+fin} below) that in such a scenario,
all of the $\mu_{1}$ and $\mu_{2}$ on the r.h.s. of \eqref{eq:BAARW fn def} are divisible by $Q_{n}$, so that if $Q_{n}>1$, then
necessarily the (deterministic) grid
\begin{equation}
\label{eq:grid G(Qn)}
\Gc_{n} = \Gc(Q_{n}) = \bigcup\limits_{k=1}^{Q-1}\left\{(x_{1},x_{2})\in\Qc:\: x_{1}= \frac{k}{Q_{n}} \right\}
\cup \bigcup\limits_{k=1}^{Q-1}\left\{(x_{1},x_{2})\in\Qc:\: x_{2}= \frac{k}{Q_{n}} \right\}
\subseteq \Qc
\end{equation}
is a.s. contained inside the nodal line $f_{n}^{-1}(0)$ (see figures \ref{fig:singular set}-\ref{fig:singular set2});
$\Gc_{n}$ is of length
\begin{equation}
\label{eq:len(grid)=2(Q-1)}
\len(\Gc_{n})=2\cdot (Q_{n}-1).
\end{equation}
Lemma \ref{lem:degenerate set grid+fin} will also assert that such a situation is only possible
in this scenario, i.e. all the components $\{\mu_{1}\}_{\mu\in\Ec_{n}/\sim}$ are divisible by a maximal number $d>1$, if and only
$d=Q_{n}$ in \eqref{eq:Qn grid numb def}.
The following proposition is the announced Kac-Rice formula, with the said caveat (namely, the length of $\Gc_{n}$,
manifested on the r.h.s. of \eqref{eq:BAARW Kac Rice}).

\begin{figure}[ht]
\centering
\includegraphics[height=50mm]{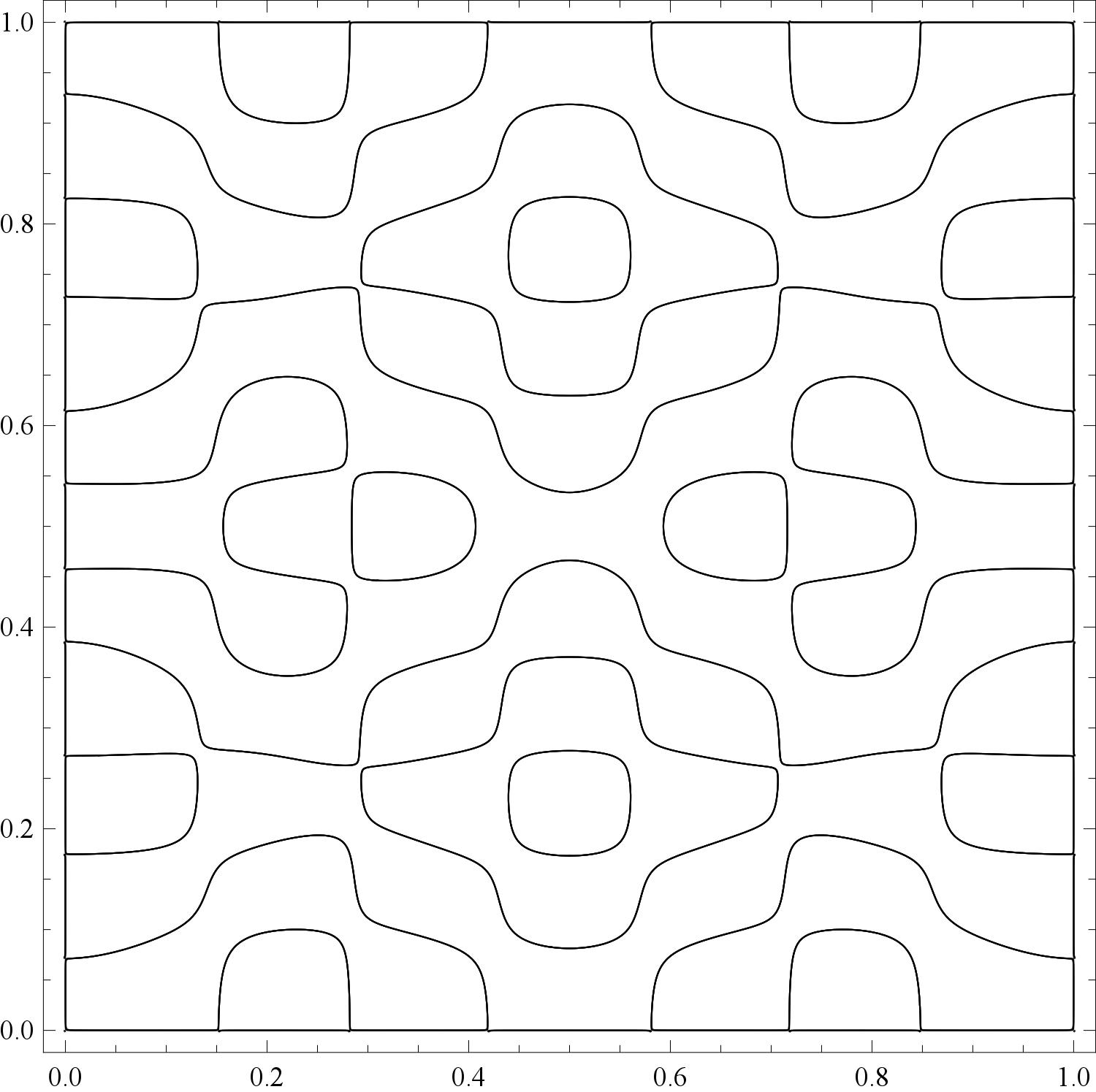}
\includegraphics[height=50mm]{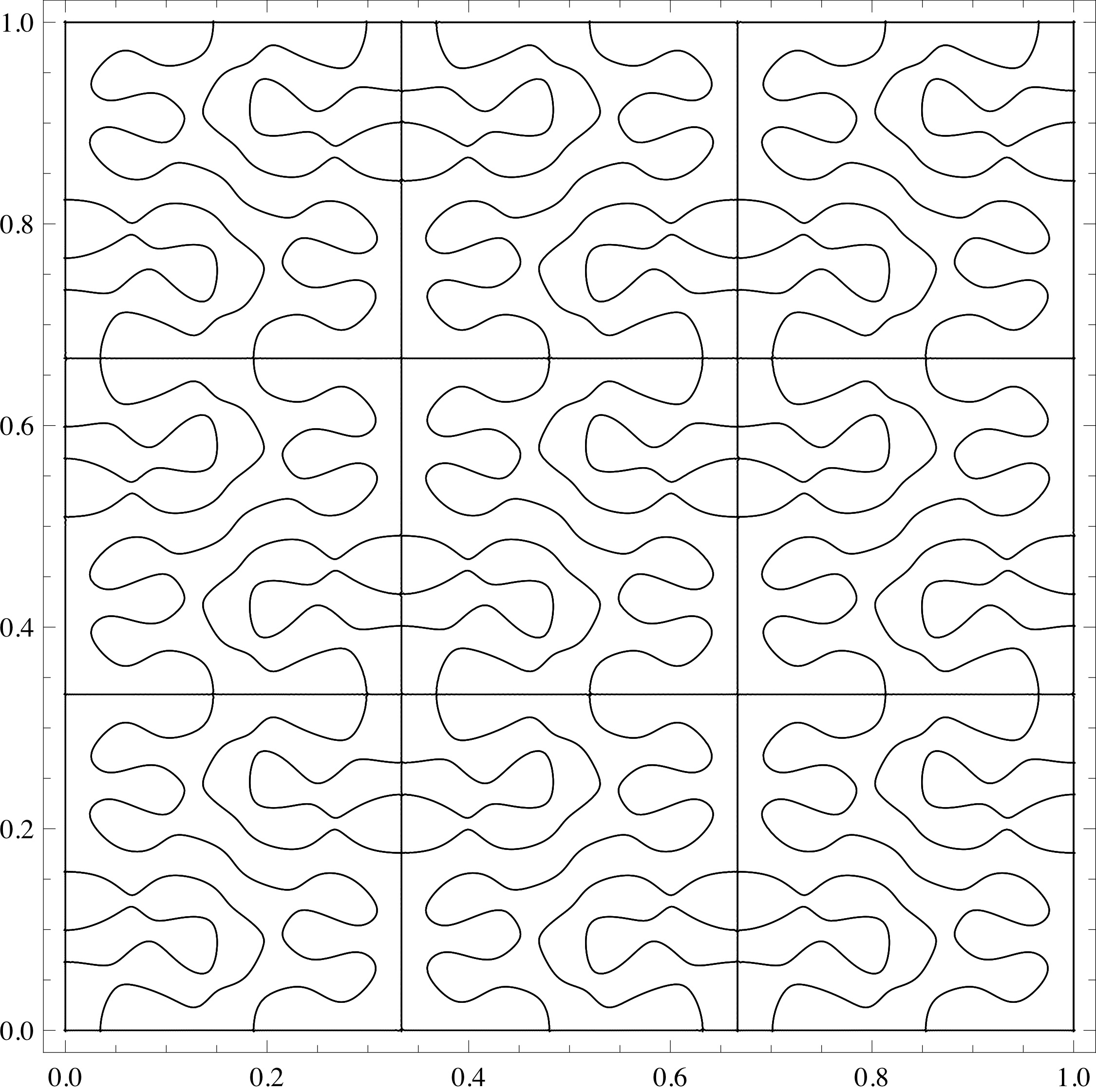}
\includegraphics[height=50mm]{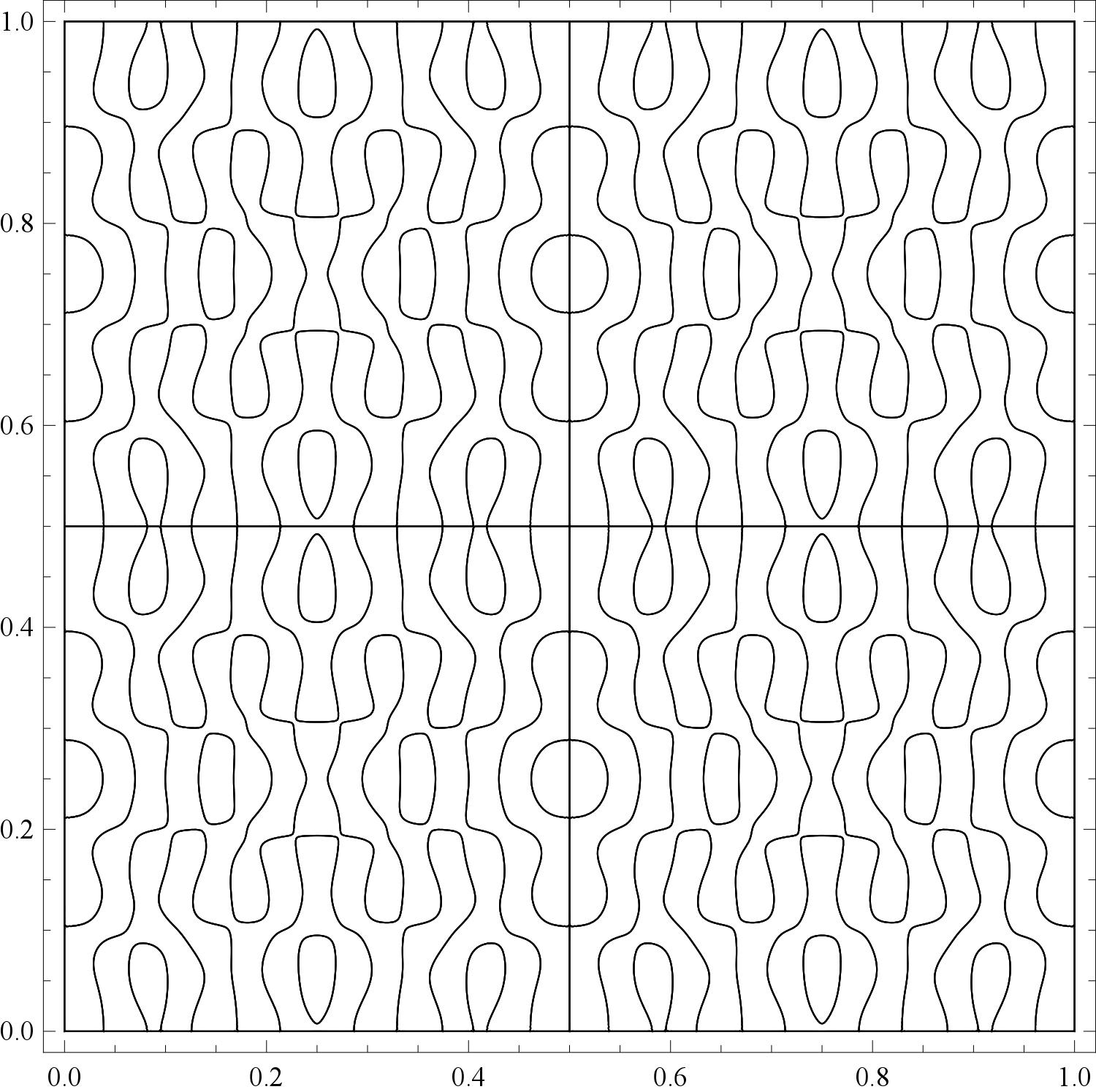}
\caption{Left: Nodal line for some $f_{n}$, $n=170$, not containing a grid ($Q_{n}=1$).
Middle: Same for some $f_{n}$, $n=765$. Here $Q_{n}=3$,
and thereupon its nodal line contains the grid $\Gc(1/3)$.
Right: Same for $n=1000$, $Q_{n}=2$.}
\label{fig:singular set}
\end{figure}

\begin{figure}[ht]
\centering
\includegraphics[height=75mm]{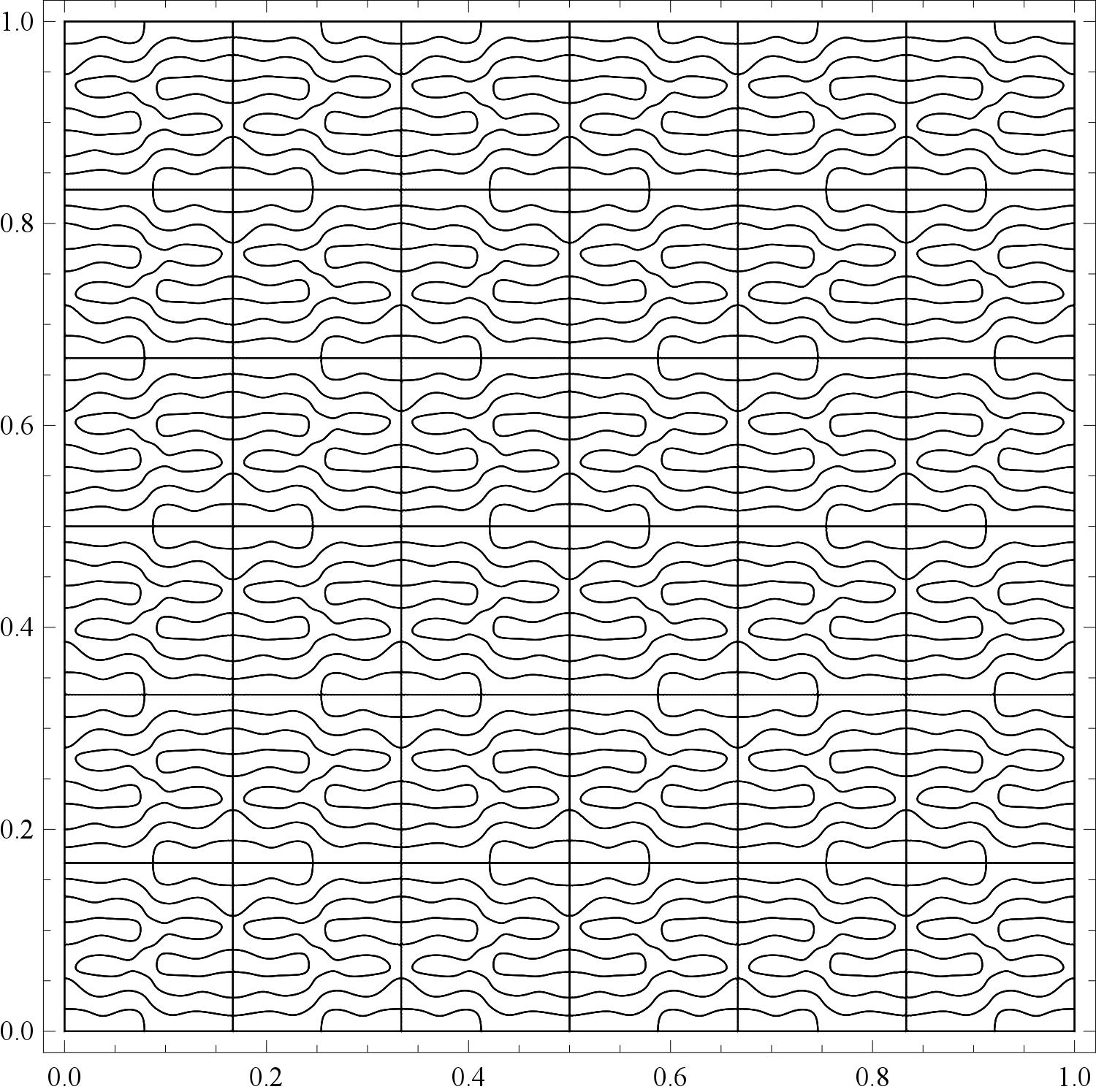}
\caption{Nodal line for some $f_{n}$, $n=3060$. Here $Q_{n}=6$,
and the nodal line of $f_{n}$ contains $\Gc(1/6)$.}
\label{fig:singular set2}
\end{figure}

\begin{proposition}
\label{prop:Kac Rice BAARW}
Let $f_{n}$ be as in \eqref{eq:BAARW fn def}, $\Lc_{n}$ the nodal length of $f_{n}$, and
\begin{equation}
\label{eq:K1 fn density def}
K_{1}(x)=K_{1;n}(x) = \frac{1}{(2\pi)^{1/2}\sqrt{\var(f_{n}(x))}} \cdot \E[|\nabla f_{n}(x)|\big| f_{n}(x)=0]
\end{equation}
be the zero density of $f_{n}$. Then, for every $n\in S$, we have $K_{1}\in L^{1}(\Qc)$, and moreover,
we have
\begin{equation}
\label{eq:BAARW Kac Rice}
\E[\Lc_{n}] = \int\limits_{\Qc}K_{1}(x)dx + 2(Q_{n}-1),
\end{equation}
where $Q_{n}$ is as in \eqref{eq:Qn grid numb def}.
\end{proposition}

For example, by comparing \eqref{eq:BAARW Kac Rice} to \eqref{eq:Kac Rice meta}, we may deduce as a particular by-product of Proposition \ref{prop:Kac Rice BAARW}, that \eqref{eq:Kac Rice meta} holds precisely in our case, if and only if $Q_{n}=1$, i.e. the grid is empty.
Below it will be demonstrated that $Q_{n}$ on the r.h.s. of \eqref{eq:BAARW Kac Rice} does not contribute to the Kac-Rice integral
(nor to the correction term $\sqrt{n}\Cc_{n}$ in \eqref{eq:exp main asympt}, with $\Cc_{n}$ given by \eqref{eq:2nd correction term}),
e.g., it is easily dominated by $\frac{\sqrt{n}}{N_{n}^{A}}$, for every $A>0$ (see the proof of Theorem \ref{thm:main thr expl semi-corr}
in \S\ref{sec:thm main semi-corr proof} below).

\subsubsection{The joint distribution of $(f_{n}(x),\nabla f_{n}(x))$}

By the definition \eqref{eq:K1 fn density def} of the zero density function of $f_{n}$, to investigate $K_{1}(x)$
we naturally encounter the value distribution of both $f_{n}(x)$, determined by
$\var(f_{n}(x))$, and $\nabla f_{n}(x)$ conditioned on $f_{n}(x)=0$, determined by its $2\times 2$ (conditional) covariance matrix.
On recalling that the covariance function $r_{n}$ of $f_{n}$ is given by \eqref{eq:rn covar def}, so that on the diagonal
$$v_{n}(x):=r_{n}(x,x) = \var(f_{n}(x)),$$ with an elementary manipulation and well-known trigonometric identities yielding
\begin{equation}
\label{eq:vn var ident}
v_{n}(x)=\var(f_{n}(x)) = 1- s_{n}(x),
\end{equation}
where
\begin{equation}
\label{eq:sn var pert ident}
s_{n}(x) =\frac{4}{N_{n}}\sum\limits_{\mu\in \Ec_{n}/\sim}\left( \cos(2\mu_{1}\pi x_{1}) +
\cos(2\mu_{2}\pi x_{2})-  \cos(2\mu_{1}\pi x_{1}) \cdot
\cos(2\mu_{2}\pi x_{2}) \right).
\end{equation}
Further, we need to evaluate the covariance matrix of $\nabla f_{n}(x)$ conditioned on $f_{n}(x)=0$. A scrupulous direct computation, carried out in
Appendix \ref{sec:covariance matrix} shows that the corresponding (normalised) covariance matrix is given by the following:

\begin{lemma} \label{lem:Omega cond covar def}
The $2\times 2$ covariance matrix of $\nabla f_{n}(x)$, conditioned on $f_{n}(x)=0$, and appropriately normalised,
is the following $2\times 2$ real symmetric matrix:
\begin{equation}
\label{eq:Omega cond covar def}
\Omega_{n}(x):=\frac{2}{\pi^{2}n}\cdot \E\left[ \nabla f_{n}(x) \cdot \nabla^{t} f_{n}(x) \big| f_{n}(x)=0\right]
=: I_{2} +\Gamma_{n}(x),
\end{equation}
where
\begin{equation}
\label{eq:Gamma comp}
\Gamma_{n}(x) = \frac{8}{n N_{n}}\left(\begin{matrix}b_{11;n}(x) &b_{12;n}(x) \\ b_{21;n}(x) &b_{22;n}(x)\end{matrix} \right)
-\frac{128}{nN_{n}^{2}v_{n}(x)} \left(\begin{matrix}d_{1;n}(x)^{2} &d_{1;n}(x)\cdot d_{2;n}(x) \\
d_{1;n}(x)\cdot d_{2;n}(x) &d_{2;n}(x)^{2}\end{matrix} \right),
\end{equation}
where $v_{n}$ is given by \eqref{eq:vn var ident} and \eqref{eq:sn var pert ident};
\begin{equation}
\label{eq:b11 def}
b_{11;n}(x) = \sum\limits_{\mu\in \Ec_{n}/\sim}\mu_{1}^{2} \left(\cos(2\mu_{1}\pi x_{1})-\cos(2\mu_{2}\pi x_{2})-
\cos(2\mu_{1}\pi x_{1})\cdot\cos(2\mu_{2}\pi x_{2}) \right),
\end{equation}
\begin{equation*}
b_{12;n}(x)=b_{21;n}(x) = \sum\limits_{\mu\in \Ec_{n}/\sim}\mu_{1}\mu_{2} \sin(2\mu_{1}\pi x_{1})\cdot\sin(2\mu_{2}\pi x_{2}),
\end{equation*}
\begin{equation*}
b_{22;n}(x) = \sum\limits_{\mu\in \Ec_{n}/\sim}\mu_{2}^{2} \left(\cos(2\mu_{2}\pi x_{2})-\cos(2\mu_{1}\pi x_{1})-
\cos(2\mu_{1}\pi x_{1})\cdot\cos(2\mu_{2}\pi x_{2}) \right),
\end{equation*}
\begin{equation*}
d_{1;n}(x) = \sum\limits_{\mu\in \Ec_{n}/\sim} \mu_{1}\sin(2\mu_{1}\pi x_{1})\cdot \sin(\mu_{2}\pi x_{2})^{2}
\end{equation*}
and
\begin{equation}
\label{eq:d2 def}
d_{2;n}(x) = \sum\limits_{\mu\in \Ec_{n}/\sim} \mu_{2}  \sin(2\mu_{2}\pi x_{2}) \cdot \sin(\mu_{1}\pi x_{1})^{2}.
\end{equation}

\end{lemma}

\subsubsection{Singular set}

Next, we aim at analysing the asymptotic behaviour of the r.h.s. of \eqref{eq:BAARW Kac Rice}. Towards this goal
we will separate the domain $\Qc$ of the integration on the r.h.s. of \eqref{eq:BAARW Kac Rice} into
the ``good" or nonsingular set, where $K_{1}$ is ``tame", i.e. admits precise asymptotic (Proposition \ref{prop:K1 asymp nonsing}),
and the ``bad" or singular sets, which itself consists of small singular squares so that to be able to control the integral of $K_{1}$.
First, we will bound the total contribution of the singular set (Corollary \ref{cor:sing contr << semi-corr}) from above, by
separately bounding
the number of small singular squares (Proposition \ref{prop:sing meas<<semi-corr}), appealing to the bound
for spectral semi-correlations in Theorem \ref{thm:semi-correlations Lyosha}, and the contribution of a singular small square (Proposition \ref{prop:sing sqr bnd}).

Below it will asserted that for {\em most} of the points $x\in \Qc$, both the value of $v_{n}(x)$
is close to unit (equivalently, $s_{n}(x)$ is {\em small}), and $\Omega_{n}$ in \eqref{eq:Omega cond covar def} is close
to the unit matrix (equivalently, $\Gamma_{n}$ is small);
we will designate the {\em other}
points as ``singular", and excise them while performing an asymptotic analysis on $K_{1}$.
To quantify it, we take $\epsilon_{0}>0$ and $c_{0}>0$, and
keep them fixed but sufficiently small throughout. We will endow the singular set with
a structure of a union of small squares (cf. ~\cite{RW08,RW2014}).

\begin{definition}[Singular set]
\label{def:sing set}

Let $\epsilon_{0}>0$ and $c_{0}>0$ be two parameters.
Take
$$\delta_{0}=\delta_{0}(n) = \frac{c_{0}}{\sqrt{n}},$$
and $K:=\left\lfloor\frac{1}{\delta_{0}}\right\rfloor + 1$.
For $1\le i\le K$ define the interval $$I_{i} = \left[(i-1)\cdot \frac{1}{K},i\cdot \frac{1}{K}\right]\subseteq [0,1],$$
and for $1\le i,j \le K$ denote the small square
\begin{equation}
\label{eq:Qij small square}
Q_{ij} := I_{i}\times I_{j}\subseteq \Qc.
\end{equation}
We have the partition $$\Qc=\bigcup\limits_{1\le i,j\le K} Q_{ij}$$ of the square
into a union of small squares, disjoint save for boundary overlaps.

\begin{enumerate}[a.]

\item Recall the notation in \eqref{eq:sn var pert ident} and \eqref{eq:Omega cond covar def}.
A small square $Q_{ij}$ in \eqref{eq:Qij small square} is ``singular" if it contains a point $x_{0}\in Q_{ij}$
satisfying either of the three inequalities:
$$|s_{n}(x_{0})|> \epsilon_{0},$$ or $$|\tr(\Gamma_{n}(x_{0}))|>\epsilon_{0},$$
or $$|\det(\Gamma_{n}(x_{0}))|>\epsilon_{0}.$$

\item The singular set is the union
\begin{equation*}
\Qc_{s}:= \bigcup\limits_{Q_{ij} \text{ singular}} Q_{ij}.
\end{equation*}
of all small singular squares.

\item The complement $\Qc\setminus \Qc_{s}$ of the singular set is called ``nonsingular set".

\end{enumerate}

\end{definition}

The following couple of propositions assert a bound for the total measure of the singular set,
and for the contribution of a single singular square respectively. Combining these two
will yield an upper bound for the total contribution of the singular set $\Qc_{s}$
to the integral on the r.h.s. of \eqref{eq:BAARW Kac Rice}.

\begin{proposition}[Bound for the measure of the singular set]
\label{prop:sing meas<<semi-corr}
For every $l\ge 4$ even integer we have the following bound for the measure of the singular set
in terms of the length-$l$ spectral correlation set \eqref{eq:M semi-corr def}:
\begin{equation}
\label{eq:sing meas<<semi-corr}
\meas(\Qc_{s}) = O\left(N_{n}^{-l}  |\Mcc_{l}(n)|  \right),
\end{equation}
where the constant involved in the $`O'$-notation depends only on $l$.

\end{proposition}

\begin{proposition}[Bound for a single small square]
\label{prop:sing sqr bnd}

Let $Q\subseteq \Qc$ be an arbitrary square of side length $\frac{c_{0}}{\sqrt{n}}$ with $c_{0}>0$ sufficiently
small. Then
\begin{equation}
\label{eq:int(K1) small cube << N^2/sqrt(n)}
\int\limits_{Q}K_{1}(x)dx = O\left(\frac{N_{n}^{2}}{\sqrt{n}}\right),
\end{equation}
with the constant involved in the `O'-notation in \eqref{eq:int(K1) small cube << N^2/sqrt(n)} absolute.

\end{proposition}

It is worthy of a mention that the doubling exponent method due to Donnelly-Fefferman ~\cite{DF}, with relation to
Yau's conjecture \eqref{eq:Yau conj}, yields the deterministic bound of $O\left(1\right)$ for the nodal length of $f_{n}$
restricted to $Q$ as in Proposition \ref{prop:sing sqr bnd},
that, being better than the global bound of $\sqrt{n}$, falls short from being sufficient for our
needs\footnote{Any bound of the form $O\left(\frac{N_{n}^{A}}{\sqrt{n}} \right)$, $A>0$, would be sufficient
for our purposes, by choosing $l$ sufficiently high (see how Theorem \ref{thm:main thm asymp exp} is inferred from Theorem
\ref{thm:main thr expl semi-corr},
as explained immediately after the latter theorem).}, by a significant margin.
We believe that the optimal upper bound on the r.h.s. of \eqref{eq:int(K1) small cube << N^2/sqrt(n)}
should be $O\left( \frac{1}{\sqrt{n}} \right)$, however, after some effort, we were not able to prove that.
Instead, we sacrifice a power of $N_{n}$ by virtually not exploiting the summation in \eqref{eq:rn covar def}
(except the invariance of $\Ec_{n}$ w.r.t. $\mu=(\mu_{1},\mu_{2})\mapsto (\mu_{2},\mu_{1})$), in the hope to gain the
lost power of $N_{n}$ while bounding the number of singular squares (equivalently, the measure of $\Qc_{s}$),
which is precisely what is achieved in Proposition \ref{prop:sing meas<<semi-corr}, with the help of
Theorem \ref{thm:semi-correlations Lyosha}.
In particular, Proposition \ref{prop:sing sqr bnd} applies to all singular squares $Q_{ij}\subseteq \Qc_{s}$, leading to the
following, possibly sub-optimal, result.

\begin{corollary}
\label{cor:sing contr << semi-corr}
For every $l\ge 4$ even integer we have the following bound for the contribution of the singular set to
the integral on the r.h.s. of \eqref{eq:BAARW Kac Rice}:
\begin{equation}
\label{eq:int sing set << N^2sqrt(n) semi-corr}
\int\limits_{\Qc_{s}}K_{1}(x)dx = O\left(N_{n}^{2-l}\sqrt{n}\cdot |\Mcc_{l}(n)|  \right).
\end{equation}

\end{corollary}

The upshot of Corollary \ref{cor:sing contr << semi-corr} is that, thanks to Theorem \ref{thm:semi-correlations Lyosha},
by choosing $l$ sufficiently big, we can make the r.h.s.
of \eqref{eq:int sing set << N^2sqrt(n) semi-corr} smaller than $\sqrt{n}\cdot N_{n}^{-A}$, with $A>0$ arbitrarily large.
That is crucial if we are to majorise it by the second term in the claimed asymptotic expansion \eqref{eq:exp main asympt},
of order of magnitude $\approx \frac{\sqrt{n}}{N_{n}}$.
The proof of Corollary \ref{cor:sing contr << semi-corr} is immediate given Propositions \ref{prop:sing meas<<semi-corr} and
\ref{prop:sing sqr bnd}, and, thereupon, conveniently omitted.

\subsubsection{Perturbative analysis on the non-singular set}

Outside the singular set, the precise analysis for the density function is feasible.

\begin{proposition}[Asymptotics for $K_{1}$ outside $\Qc_{s}$]
\label{prop:K1 asymp nonsing}
Let $\epsilon_{0}>0$ be a sufficiently small number, and recall that $s_{n}(\cdot)$ and $\Gamma_{n}(\cdot)$
are given by \eqref{eq:sn var pert ident} and \eqref{eq:Gamma comp} respectively. Then $K_{1}$ admits the following
asymptotics, uniformly for $x\in \Qc\setminus\Qc_{s}$,
\begin{equation*}
K_{1}(x) = \frac{\sqrt{\lambda_{n}}}{2\sqrt{2}} + L_{n}(x) + \Upsilon_{n}(x),
\end{equation*}
where the leading term is given by
\begin{equation}
\label{eq:Ln asympt K1 def}
L_{n}(x) = \frac{\sqrt{n}\pi}{4\sqrt{2}}\left(s_{n}(x)+\frac{\tr{\Gamma_{n}(x)}}{2} + \frac{3}{4}s_{n}(x)^{2}+
\frac{1}{4}s_{n}(x)\cdot\tr{\Gamma_{n}(x)}- \frac{\tr(\Gamma_{n}(x)^{2})}{16} - \frac{\left(\tr(\Gamma_{n}(x)\right)^{2})}{32} \right),
\end{equation}
and the error term is bounded by
\begin{equation}
\label{eq:Ups error bnd}
|\Upsilon_{n}(x)| = O\left( \sqrt{n}\cdot \left(|s_{n}(x)|^{3} + |\Gamma_{n}(x)|^{3}\right) \right),
\end{equation}
with constant involved in the `$O$'-notation absolute.

\end{proposition}

We observe that, by the definition \eqref{eq:Gamma comp} with \eqref{eq:b11 def}-\eqref{eq:d2 def},
the diagonal entries of $\Gamma_{n}(x)$ are bounded by an absolute constant (using the nonnegativity of the $d_{\cdot;n}(x)^{2}$),
and therefore, so are the diagonal
entries of $\Omega_{n}(x)$ in \eqref{eq:Omega cond covar def}, and, further, {\em all} the entries of $\Omega_{n}$
are bounded by an absolute constant, by the Cauchy-Schwarz inequality\footnote{This argument recovers the Gaussian Correlation Inequality
in this particular case, see e.g. ~\cite{Roy}.}. Taking also into account \eqref{eq:sn var pert ident},
and the definition \eqref{eq:Ln asympt K1 def} of $L(x)=L_{n}(x)$, we conclude that $L(x)$ is uniformly bounded
\begin{equation}
\label{eq:L(x) unif bnd}
|L(x)| = O(\sqrt{n}),
\end{equation}
with the involved constant absolute;
this will prove useful later, while restricting (or, rather, un-restricting) the range of the integration
in \eqref{eq:BAARW Kac Rice} to $\Qc\setminus \Qc_{s}$.

\subsubsection{Proof of Theorem \ref{thm:main thr expl semi-corr}}
\label{sec:thm main semi-corr proof}

The following lemma will be proved in \S\ref{sec:bnd eff err sec} below.

\begin{lemma}
\label{lem:L1 int, Ups bnd}

\begin{enumerate}[a.]

\item
\label{it:int L(x) val}
Let $L(x)=L_{n}(x)$ be as in \eqref{eq:Ln asympt K1 def}, and $l\ge 4$ an even number.
Then
\begin{equation*}
\int\limits_{\Qc}L(x)dx = - \frac{\pi (1+4 \hat{\nu}_n)}{32\sqrt{2} }\cdot \frac{\sqrt{n}}{N_n}
+O\left(\frac{\sqrt n }{N^2_n} \right)+ O\left(n^{1/2} N_{n}^{-l-1}  |\Mcc_{l}(n)|\right).
\end{equation*}

\item
\label{it:int Ups bnd}
Let $\Upsilon(x)=\Upsilon_{n}(x)$ be a function defined on $\Qc$, satisfying \eqref{eq:Ups error bnd}.
Then
\begin{equation*}
\int\limits_{\Qc}|\Upsilon(x)|dx = O\left( \frac{\sqrt{n}}{N_{n}^{2}} \right).
\end{equation*}

\end{enumerate}

\end{lemma}

Given the above results, the proof of Theorem \ref{thm:main thr expl semi-corr} is rather straightforward.

\begin{proof}[Proof of Theorem \ref{thm:main thr expl semi-corr} assuming Corollary \ref{cor:sing contr << semi-corr},
Proposition \ref{prop:K1 asymp nonsing} and Lemma \ref{lem:L1 int, Ups bnd}]

Let $l$ be given. We invoke Proposition \ref{prop:Kac Rice BAARW}, and separate the contribution of the nonsingular and the singular sets in
the integral on the r.h.s. of \eqref{eq:BAARW Kac Rice} to write, with the help of Corollary \ref{cor:sing contr << semi-corr},
\begin{equation}
\label{eq:exp len nonsing dens}
\E\left[\Lc_{n} \right] = \int\limits_{\Qc\setminus\Qc_{s}}K_{1}(x)dx+\int\limits_{\Qc_{s}}K_{1}(x)dx+ 2 (Q_n-1)
= \int\limits_{\Qc\setminus\Qc_{s}}K_{1}(x)dx + O(Q_{n}+N_{n}^{2-l}\sqrt{n}\cdot |\Mcc_{l}(n)|),
\end{equation}
with $Q_{n}$ given by \eqref{eq:Qn grid numb def}.
First, we claim that for every $A>0$,
\begin{equation}
\label{eq:Qn small}
Q_{n}=O\left( \frac{\sqrt{n}}{N_{n}^{A}}  \right),
\end{equation}
so that the contribution of the grid length is majorised by the error term on the r.h.s. of \eqref{eq:exp len lead semi-corr}.
To this end, we recall the prime decomposition \eqref{eq:n in S factor primes} of $n$, and write
\begin{equation*}
P_{n} = \prod\limits_{j=1}^{s}p_{j}^{e_{j}},
\end{equation*}
so that
\begin{equation*}
n=\begin{cases}
P_{n}Q_{n}^{2} \\ 2P_{n}Q_{n}^{2}
\end{cases}.
\end{equation*}
Since, as it is well-known, $$N_{n}=N_{P_{n}}=4\prod\limits_{j=1}^{s}(e_{j}+1),$$
and $N_{n}=O(n^{\epsilon})$ for every $\epsilon>0$, we may easily write
(taking $A:=1/(2\epsilon)$)
\begin{equation*}
\sqrt{n} \ge Q_{n} \cdot \sqrt{P_{n}} \gg Q_{n} \cdot N_{P_{n}}^{A} = Q_{n}\cdot N_{n}^{A},
\end{equation*}
which is \eqref{eq:Qn small}.

Next, we use the asymptotics of $K_{1}$ on the nonsingular set claimed in Proposition \ref{prop:K1 asymp nonsing} to write
\begin{equation}
\label{eq:int Kc nonsing asymp}
\begin{split}
\int\limits_{\Qc\setminus\Qc_{s}}K_{1}(x)dx &= \frac{\sqrt{\lambda_{n}}}{2\sqrt{2}} \cdot \meas(\Qc\setminus\Qc_{s})+\int\limits_{\Qc\setminus\Qc_{s}}L_{n}(x)dx + \int\limits_{\Qc\setminus\Qc_{s}}\Upsilon_{n}(x)dx
\\&= \frac{\sqrt{\lambda_{n}}}{2\sqrt{2}} + \int\limits_{\Qc}L_{n}(x)dx +
O\left(\int\limits_{\Qc\setminus\Qc_{s}}|\Upsilon_{n}(x)|dx + \sqrt{n} \cdot N_{n}^{-l}|\Mcc_{l}(n)|\right)
\end{split}
\end{equation}
with a bound \eqref{eq:Ups error bnd} for the error term,
thanks to the uniform bound \eqref{eq:L(x) unif bnd} on $L_{n}(x)$ together with Proposition \ref{prop:sing meas<<semi-corr}.

Upon substituting \eqref{eq:int Kc nonsing asymp} into \eqref{eq:exp len nonsing dens}, and exploiting \eqref{eq:Qn small},
we may deduce (with the error term of $\sqrt{n} \cdot N_{n}^{-l}|\Mcc_{l}(n)|$ being majorized by $N_{n}^{2-l}\sqrt{n} \cdot |\Mcc_{l}(n)|$):
\begin{equation}
\label{eq:exp nod len int L(x) err}
\E\left[\Lc_{n} \right] = \frac{\sqrt{\lambda_{n}}}{2\sqrt{2}} + \int\limits_{\Qc}L_{n}(x)dx +
O\left(\int\limits_{\Qc\setminus\Qc_{s}}|\Upsilon_{n}(x)|dx + N_{n}^{2-l}\sqrt{n} \cdot |\Mcc_{l}(n)|\right).
\end{equation}
The statement \eqref{eq:exp len lead semi-corr} of Theorem \ref{thm:main thr expl semi-corr} now follows
upon employing Lemma \ref{lem:L1 int, Ups bnd}\ref{it:int L(x) val} for evaluating the integral $\int\limits_{\Qc}L_{n}(x)dx$ on the r.h.s. of \eqref{eq:exp nod len int L(x) err}, and
Lemma \ref{lem:L1 int, Ups bnd}\ref{it:int Ups bnd} for bounding the error term
$\int\limits_{\Qc\setminus\Qc_{s}}|\Upsilon_{n}(x)|dx$.

\end{proof}

\subsection{Outline of the proof of Theorem \ref{thm:semi-correlations Lyosha}}

To prove Theorem \ref{thm:semi-correlations Lyosha}, our first step as in~\cite{BB}, is to restrict ourselves to the set of integers $n\in S$ with ``typical" factorization type.  This is accomplished (for square-free $n$) with the help of Lemma~\ref{bb1}. Fixing ``typical"  $n=p_1p_2\dots p_r$ with $p_1<p_2<\dots<p_r,$ the key observation then is that any non-trivial relation of the form
\[(\mu_{1},\ldots \mu_{l})\in\Ec_{n}^{l}: \: \sum\limits_{j=1}^{l} \mu^{j}_{1} = 0,\]
can be rewritten as a non-degenerate {\em quasi-linear} equation with respect to the Gaussian primes $p_j=\pi_j\bar{\pi}_j,$
\begin{equation}\label{model}
\sum_{s=1}^{\l}\text{Re}\left(i^{\alpha_s}\prod_{j=1}^r\pi_{j,s}^*\right)=0,
\end{equation}
where each $\pi_{j,s}^*=\{\pi_j,\bar{\pi_j}\}$ with rotation factor $\alpha_s\in\mathbb{Z}.$ Thus, having primes $p_1,p_2\dots p_{r-1}$ fixed, the equation \eqref{model} determines $\arg{\pi}_r,$ and therefore prime $p_r$ in a unique fashion (see the proof of Proposition \ref{squarefree} for the details). After conditioning on this value of $p_r$ and taking into account Lemma~\ref{bb1}, we deduce that equality \eqref{model} can occur only for small proportion of numbers $n\in S.$ This is accomplished in Proposition \ref{squarefree} for those $n\in S$ which are free of small prime factors and in Proposition \ref{key2} for general $n\in S.$

In order to show that, the set of accumulation points of the sequence of numbers $\{\widehat{\nu_{n}}(4)\}_{n\in S'}$ is
the whole of $[-1,1],$ we choose $n\in S$ of the form $n=p_n^mp,$ where $p_n$ and $p$ are appropriately chosen primes and $m\in S$ using classical result due to Kubilius (Lemma \ref{pnt}). This is the content of Proposition \ref{extremal}.

\subsection*{Outline of the paper}

The rest of the paper is organised as follows. The proof of a version of the Kac-Rice formula in Proposition \ref{prop:Kac Rice BAARW}
will be given in \S\ref{sec:Kac-Rice proof}, and the asymptotics for the Kac-Rice integral on the r.h.s. of \eqref{eq:BAARW Kac Rice} will be analysed throughout \S\ref{sec:sing meas bnd}-\S\ref{sec:bnd eff err sec}, as follows. An upper bound of Proposition \ref{prop:sing meas<<semi-corr} for the singular set $\Qc_{s}$ as in Definition \eqref{def:sing set} in terms of the semi-correlations set will be established in \S\ref{sec:sing meas bnd}, whereas a contribution of a single small square $Q_{ij}\subseteq\Qc_{s}$ of Proposition \ref{prop:sing sqr bnd}
will be controlled in \S\ref{sec:bound for a single small square}.

The perturbative analysis of Proposition \ref{prop:K1 asymp nonsing}
for the zero density on the nonsingular set will be carried out in \S\ref{sec:pert anal proof}, whose contribution to the
expected nodal length of $f_{n}$ will be evaluated in \S\ref{sec:bnd eff err sec}. A proof of Theorem
\ref{thm:semi-correlations Lyosha}, bounding the semi-correlation set, will be given in \S\ref{sec:Lyosha proof},
whereas some more technically demanding computations, required as part of proofs for the said results, will be performed in the appendix.

\section{Proof of Proposition \ref{prop:Kac Rice BAARW}: Kac-Rice formula for expected nodal length}
\label{sec:Kac-Rice proof}

In view of \cite[Theorem 6.3]{AW} (see also \cite[Proposition 1.2]{AW}), the equality \eqref{eq:BAARW Kac Rice} holds provided that the Gaussian distribution of $f_n(x)$ is non-degenerate for every $x \in \Qc$. It is easy to construct examples of numbers $n$, so that this non-degeneracy condition fails decisively for some points in $\Qc$. Let
\begin{equation}
\label{eq:Hc deg def}
\Hc_n=\{x\in\Qc:\: v_{n}(x)=0\}=\left\{ x \in \Qc \setminus \partial \Qc: \forall \mu \in \Ec_{n}/\sim, \;\;\; \mu_1 x_1 \in \mathbb{Z} \; \vee \;   \mu_2 x_2 \in \mathbb{Z}  \right\} \subseteq\Qc
\end{equation}
be the {\em degenerate} set.

\begin{lemma} \label{lem:degenerate set grid+fin}
Let $n$ be of the form \eqref{eq:n in S factor primes}, recall that $Q_{n}$ is given by \eqref{eq:Qn grid numb def}, and the grid
$\Gc_{n}$ as in \eqref{eq:grid G(Qn)}.
Then we have the decomposition
\begin{equation}
\label{eq:Hc degen decomp}
\Hc_{n} = \Gc_{n}\cup \Ac_{n},
\end{equation}
where $\Ac_{n}\subseteq \Qc$ is a finite set of isolated points in $\Qc$.

\end{lemma}

\begin{proof}

First we aim at proving the announced decomposition \eqref{eq:Hc degen decomp}.
Let $x=(x_{1},x_{2})\in \Hc_{n}$, whence, by the definition \eqref{eq:Hc deg def} of $\Hc_{n}$, for all $\mu\in\Ec_{n}/\sim$
either $\mu_{1}x_{1}\in \Z$ or $\mu_{2}x_{2}\in \Z$ holds, and recall that we assumed that $\mu_{1},\mu_{2}\ne 0$ for
all $\mu\in\Ec_{n}/\sim$ (as otherwise the corresponding summand in \eqref{eq:BAARW fn def} necessarily vanishes).
Assume that for some $\mu\in\Ec_{n}$ we have $\mu_{1}x_{1}\in\Z$,
and denote $l:=\mu_{1}x_{1}\in\Z$. Then, since $x_{1}\in [0,1]$ and $\mu_{1}^{2}+\mu_{2}^{2}=n$,
necessarily $|l|\le \sqrt{n}$, and $x_{1}=\frac{l}{\mu_{1}}$. Hence, if for some $\mu=(\mu_{1},\mu_{2})\in\Ec_{n}/\sim$
we have $\mu_{1}x_{1}\in \Z$,
and for some $\widetilde{\mu}=(\widetilde{\mu}_{1},\widetilde{\mu}_{2})\in\Ec_{n}/\sim$ we have $\widetilde{\mu}_{2}x_{2}\in\Z$,
then both coordinates $(x_{1},x_{2})$ belong to the finite set
$$\Ac'_{n}:=\left\{\frac{l}{\mu_{j}}:\:  0<|l|\le \sqrt{n}, \mu\in \Ec_{n}/\sim\right\},$$ so that, by prescribing $\Ac_{n}\subseteq\Ac'_{n}$,
the $\Ac_{n}$ in the decomposition \eqref{eq:Hc degen decomp}, is finite, provided that we prove that the rest of $\Hc_{n}$
is indeed the grid $\Gc_{n}$.

By the above, we may assume that $x\in \Hc_{n}$ satisfies
\begin{equation}
\label{eq:mu1x1 int for all mu}
\mu_{1}x_{1}\in\Z \text{ for all } \mu\in\Ec_{n}/\sim,
\end{equation}
and claim that in this
case necessarily $x_{1}$ is of the form
\begin{equation}
\label{eq:x1=k/Q}
x_{1}=\frac{k}{Q_{n}}
\end{equation}
for some $1\le k\le Q_{n}-1$, taking care
of the symmetric case ($\mu_{2}x_{2}\in\Z$ {\em for all} $\mu\in\Ec_{n}/\sim$) along identical lines. Once having \eqref{eq:x1=k/Q}
established, that would yield that $x\in\Gc_{n}$ on the grid (see \eqref{eq:grid G(Qn)}), and we would only have the burden of proving the converse
inclusion $\Gc_{n}\subseteq \Hc_{n}$ (which is easy).

To the end of proving \eqref{eq:x1=k/Q}, we let
\begin{equation}
\label{eq:Qn' gcd def}
Q'_{n}:= \gcd\{\mu_{1}:\:\mu\in\Ec_{n}/\sim\}=\gcd\{\mu_{2}:\:\mu\in\Ec_{n}/\sim\}
\end{equation}
be the greatest common divisor of the abscissas of all the lattice points in $\Ec_{n}$. Then, since the set of integers $d$ so that
$d\in x_{1}\in\Z$ is an ideal in $\Z$, we have $Q'_{n}x_{1}\in \Z$ (equivalently, since we can express $Q'_{n}$ as a linear combination
of $\{\mu_{1}:\: \mu\in\Ec_{n}/\sim\})$.
The above shows that \eqref{eq:mu1x1 int for all mu} is equivalent to the single condition $Q'_{n}x_{1}\in \Z$. That is, $x_{1}=\frac{k}{Q'_{n}}$,
and, since $x_{1}\in (0,1)$, we also get $1\le k\le Q'_{n}-1$, yielding \eqref{eq:x1=k/Q} (that, as mentioned above, in turn
implies $x\in\Gc_{n}$), once we prove that $Q_{n}=Q'_{n}$, to be shown next.

\vspace{2mm}

To this end we recall the prime decomposition \eqref{eq:n in S factor primes} of $n$,
and work in the ring of Gaussian integers $\Z[i]$ (which is a unique factorization domain, or, simply, UFD), where we think of
$\Ec_{n}\subseteq\R^{2}$ as embedded into $\C$, via the map $\mu=(\mu_{1},\mu_{2})\mapsto \mu_{1}+i\mu_{2}$.
For every prime $p_{j}$ in the decomposition \eqref{eq:n in S factor primes} we associate a prime element
$\pi_{j}\in\Z[i]$ with norm $\|\pi_{j}\|^{2}=p_{j}$, and take
$$a_{0}=a-2\lfloor a/2\rfloor=\begin{cases}
0 &a\text{ even}\\ 1 &a\text{ odd}
\end{cases}.$$ With this notation, and \eqref{eq:Qn grid numb def}, one may express every $\mu\in\Ec_{n}/\sim$ (i.e., as an element of $\C$,
up to a sign or complex conjugation) as
\begin{equation}
\label{eq:mu decomp Z[i]}
\mu=Q_{n}\cdot (1+i)^{a_{0}} \prod\limits_{j=1}^{s}\pi_{j}^{k_{j}}\overline{\pi_{j}}^{e_{j}-k_{j}},
\end{equation}
for some $0\le k_{j}\le e_{j}$, $j=1,\ldots s$.
This implies that $Q_{n}|Q_{n}'$ at once. To see that also $Q_{n}'|Q_{n}$, we further exploit the UFD property of $\Z[i]$,
implying, in particular, that the gcd in $\Z[i]$ is well-defined.

By the definition \eqref{eq:Qn' gcd def} of $Q_{n}'$, we have that  $Q_{n}'|\mu_{1}, \mu_{2}$ for all $\mu\in\Ec_{n}/\sim$, and so $Q_{n}'|\mu$ in $\Z[i]$, valid {\em for all} $\mu\in\Ec_{n}$, i.e. $$Q_{n}'|Q_{n}'':=\gcd\{\mu:\:\mu\in\Ec_{n}/\sim\}\in\Z[i].$$
However, by making the two choices $k_{j}:=e_{j}$, and $k_{j}:=0$, having only $Q_{n}\cdot (1+i)^{a_{0}}$ as a common factor in
\eqref{eq:mu decomp Z[i]}, it shows that $Q_{n}''=Q_{n}\cdot (1+i)^{a_{0}}$, and recalling that either $a_{0}=0$ or $a_{0}=1$,
the readily established $|Q_{n}|Q_{n}'$, and $Q_{n}'|Q_{n}''$ (so that $Q_{n}'$ could be either $Q_{n}$ or $Q_{n}\cdot (1+i)$, the latter not
being an integer number), this readily implies $Q_{n}'=Q_{n}$, that, as it was mentioned above, yields that $x\in \Gc_{n}$. To finish
the statement of Lemma \ref{lem:degenerate set grid+fin} it is sufficient to observe that if $x_{1}$ is of the form \eqref{eq:x1=k/Q}, then,
in light of \eqref{eq:Qn' gcd def}, and $Q_{n}=Q_{n}'$ above,
\eqref{eq:mu1x1 int for all mu} is satisfied, so that the inclusion $\Gc_{n}\subseteq\Hc_{n}$ holds.


\end{proof}

With the above preparatory result we are now in a position to conclude the proof of Proposition \ref{prop:Kac Rice BAARW}.

\begin{proof}[Proof of Proposition \ref{prop:Kac Rice BAARW}]

Around each points in $x_{i}\in \Hc_n$ we excise a small ball $\pazocal{B}(x_i,\varepsilon)$, and denote
$$\Qc_{\varepsilon}=\Qc \setminus \bigcup_{x_i \in \Hc_n} \pazocal{B}(x_i,\varepsilon),$$ with the intention to
apply the Kac-Rice method to evaluate the expected nodal length of the restriction
$f_n |_{\Qc_{\varepsilon}}$ of $f_{n}$ to the remaining set. That is, we excised the radius-$\varepsilon$
balls centred at each of the finitely many points $\Ac_{n}$, and, possibly, finitely many rectangles of the form
$(x_{1}-\varepsilon,x_{1}+\varepsilon) \times (0,1)$ and $(0,1)\times (x_{2}-\varepsilon,x_{2}+\varepsilon)$,
centred at the horizontal and vertical bars of the grid $\Gc_{n}$, in case it is non-empty.
Since outside $\Hc_{n}$, the random field $f_{n}$ satisfies the non-degeneracy hypothesis of \cite[Theorem 6.3]{AW},
the Kac-Rice formula \eqref{eq:Kac Rice meta} holds for $f_{n}$ restricted to $\Qc_{\varepsilon}$, asserting that the restricted
expected nodal length is given by
$$\E[\Lc(f_n |_{\Qc_{\varepsilon}} )] = \int\limits_{\Qc_\varepsilon}K_{1}(x)dx,$$
with $K_{1}$ as in \eqref{eq:K1 fn density def}.

On one hand, since, on recalling \eqref{eq:len(grid)=2(Q-1)}, the restricted nodal length
$\{\Lc(f_n |_{\Qc_{\varepsilon}})\}_{\varepsilon>0}$ is an increasing sequence of nonnegative random variables
with the a.s. limit
\begin{equation*}
\lim\limits_{\epsilon\rightarrow 0}\Lc(f_n |_{\Qc_{\varepsilon}}) = \Lc_{n}-\len(\Gc_{n})=\Lc(f_n)-2(Q_{n}-1),
\end{equation*}
the Monotone Convergence Theorem applied as $\epsilon\rightarrow 0$, yields
\begin{equation}
\label{eq:exp rest len eps->}
\lim_{\varepsilon \to 0}\E[\Lc(f_n |_{\Qc_{\varepsilon}} )]=\E[\Lc_n] - 2(Q_{n}-1).
\end{equation}
On the other hand, by the definition,
\begin{equation}
\label{eq:exp int rest}
\lim_{\varepsilon \to 0} \int\limits_{\Qc_\varepsilon}K_{1}(x)dx= \int\limits_{\Qc}K_{1}(x)dx,
\end{equation}
with the r.h.s. of \eqref{eq:exp int rest} finite on infinite.
The equality of the limits in \eqref{eq:exp rest len eps->} and \eqref{eq:exp int rest} show that the
main statement \eqref{eq:BAARW Kac Rice} of Proposition
\ref{prop:Kac Rice BAARW} holds, whether both the l.h.s. and r.h.s. of \eqref{eq:BAARW Kac Rice} are finite or infinite.
That $K_{1}\in L^{1}(\Qc)$ also follows, since the l.h.s. of \eqref{eq:BAARW Kac Rice} is finite by the deterministic bound
\eqref{eq:Yau conj} (alternatively, from the asymptotic analysis within Theorem \ref{thm:main thr expl semi-corr} of the integral on the r.h.s.
of \eqref{eq:BAARW Kac Rice}, with no circular logic).

\end{proof}

\section{Proof of Proposition \ref{prop:sing meas<<semi-corr}: Controlling the measure of the singular set}
\label{sec:sing meas bnd}

\subsection{Proof of Proposition \ref{prop:sing meas<<semi-corr}}

We will need the following auxiliary lemmas, whose proofs will be given in \S\ref{sec:aux lemmas} below.

\begin{lemma} \label{lem:Lip 1}
Let $Q_{ij} \subset \Qc_s$ be a singular small square. Then necessarily one of the followings holds:
\begin{enumerate}[a.]

\item For every $y \in Q_{ij}$,
$$|s_{n}(y)|> {\epsilon_{0}}/{2}.$$

\item For every $y \in Q_{ij}$,
$$|\tr(\Gamma_{n}(y))|>{\epsilon_{0}}/{2}.$$

\item For every $y \in Q_{ij}$,
$$|\det(\Gamma_{n}(y))|>{\epsilon_{0}}/{2}.$$

\end{enumerate}

\end{lemma}

Lemma \ref{lem:Lip 1} allows for the following decomposition of $\Qc_{s}$.

\begin{definition}[Singular decomposition]

\begin{enumerate}[a.]

\item The set $\Qc_{s,1} \subset \Qc_s$ is the union of all small squares $Q_{ij} \subset \Qc_s$ so that for all $x\in Q_{ij}$
the inequality
\begin{equation}
\label{eq:sn>eps/2}
|s_n(x)| >\epsilon_0/2
\end{equation}
is satisfied.

\item The set $\Qc_{s,2}$ is the union of all small squares $Q_{ij}\subseteq \Qc_{s}\setminus \Qc_{s,1}$,
so that either for all $x\in Q_{ij}$ the inequality
$$|\tr(\Gamma_{n}(x))|>\epsilon_{0}/2$$ holds, or for all $x\in Q_{ij}$, the inequality $$|\det(\Gamma_{n}(x))|>\epsilon_{0}/2$$
holds.

\item By Lemma \ref{lem:Lip 1},
\begin{equation}
\label{eq:Qcs=Qcs1+Qcs2}
\Qc_{s} = \Qc_{s,1} \cup \Qc_{s,2},
\end{equation}
(``singular decomposition"), with $\Qc_{s,1}$ and $\Qc_{s,2}$ disjoint save for boundary overlaps.

\end{enumerate}

\end{definition}

The respective measures of $\Qc_{s,1}$ and $\Qc_{s,2}$ will be bounded in the following lemma.
Recall that $\Mcc_{l}(n)$ is the length-$l$ spectral semi-correlation set \eqref{eq:M semi-corr def}.

\begin{lemma} \label{lem:Cheby 1}
For every $l\ge 4$ even integer we have the following bounds for the measures of the singular sets $\Qc_{s,1}$,
$\Qc_{s,2}$:
\begin{equation}
\label{eq:sing meas<<semi-corr 1}
\meas(\Qc_{s,1}) =  O\left( N_{n}^{-l} \cdot |\Mcc_{l}(n)| \right),
\end{equation}
\begin{equation}
\label{eq:sing meas<<semi-corr 2}
\meas(\Qc_{s,2})=O\left( N_{n}^{-l} \cdot |\Mcc_{l}(n)| \right),
\end{equation}
with the constant involved in the $`O'$-notation depending only on $l$ (also $\epsilon_{0}$ and $c_{0}$).
\end{lemma}

\begin{proof}[Proof of Proposition \ref{prop:sing meas<<semi-corr} assuming lemmas \ref{lem:Lip 1}-\ref{lem:Cheby 1}]
In light of the singular decomposition \eqref{eq:Qcs=Qcs1+Qcs2}, the statement \eqref{eq:sing meas<<semi-corr} of Proposition
\ref{prop:sing meas<<semi-corr} follows at once from Lemma \ref{lem:Cheby 1}.
\end{proof}

\subsection{Proofs of the auxiliary lemmas \ref{lem:Lip 1}-\ref{lem:Cheby 1}}
\label{sec:aux lemmas}

\begin{proof}[Proof of Lemma \ref{lem:Lip 1}]
We assume that for some $i,j$, there exists $x_{0} \in Q_{ij}$ with
\begin{equation}
\label{eq:sn(x0)>eps}
|s_n(x_{0})| > \epsilon_0,
\end{equation}
and claim that for all $x\in Q_{ij}$, the inequality
\begin{equation}
\label{eq:sn(x)>eps0/2}
|s_n(x)| > \epsilon_0/2
\end{equation}
holds (assuming $c_{0}$ is sufficiently small), i.e. scenario (a) of Lemma \ref{lem:Lip 1} prevails.
On recalling the definition \eqref{eq:sn var pert ident} of $s_{n}$, and differentiating \eqref{eq:sn var pert ident} explicitly,
it is easy to obtain the uniform bound
\begin{equation*}
\|\nabla s_{n}(x)\| \le c_{1}\cdot \sqrt{n},
\end{equation*}
with some {\em absolute} constant $c_{1}>0$. This readily implies that $s_{n}(\cdot/\sqrt{n})$ is a Lipschitz function with associated
constant absolute, i.e.
\begin{equation}
\label{eq:sn Lipschitz}
|s_{n}(x)-s_{n}(y)|\le c_{1}\sqrt{n}\|x-y\|.
\end{equation}
Hence, if $x\in Q_{ij}$, we have that
$\|x-x_{0}\| \le \sqrt{2}c_{0}\cdot \sqrt{n}$, and together with \eqref{eq:sn(x0)>eps} and \eqref{eq:sn Lipschitz}
this implies \eqref{eq:sn(x)>eps0/2},
so long as we choose $c_{0}>0$ sufficiently small, depending on $\epsilon_{0}$ (and $c_{1}$).

\vspace{2mm}

Essentially the same argument works for the other two scenarios (b) and (c) of Lemma \ref{lem:Lip 1},
on recalling the definition \eqref{eq:Gamma comp} of $\Gamma_{n}(\cdot)$, and
exploiting the Lipschitz property of both $\tr(\Gamma_{n}(\cdot))$ and $\det(\Gamma_{n}(\cdot))$, in place of $s_{n}(\cdot)$.
These are easy to establish to be with Lipschitz constant of order of magnitude at most $\sqrt{n}$, by differentiating the individual entries of $\Gamma_{n}(\cdot)$.

\end{proof}

\begin{proof}[Proof of Lemma \ref{lem:Cheby 1}]
We first aim to prove the first statement \eqref{eq:sing meas<<semi-corr 1} of Lemma \ref{lem:Cheby 1},
the proof of the second statement \eqref{eq:sing meas<<semi-corr 2} being quite similar, as explained in the last paragraph of this proof.
By the defining inequality \eqref{eq:sn>eps/2} of $\Qc_{s,1}$, holding for all $x\in\Qc_{s,1}$, we have that
\begin{equation*}
\Qc_{s,1} \subseteq \{x\in\Qc:\: |s_{n}(x)| > \varepsilon_{0}/2\},
\end{equation*}
so that we may apply the Chebyshev inequality to yield for every $l\ge 4$ even integer the bound
\begin{align*}
\meas(\Qc_{s,1})& \le
\meas\left(\left\{x \in \Qc_{s,1}: |s_n(x)|^l> \frac{\epsilon_0^l}{2^l}\right\}\right)\le
\frac{2^l}{\epsilon_0^l} \int\limits_{\Qc_{s,1}}  |s_n(x)|^l d x.
\end{align*}
By the definition \eqref{eq:sn var pert ident} of $s_{n}(x)$, we have $s_n(x)=A_n(x)+B_n(x)+C_{n}(x)$, where
\begin{align*}
&A_n(x)= \sum\limits_{\mu \in \Ec_n  / \sim} \cos(2 \mu_1 \pi x_1) ,\\
&B_n(x)= \sum\limits_{\mu \in \Ec_n  / \sim} \cos(2 \mu_2 \pi x_2),\\
&C_n(x)=- \sum\limits_{\mu \in \Ec_n  / \sim} \cos(2 \mu_1 \pi x_1) \cos(2 \mu_2 \pi x_2).
\end{align*}
Since $l$ is even, we may bound
\begin{equation}
\label{eq:AB}
\begin{split}
 \int\limits_{\Qc_{s,1}} s_n(x)^l d x& \le \frac{4^l}{N^l_n} \int\limits_{\Qc} [A_n(x)+B_n(x)+C_{n}(x)]^l d x\\&\le\frac{4^l \cdot 3^l }{N^l_n}  \int\limits_{\Qc} [A_n(x)^l+B_n(x)^l+C_{n}(x)^{l}] d x.
\end{split}
\end{equation}
Recalling the definition of the semi-correlation set $\Mcc_{l}(n)$ in \eqref{eq:M semi-corr def}, we observe that, for $i=1,2$, we
easily evaluate:
\begin{equation}
\label{eq:mom bnd sum semi-corr}
\int\limits_{\Qc} A_{n}(x)^{l}dx = \int\limits_{0}^{1} A_{n}(x)^{l}dx_{1} = \frac{2 \pi}{2^l}  |\Mcc_{l}(n)|,
\end{equation}
and the same
\begin{equation}
\label{eq:mom bnd sum semi-corr2}
\int\limits_{\Qc} B_{n}(x)^{l}dx =\frac{2 \pi}{2^l}  |\Mcc_{l}(n)|,
\end{equation}
whereas for the other integral in \eqref{eq:AB}, we recall the correlation set \eqref{eq:prec corr def} to bound
\begin{equation}
\label{eq:mom bnd sum semi-corr3}
\int\limits_{\Qc} C_n(x)^l d x \ll |\Rc_{l}(n)| \le |\Mcc_{l}(n)|,
\end{equation}
as, obviously, $\Rc_{l}(n)\subseteq \Mcc_{l}(n)$.

The first statement \eqref{eq:sing meas<<semi-corr 1} of Lemma \ref{lem:Cheby 1}
follows directly from \eqref{eq:AB}, \eqref{eq:mom bnd sum semi-corr}, \eqref{eq:mom bnd sum semi-corr2} and \eqref{eq:mom bnd sum semi-corr3},
via Chebyshev's inequality.
Finally, the same argument as above also yields the second statement \eqref{eq:sing meas<<semi-corr 2} of Lemma \ref{lem:Cheby 1},
upon observing that for every $y \in \Qc_{s,2}$ we have $|s_n(y)|\le \epsilon_0$ with $\epsilon_0$ small; so on $\Qc_{s,2}$ we can Taylor expand the function $$\frac{1}{v_n(x)}=\frac{1}{1-s_n(x)}=1+O(s_n(x))$$
that appear in the entries of the conditional covariance matrix $\Gamma_n(x)$.

\end{proof}

\section{Proof of Proposition \ref{prop:sing sqr bnd}: controlling the contribution of a small square}

\label{sec:bound for a single small square}

\subsection{Proof of Proposition \ref{prop:sing sqr bnd}}

We will first state the following lemma, whose proof will be given in
\S\ref{sec:lem bnd los sum prf} immediately below.

\begin{lemma}
\label{lem:bnd lose sum}
Let $Q\subseteq \Qc$ be an arbitrary square of side length $\frac{c_{0}}{\sqrt{n}}$ with $c_{0}>0$ sufficiently
small. We have the following uniform bound, holding for all $\eta,\mu\in \Ec_n$ with $\eta_{1}\ne 0$ and $\mu_{1}\ne 0$:
\begin{equation}
\label{eq:int Sij<<1/sqrt(n)}
\begin{split}
&\int\limits_{Q}
\frac{|\sin(\eta_{2} \pi x_{2})\sin(\mu_{2} \pi x_{2})|\cdot \left|\mu_{1}\sin(\eta_{1} \pi x_{1})\cos(\mu_{1} \pi x_{1})-\eta_{1}
\sin(\mu_{1} \pi x_{1})\cos(\eta_{1} \pi x_{1})\right|}{\left\{\sin^{4}(\eta_{1} \pi x_{1})\sin^{4}(\eta_{2} \pi x_{2})+
\sin^{4}(\mu_{1} \pi x_{1}) \sin^{4}(\mu_{2} \pi x_{2}) \right\}^{1/2}} dx_{1}dx_{2} \\&= O\left(\frac{1}{\sqrt{n}}\right),
\end{split}
\end{equation}
with constant involved in the `$O$'-notation depending only on the constant $c_{0}>0$.
\end{lemma}


\begin{proof}[Proof of Proposition \ref{prop:sing sqr bnd} assuming Lemma \ref{lem:bnd lose sum}]
Let $X$ and $Y$ be the conditional random variables
\begin{equation*}
X= \left.  \partial_{x_1} f_n(x) \right | f_n(x)=0, \hspace{1cm} Y= \left.  \partial_{x_2} f_n(x) \right | f_n(x)=0,
\end{equation*}
and recall that the joint distribution of $(X,Y)$ is centred Gaussian, with covariance matrix equal to $\Omega_{n}(x)$ in
\eqref{eq:Omega cond covar def} (and \eqref{eq:Gamma comp}-\eqref{eq:d2 def}), up to the normalising constant read from
\eqref{eq:Omega cond covar def}.
Then, with the aid of the Cauchy-Schwarz inequality, we obtain the inequality
\begin{align*}
\mathbb{E}[||\nabla f_n(x)||\; | \;f_n(x)=0]&\le \mathbb{E}[|X|] + \mathbb{E}[|Y|] \le\sqrt{(\E[|X|])^2}+ \sqrt{(\E[|Y|])^2} \\
&\le \sqrt{\E[X^{2}]}+ \sqrt{\E[Y^{2}]},
\end{align*}
so that we may bound the zero density \eqref{eq:K1 fn density def} as
\begin{equation*}
K_{1}(x) \le \widetilde{K}_{1,1}(x)+\widetilde{K}_{1,2}(x),
\end{equation*}
where $$\widetilde{K}_{1,1}(x) = \frac{1}{\sqrt{2\pi}\sqrt{v_{n}(x)}}\cdot \sqrt{\E[X^{2}]},$$
and $$\widetilde{K}_{1,2}(x) = \frac{1}{\sqrt{2\pi}\sqrt{v_{n}(x)}}\cdot \sqrt{\E[Y^{2}]}.$$

In what follows we are going to prove that
\begin{equation}
\label{eq:K11 cont<<1/sqrt(n)}
\int\limits_{Q} \widetilde{K}_{1,1}(x)dx = O\left(\frac{N_{n}^{2}}{\sqrt{n}}\right),
\end{equation}
and the same proof (with all coordinates switched) yields the other estimate
\begin{equation}
\label{eq:K12 cont<<1/sqrt(n)}
\int\limits_{Q} \widetilde{K}_{1,2}(x)dx = O\left(\frac{N_{n}^{2}}{\sqrt{n}}\right);
\end{equation}
together these imply \eqref{eq:int(K1) small cube << N^2/sqrt(n)}.
By \eqref{eq:Omega cond covar def}, and on recalling that
\begin{align*}
\frac{n \pi^2}{2}+ \frac{4 \pi^2}{N_n}b_{11;n}(x)=  \frac{16 \pi^2}{N_n} \sum_{\mu\in \Ec_{n}/\sim} \mu_1^2
 \sin^2 (  \mu_2 \pi x_2)  \cos^2 (  \mu_1 \pi x_1),
\end{align*}
we express the variance of the conditional derivative as
\begin{align*}
\mathbb{E}[X^2]=\frac{16 \pi^2}{N_n} \sum_{\mu\in \Ec_{n}/\sim} \mu_1^2
 \sin^2 (  \mu_2 \pi x_2)  \cos^2 (  \mu_1 \pi x_1)-  \frac{64 \pi^2}{N_n^2\, v_n(x)} d_{1;n}(x)^2
\end{align*}
since
\begin{align*}
v_n(x)&=\frac{16}{N_n} \sum_{\mu\in \Ec_{n}/\sim }
 \sin^2 (  \mu_1 \pi x_1)  \sin^2 (  \mu_2 \pi x_2),
\end{align*}
it follows that
\begin{align*}
\frac{\mathbb{E}[X^2] }{v_n(x)}&= \frac{16 \, \pi^2}{ N_n^2 \, v^2_n(x)  } \left[  v_n(x) N_n \sum_{ \mu\in \Ec_{n}/\sim } \mu_1^2
 \sin^2 (  \mu_2 \pi x_2)  \cos^2 (  \mu_1 \pi x_1)-  4  d_{1;n}(x)^2  \right]\\
  &=  \frac{16 \, \pi^2}{ N_n^2 \, v^2_n(x)  } \left[  v_n(x) N_n \sum_{ \mu\in \Ec_{n}/\sim  } \mu_1^2
 \sin^2 (  \mu_2 \pi x_2)  \cos^2 (  \mu_1 \pi x_1) \right.\\
 &\;\;- \left. 4^2  \left( \sum_{\mu\in \Ec_{n}/\sim}
  \mu_1    \sin (  \mu_1 \pi x_1)  \cos (  \mu_1 \pi x_1)  \sin^2 (  \mu_2 \pi x_2)   \right)^2  \right]\\
  &=  \frac{16^2 \, \pi^2}{ N_n^2 \, v^2_n(x)  } \left[  v_n(x) \frac{N_n}{16} \sum_{ \mu\in \Ec_{n}/\sim } \mu_1^2
 \sin^2 (  \mu_2 \pi x_2)  \cos^2 (  \mu_1 \pi x_1) \right.\\
 &\;\;-  \left.  \left( \sum_{ \mu\in \Ec_{n}/\sim }
  \mu_1    \sin (  \mu_1 \pi x_1)  \cos (  \mu_1 \pi x_1)  \sin^2 (  \mu_2 \pi x_2)   \right)^2  \right]
 \end{align*}
where
\begin{align*}
&v_n(x) \frac{N_n}{16} \sum_{ \mu\in \Ec_{n}/\sim } \mu_1^2
 \sin^2 (  \mu_2 \pi x_2)  \cos^2 (  \mu_1 \pi x_1)-    \left( \sum_{ \mu\in \Ec_{n}/\sim }
  \mu_1    \sin (  \mu_1 \pi x_1)  \cos (  \mu_1 \pi x_1)  \sin^2 (  \mu_2 \pi x_2)   \right)^2\\
  &  =
\sum_{ \eta \in \Ec_n /\sim}
 \sin^2 (  \eta_1 \pi x_1)  \sin^2 (  \eta_2 \pi x_2) \sum_{\mu \in \Ec_n /\sim} \mu_1^2
 \sin^2 (  \mu_2 \pi x_2)  \cos^2 (  \mu_1 \pi x_1)\\
 &\;\; -     \sum_{\eta \in \Ec_n /\sim}
  \eta_1    \sin (  \eta_1 \pi x_1)  \cos (  \eta_1 \pi x_1)  \sin^2 (  \eta_2 \pi x_2)
  \sum_{\mu \in \Ec_n /\sim}
  \mu_1    \sin (  \mu_1 \pi x_1)  \cos (  \mu_1 \pi x_1)  \sin^2 (  \mu_2 \pi x_2).
\end{align*}
It follows that we have the following explicit expression:
\begin{equation}
\label{eq:EX^2/vn double sum}
\begin{split}
\frac{\E[X^{2}]}{v_{n}(x)}& =
\frac{\pi^{2}}{\sum\limits_{\eta,\mu \in \Ec_n/\sim} \sin^{2}(\eta_{1} \pi x_{1})\sin^{2}(\eta_{2} \pi x_{2})\sin^{2}(\mu_{1} \pi x_{1})
\sin^{2}(\mu_{2} \pi x_{2})} \\
& \;\; \times \sum\limits_{\eta,\mu \in \Ec_n/\sim} \sin(\eta_{1} \pi x_{1})\sin^{2}(\eta_{2} \pi x_{2}) \cos(\mu_{1} \pi x_{1})\sin^{2}(\mu_{2} \pi x_{2})\\
&\;\; \times \left[\mu_{1}^{2}\sin(\eta_{1} \pi x_{1})\cos(\mu_{1} \pi x_{1})-\eta_{1}\mu_{1}\sin(\mu_{1} \pi x_{1})\cos(\eta_{1} \pi x_{1}) \right].
\end{split}
\end{equation}

Next, by grouping $(\eta,\mu)$ together with $(\mu,\eta)$ on the r.h.s. of \eqref{eq:EX^2/vn double sum}, this is equivalent to
\begin{equation*}
\begin{split}
\frac{\E[X^{2}]}{v_{n}(x)} &=
\frac{1}{2}\frac{\pi^{2}}{\sum\limits_{\eta,\mu \in \Ec_n/\sim} \sin^{2}(\eta_{1} \pi x_{1})\sin^{2}(\eta_{2}\pi x_{2}) \sin^{2}(\mu_{1} \pi x_{1}) \sin^{2}(\mu_{2} \pi x_{2})}\\
&\;\; \times\sum\limits_{\eta,\mu \in \Ec_n/\sim} \sin^{2}(\eta_{2} \pi x_{2}) \sin^{2}(\mu_{2} \pi x_{2})\\
&\;\; \times \left[\mu_{1}\sin(\eta_{1} \pi x_{1})\cos(\mu_{1} \pi x_{1})-\eta_{1}\sin(\mu_{1} \pi x_{1})\cos(\eta_{1} \pi x_{1}) \right]^{2}.
\end{split}
\end{equation*}
Hence
\begin{equation}
\label{eq:contr K11 lose den}
\begin{split}
&\widetilde{K}_{1,1}(x)\ll\\
&\frac{\left\{\sum\limits_{\eta,\mu \in \Ec_n/\sim} \sin^{2}(\eta_{2} \pi x_{2})\sin^{2}(\mu_{2} \pi x_{2})
\cdot \left[\mu_{1}\sin(\eta_{1} \pi x_{1})\cos(\mu_{1} \pi x_{1})-\eta_{1}\sin(\mu_{1} \pi x_{1})\cos(\eta_{1} \pi x_{1}) \right]^{2}\right\}^{1/2}}{\left\{\sum\limits_{\eta,\mu \in \Ec_n/\sim} \sin^{2}(\eta_{1} \pi x_{1})\sin^{2}(\eta_{2} \pi x_{2})\sin^{2}(\mu_{1}  \pi x_{1}) \sin^{2}(\mu_{2} \pi x_{2})\right\}^{1/2}}
\\& \le  \sum\limits_{\eta,\mu \in \Ec_n/\sim} \frac{|\sin(\eta_{2} \pi x_{2})\sin(\mu_{2} \pi x_{2})|
\cdot \left|\mu_{1}\sin(\eta_{1} \pi x_{1})\cos(\mu_{1} \pi x_{1})-\eta_{1}\sin(\mu_{1} \pi x_{1})\cos(\eta_{1} \pi x_{1}) \right|}{\left\{\sum\limits_{\eta,\mu \in \eta_{n}/\sim} \sin^{2}(\eta_{1} \pi x_{1})\sin^{2}(\eta_{2} \pi x_{2}) \sin^{2}(\mu_{1} \pi x_{1}) \sin^{2}(\mu_{2} \pi x_{2})\right\}^{1/2}}
\\& \le  \sum\limits_{\eta,\mu \in \Ec_n/\sim} \frac{|\sin(\eta_{2} \pi x_{2})\sin(\mu_{2} \pi x_{2})|
\cdot \left|\mu_{1}\sin(\eta_{1} \pi x_{1})\cos(\mu_{1} \pi x_{1})-\eta_{1}\sin(\mu_{1} \pi x_{1})\cos(\eta_{1} \pi x_{1}) \right|}{\left\{ \sin^{2}(\eta_{1} \pi x_{1})\sin^{2}(\eta_{2} \pi x_{2}) \sin^{2}(\mu_{1} \pi x_{1}) \sin^{2}(\mu_{2} \pi x_{2})\right\}^{1/2}},
\end{split}
\end{equation}
where we used the positivity of all the summands in the denominator, as well as the easy inequality
$$\sum\limits_{k=1}^{K}a_{k}^{2} \le  \left(\sum\limits_{k=1}^{K}|a_{k}|\right)^{2},$$
valid for any sequence of real numbers $\{a_{k}\}_{k=1}^{K}$.
The bound \eqref{eq:K11 cont<<1/sqrt(n)} (and similarly \eqref{eq:K12 cont<<1/sqrt(n)}), implying, as it was mentioned above,
the statement of Proposition \ref{prop:sing sqr bnd},
finally follows upon integrating the individual summands on the r.h.s. of \eqref{eq:contr K11 lose den},
and using Lemma \ref{lem:bnd lose sum} to bound the contribution of each one of them
(recall that we are allowed to assume that $\eta_{1},\mu_{1}\ne 0$).

\end{proof}

\subsection{Proof of Lemma \ref{lem:bnd lose sum}}
\label{sec:lem bnd los sum prf}

\begin{proof}

Let $\eta, \mu \in \Ec_n/\sim$, and $Q\subseteq \Qc$ be an arbitrary square of side length $\frac{c_{0}}{\sqrt{n}}$ with $c_{0}>0$ sufficiently small. Now we write
\begin{equation}
\label{eq:hlambda def}
\eta_{1} \pi x_{1} = k_{\eta}\pi+h_{\eta},
\end{equation}
where
$x_1 \in [0,1]$, $k_{\eta}= \left[ \frac{\eta_{1}x_{1}}{\pi} \right]\in\Z$ is the integer value of $\eta_{1}x_{1}/\pi$,
independent of $x_{1}$ by the above, and $h_{\eta}=h_{\eta}(x_{1})\in (-\pi/2,\pi/2]$. We also denote
\begin{equation}
\label{eq:x1lambda def}
x_{1}^{\eta}= \frac{k_{\eta}\pi}{\eta_{1}} = \pi x_{1}- \frac{h_{\eta}}{\eta_{1}};
\end{equation}
the numbers $k_{\mu}$, $x_{1}^{\mu}$ and $h_{\mu}=h_{\mu}(x_{1})$ are defined analogously, with $\eta_{1}$ replaced by $\mu_{1}$.
Note that, $h_{\eta}$ and $h_{\mu}$, both being linear functions of $x_{1}$, satisfy the relation
\begin{equation}
\label{eq:mu1h1-lam1h2=mult pi}
\eta_{1}h_{\mu}-\mu_{1}h_{\eta} = \pi\left(\mu_{1}k_{\eta}-\eta_{1}k_{\mu}\right) =
\eta_{1}\mu_{1} \left(x_{1}^{\eta} - x_{1}^{\mu}\right) ,
\end{equation}
that will be exploited below. Finally, we denote
\begin{equation}
\label{eq:d dist def}
d(\eta,\mu)=d_{ij}(\eta,\mu)=|x_{1}^{\eta}-x_{1}^{\mu}|,
\end{equation}
so that \eqref{eq:mu1h1-lam1h2=mult pi} reads
\begin{equation}
\label{eq:mu1h1-lam1h2=dist x1lam,mu}
|\eta_{1}h_{\mu}-\mu_{1}h_{\eta}| = \eta_{1}\mu_{1}\cdot d(\eta,\mu)
\end{equation}
crucially depending on $\eta,\mu$ and $Q$ only (but independent of $x_{1}$).\\

First, we assume, that $x_{1}^{\eta}\ne x_{1}^{\mu}$, i.e.
\begin{equation}
\label{eq:d(lam,mu)>0}
d(\eta,\mu) > 0,
\end{equation}
and by switching between $\eta$ and $\mu$ if necessary, we may assume w.l.o.g., that
\begin{equation}
\label{eq:x1mu>x1lam}
d(\eta,\mu) = x_{1}^{\mu}-x_{1}^{\eta} > 0.
\end{equation}
We expand the numerator of the integrand on l.h.s. of \eqref{eq:int Sij<<1/sqrt(n)} (first, without the absolute value):
\begin{equation*}
\begin{split}
&\mu_{1}\sin(\eta_{1} \pi x_{1})\cos(\mu_{1} \pi x_{1})-\eta_{1}
\sin(\mu_{1} \pi x_{1})\cos(\eta_{1} \pi x_{1}) \\&= (-1)^{k_{\eta}+k_{\mu}}\cdot\left[ \mu_{1}\sin(h_{\eta})\cos(h_{\mu}) - \eta_{1}\sin(h_{\mu})\cos(h_{\eta})\right],
\end{split}
\end{equation*}
so that
\begin{equation}
\label{eq:num exp dist err}
\begin{split}
&\left|\mu_{1}\sin(\eta_{1} \pi x_{1})\cos(\mu_{1} \pi x_{1})-\eta_{1}
\sin(\mu_{1} \pi x_{1})\cos(\eta_{1} \pi x_{1})\right| \\&=\left|\mu_{1}h_{\eta}- \eta_{1}h_{\mu}\right|+E_{\eta,\mu}(x_{1})
= \eta_{1}\mu_{1}d(\eta,\mu) +E_{\eta,\mu}(x_{1})
\end{split}
\end{equation}
by \eqref{eq:mu1h1-lam1h2=dist x1lam,mu}, where $d(\eta,\mu)$ is as in \eqref{eq:d dist def},
and the error term is bounded by
\begin{equation}
\label{eq:Elammu bnd}
|E_{\eta,\mu}(x_{1})| = O\left( \mu_{1}h_{\eta}^{3} +  \mu_{1}h_{\eta}h_{\mu}^{2} +\eta_{1}h_{\mu}h_{\eta}^{2}
+\eta_{1} h_{\mu}^{3} \right).
\end{equation}
For the denominator of the integrand on l.h.s. of \eqref{eq:int Sij<<1/sqrt(n)}, we have
\begin{equation}
\label{eq:K1 denom low bnd1}
\begin{split}
&\left\{\sin^{4}(\eta_{1} \pi x_{1}) \sin^{4}(\eta_{2} \pi x_{2})+
\sin^{4}(\mu_{1} \pi x_{1}) \sin^{4}(\mu_{2} \pi x_{2}) \right\}^{1/2}\\&\gg
\sin^{2}(\eta_{1} \pi x_{1})\sin^{2} (\eta_{2} \pi x_{2})+
\sin^2(\mu_{1} \pi x_{1}) \sin^2(\mu_{2} \pi x_{2}) \\&\gg {h_{\eta}}^{2}\sin(\eta_{2}x_{2})^{2} +
{h_{\mu}}^{2}\sin(\mu_{2}x_{2})^{2},
\end{split}
\end{equation}
and also
\begin{equation*}
\begin{split}
&\left\{\sin^4(\eta_{1} \pi x_{1})\sin^4(\eta_{2} \pi x_{2})+
\sin^4(\mu_{1} \pi x_{1}) \sin^4(\mu_{2} \pi x_{2})\right\}^{1/2} \\&\gg
\left|\sin(\eta_{1}  \pi x_{1})\sin(\eta_{2}  \pi x_{2})\sin(\mu_{1}  \pi x_{1})\sin(\mu_{2}  \pi x_{2})\right|
\gg \left|h_{\eta}h_{\mu} \sin(\eta_{2}  \pi x_{2})\sin(\mu_{2}  \pi x_{2})\right|,
\end{split}
\end{equation*}
and we plan to use the former inequality whenever we are going to bound the contribution
of the main term of \eqref{eq:num exp dist err}, and both inequalities for the error term in \eqref{eq:num exp dist err}.\\

We denote $s_{\eta}=s_{\eta}(x)=\sin(\eta_{2} \pi x_{2})$ and $s_{\mu}=s_{\mu}(x) =\sin(\mu_{2} \pi x_{2})$. With the newly
introduced notation, the denominator \eqref{eq:K1 denom low bnd1} is
\begin{equation}
\label{eq:denom low bnd sh}
\begin{split}
&\left\{\sin^{4}(\eta_{1} \pi x_{1}) \sin^{4}(\eta_{2} \pi x_{2})+
\sin^{4}(\mu_{1} \pi x_{1}) \sin^{4}(\mu_{2}\pi x_{2}) \right\}^{1/2} \gg s_{\eta}^{2}h_{\eta}^{2} +
s_{\mu}^{2}{h_{\mu}}^{2};
\end{split}
\end{equation}
as we will keep $x_{2}$ fixed, and aim at first to integrate w.r.t. $x_{1}$, we will treat $s_{\eta}$ and $s_{\mu}$
as parameters.
The estimates \eqref{eq:num exp dist err} and \eqref{eq:denom low bnd sh} imply that the integral in
\eqref{eq:int Sij<<1/sqrt(n)} is bounded by
\begin{equation}
\label{eq:int Sij<<main + error}
\begin{split}
&\int\limits_{Q}
\frac{|\sin(\eta_{2} \pi x_{2})\sin(\mu_{2} \pi x_{2})|\left|\mu_{1}\sin(\eta_{1} \pi x_{1})\cos(\mu_{1} \pi x_{1})-\eta_{1}
\sin(\mu_{1} \pi x_{1})\cos(\eta_{1} \pi x_{1})\right|}{\left\{\sin^{4}(\eta_{1} \pi x_{1})\sin^{4}(\eta_{2} \pi x_{2})+
\sin^{4}(\mu_{1}\pi x_{1})\sin^{4}(\mu_{2} \pi x_{2})\right\}^{1/2}} dx_{1}dx_{2}
\\&\ll \int\limits_{Q} |s_{\eta}s_{\mu}|  \frac{ \eta_{1}\mu_{1}d(\eta,\mu) +E_{\eta,\mu}(x_{1}) }{s_{\eta}^{2}h_{\eta}^{2} +
s_{\mu}^{2}{h_{\mu}}^{2}} dx_{1}dx_{2},
\end{split}
\end{equation}
and, in what follows, we are going to bound the contribution of the main and the error terms by separate arguments.

\vspace{2mm}

First, we bound the contribution of the main term in the integrand on the r.h.s. of \eqref{eq:int Sij<<main + error}.
To this end we notice that $h_{\eta}$ and $h_{\mu}$ are both linear functions
of $x_{1}$, so we are going to exploit their inter-dependence \eqref{eq:hlambda def} (and its analogue for $h_{\mu}$)
to write:
\begin{equation}
\label{eq:num deltax1}
\begin{split}
&s_{\eta}^{2}h_{\eta}^{2} + s_{\mu}^{2}h_{\mu}^{2} = s_{\eta}^{2}\eta_{1}^{2}(\pi x_{1}-x_{1}^{\eta})^{2} +
s_{\mu}^{2}\mu_{1}^{2}(\pi x_{1}-x_{1}^{\mu})^{2} \\&=
s_{\eta}^{2}\eta_{1}^{2}\triangle x_{1}^{2} +
s_{\mu}^{2}\mu_{1}^{2}(d(\eta,\mu)-\triangle x_{1})^{2},
\end{split}
\end{equation}
where we denoted
\begin{equation}
\label{eq:delta x1 trans coord}
\triangle x_{1}:=\pi x_{1}-x_{1}^{\eta},
\end{equation}
a linear transformation of the variable $x_{1}$, and used \eqref{eq:x1mu>x1lam}.
We complete the expression on the r.h.s. of \eqref{eq:num deltax1} to a square:
\begin{equation}
\label{eq:denom complete square deltax}
\begin{split}
&s_{\eta}^{2}h_{\eta}^{2} + s_{\mu}^{2}h_{\mu}^{2} = s_{\eta}^{2}\eta_{1}^{2}\triangle x_{1}^{2} +
s_{\mu}^{2}\mu_{1}^{2}(d(\eta,\mu)-\triangle x_{1})^{2} \\&=
\frac{1}{s_{\eta}^{2}\eta_{1}^{2}+s_{\mu}^{2}\mu_{1}^{2}} \left\{ \left[\left( s_{\eta}^{2}\eta_{1}^{2}+s_{\mu}^{2}\mu_{1}^{2} \right)\triangle x_{1} -s_{\mu}^{2}\mu_{1}^{2}d(\eta,\mu) \right]^{2}   + s_{\eta}^{2}s_{\mu}^{2}\eta_{1}^{2}\mu_{1}^{2}d^2(\eta,\mu)\right\};
\end{split}
\end{equation}
note that we may assume that $s_{\mu}\ne 0$ and $s_{\eta} \ne 0$ (holding outside a discrete set of $x_{2}$).
Substituting the identity \eqref{eq:denom complete square deltax} into \eqref{eq:num deltax1}, we may bound the
contribution of the main term of the integral on the r.h.s. of \eqref{eq:int Sij<<main + error} as
\begin{equation}
\label{eq:int Sij transform atan}
\begin{split}
&\int\limits_{Q} |s_{\eta}s_{\mu}| \frac{\eta_{1}\mu_{1}d(\eta,\mu)}{s_{\eta}^{2}h_{\eta}^{2} +
s_{\mu}^{2}h^{2}_{\mu}} dx_{1}dx_{2}= \eta_{1}\mu_{1}d(\eta,\mu)
\int\limits_{I_{j}}|s_{\eta}s_{\mu}|\left(s_{\eta}^{2}\eta_{1}^{2}+s_{\mu}^{2}\mu_{1}^{2}\right)dx_{2} \times\\&\times
\int\limits_{I_{i}}
\frac{1}{\left[\left( s_{\eta}^{2}\eta_{1}^{2}+s_{\mu}^{2}\mu_{1}^{2} \right)\triangle x_{1} -s_{\mu}^{2}\mu_{1}^{2}d(\eta,\mu) \right]^{2}   + s_{\eta}^{2}s_{\mu}^{2}\eta_{1}^{2}\mu_{1}^{2}d^{2}(\eta,\mu)}dx_{1}
\\&=
\frac{1}{\eta_{1}\mu_{1}d(\eta,\mu)}
\int\limits_{I_{j}}\frac{s_{\eta}^{2}\eta_{1}^{2}+s_{\mu}^{2}\mu_{1}^{2}}{|s_{\eta}s_{\mu}|}dx_{2}\int\limits_{\tilde{I}_{i}}
\frac{d\triangle x_{1}}{\left[\frac{ s_{\eta}^{2}\eta_{1}^{2}+s_{\mu}^{2}\mu_{1}^{2} }{\eta_{1}\mu_{1}s_{\eta}s_{\mu}d(\eta,\mu)}\triangle x_{1} -
\frac{s_{\mu}\mu_{1}}{s_{\eta}\eta_{1}} \right]^{2}+1 },
\end{split}
\end{equation}
where we have transformed the coordinates \eqref{eq:delta x1 trans coord}, and $\tilde{I}_{i}$ is some shift of the interval
$I_{i}$.

Another transformation of coordinates w.r.t. $\triangle x_{1}$ shows that,
denoting $\widetilde{\widetilde{I}}_{i}$ the new range of integration, \eqref{eq:int Sij transform atan} is equal to
\begin{equation}
\label{eq:int Sij main <<1/sqrt(n)}
\begin{split}
&\int\limits_{Q} |s_{\eta}s_{\mu}| \frac{\eta_{1}\mu_{1}d(\eta,\mu)}{s_{\eta}^{2}h_{\eta}^{2} +
s_{\mu}^{2}h_{\mu}^{2}} dx_{1}dx_{2} = \int\limits_{I_{j}}dx_{2} \int\limits_{\widetilde{\widetilde{{ I}_{i}}}}\frac{dt}{1+t^{2}} \ll
\frac{1}{\sqrt{n}},
\end{split}
\end{equation}
since the integral w.r.t. $t$ is bounded by an absolute constant.

\vspace{2mm}

Next, we turn to evaluating the contribution of the error term
\begin{equation*}
\int\limits_{Q} |s_{\eta}s_{\mu}| \cdot \frac{E_{\eta,\mu}(x_{1})}{s_{\eta}^{2}h_{\eta}^{2} +
s_{\mu}^{2}h^{2}_{\mu}} dx_{1}dx_{2}
\end{equation*}
to the integral \eqref{eq:int Sij<<main + error}.
By \eqref{eq:Elammu bnd}, we have
\begin{equation}
\label{eq:contr error h1^3 h2^3}
\begin{split}
&\int\limits_{Q}|s_{\eta}s_{\mu}| \frac{|E_{\eta,\mu}(x_{1})|}{s_{\eta}^{2}h_{\eta}^{2} +
s_{\mu}^{2}h_{\mu}^{2}} dx_{1}dx_{2} \\&\ll
\int\limits_{Q}|s_{\eta}s_{\mu}|  \frac{ |\mu_{1}h_{\eta}^{3}| }{s_{\eta}^{2}h_{\eta}^{2} + s_{\mu}^{2}h_{\mu}^{2}} dx_{1}dx_{2}+
\int\limits_{Q} |s_{\eta}s_{\mu}| \frac{ |\mu_{1}h_{\eta}h_{\mu}^{2}|}{s_{\eta}^{2}h_{\eta}^{2} + s_{\mu}^{2}h_{\mu}^{2}} dx_{1}dx_{2}\\& + \int\limits_{Q} |s_{\eta}s_{\mu}| \frac{ |\eta_{1}h_{\mu}h_{\eta}^{2}}{s_{\eta}^{2}h_{\eta}^{2} + s_{\mu}^{2}h_{\mu}^{2}} dx_{1}dx_{2} + \int\limits_{Q} |s_{\eta}s_{\mu}| \frac{ |\eta_{1} h_{\mu}^{3}|}{s_{\eta}^{2}h_{\eta}^{2} + s_{\mu}^{2}h_{\mu}^{2}} dx_{1}dx_{2};
\end{split}
\end{equation}
we will bound the $1$st integral on the r.h.s. of \eqref{eq:contr error h1^3 h2^3}, with the last one being bounded along similar lines,
and the $2$nd and the $3$rd ones are easier, as the corresponding numerator is divisible by both $h_{\eta}$ and $h_{\mu}$
(see the argument below).

We might bound the $1$st integral on the r.h.s. of \eqref{eq:contr error h1^3 h2^3}
(or the integral w.r.t. $x_{1}$),
using the Cauchy-Schwarz inequality
\begin{equation}
\label{eq:denominator lower CS}
s_{\eta}^{2}h_{\eta}^{2} + s_{\mu}^{2}h_{\mu}^{2} \gg s_{\eta}h_{\eta}s_{\mu}h_{\mu},
\end{equation}
to bound the denominator from below, so that
\begin{equation*}
\int\limits_{Q}|s_{\eta}s_{\mu}|  \frac{ |\mu_{1}h_{\eta}^{3}| }{s_{\eta}^{2}h_{\eta}^{2} + s_{\mu}^{2}h_{\mu}^{2}} dx_{1}
\le \int\limits_{Q} \frac{|s_{\eta}s_{\mu}|\cdot |\mu_{1}h_{\eta}^{3}| }{s_{\eta}h_{\eta}s_{\mu}h_{\mu} }dx_{1}
\ll \int\limits_{Q} \frac{|\mu_{1}h_{\eta}^{2}| }{|h_{\mu}|}dx_{1}
\end{equation*}
which has the unfortunate burden of having to deal with $h_{\mu}$ in the denominator, that could potentially be much
smaller than $h_{\eta}$. To deal with this obstacle we exploit the relation \eqref{eq:mu1h1-lam1h2=dist x1lam,mu} between $h_{\eta}$ and
$h_{\mu}$ once again, both being linear functions of $x_{1}$. We write
\begin{equation*}
\begin{split}
|\mu_{1} h_{\eta}^{3} |&=
|\mu_{1}h_{\eta} \cdot h_{\eta}^{2}| = |\mu_{1} h_{\eta} \pm \eta_{1}\mu_{1}d(\eta,\mu)| \cdot h_{\eta}^{2}
\le \eta_{1}h_{\eta}^{2}|h_{\mu}| + \eta_{1}\mu_{1}d(\eta,\mu)h_{\eta}^{2},
\end{split}
\end{equation*}
so that
\begin{equation}
\label{eq:int(Ii) relation}
\begin{split}
&\int\limits_{{I}_{i}} |s_{\eta}s_{\mu}| \frac{ |\mu_{1}h_{\eta}^{3}| }{s_{\eta}^{2}h_{\eta}^{2} + s_{\mu}^{2}h_{\mu}^{2}} dx_{1}
\le \int\limits_{{ I}_{i}} \frac{|s_{\eta}s_{\mu}| \, |\eta_{1}|  \, |h_{\mu}|\, h_{\eta}^{2}}
{s_{\eta}^{2}h_{\eta}^{2} + s_{\mu}^{2}h_{\mu}^{2}} dx_{1}+\int\limits_{{ I}_{i}} \frac{|s_{\eta}s_{\mu}\eta_{1}\mu_{1}| \, d(\eta,\mu) \, h_{\eta}^{2}}{s_{\eta}^{2}h_{\eta}^{2} + s_{\mu}^{2}h_{\mu}^{2}} dx_{1}.
\end{split}
\end{equation}

For the former integral on the r.h.s. of \eqref{eq:int(Ii) relation} we use the aforementioned idea
\eqref{eq:denominator lower CS} to bound the denominator from below
\begin{equation*}
\int\limits_{{ I}_{j}} \frac{|s_{\eta}s_{\mu}| \, |\eta_{1}|  \, |h_{\mu}| \, h_{\eta}^{2}}
{s_{\eta}^{2}h_{\eta}^{2} + s_{\mu}^{2}h_{\mu}^{2}} dx_{1} \ll
|\eta_{1}|\int\limits_{{I}_{j}} h_{\eta}dx_{1} \ll \sqrt{n}\cdot \frac{1}{\sqrt{n}} = 1,
\end{equation*}
since $|h_{\eta}| \le \frac{\pi}{2}$ is bounded, so that, after the integration w.r.t. $x_{2}$, we obtain
\begin{equation}
\label{eq:int Sij error diag}
\int\limits_{Q} \frac{|s_{\eta}s_{\mu}| \, |\eta_{1}|\, |h_{\mu}| \, h_{\eta}^{2}}
{s_{\eta}^{2}h_{\eta}^{2} + s_{\mu}^{2}h_{\mu}^{2}} dx_{1}dx_{2} \ll \frac{1}{\sqrt{n}}.
\end{equation}
Concerning the latter integral on the r.h.s. of \eqref{eq:int(Ii) relation} (or the double integral on $Q$),
since, as above, $h_{\eta}$ is bounded, we have
\begin{equation}
\label{eq:intSij error dist}
\int\limits_{Q} \frac{|s_{\eta}s_{\mu}\eta_{1}\mu_{1}| d(\eta,\mu)h_{\eta}^{2}}{s_{\eta}^{2}h_{\eta}^{2} + s_{\mu}^{2}h_{\mu}^{2}} dx_{1}dx_{2} \ll \int\limits_{Q} \frac{|s_{\eta}s_{\mu}\eta_{1}\mu_{1}| \, d(\eta,\mu)}{s_{\eta}^{2}h_{\eta}^{2} + s_{\mu}^{2}h_{\mu}^{2}} dx_{1}dx_{2} \ll \frac{1}{\sqrt{n}},
\end{equation}
readily evaluated above as the main term (see \eqref{eq:int Sij main <<1/sqrt(n)}).

Combining the estimates \eqref{eq:intSij error dist} and \eqref{eq:int Sij error diag}, and substituting them into
\eqref{eq:int(Ii) relation} (integrated w.r.t. $x_{2}$ on ${ I}_{j}$) yield the bound
\begin{equation}
\label{eq:error 1st}
\int\limits_{Q} |s_{\eta}s_{\mu}|  \frac{|\mu_{1}h_{\eta}^{3}| }{s_{\eta}^{2}h_{\eta}^{2} + s_{\mu}^{2}h_{\mu}^{2}} dx_{1}dx_{2} \ll \frac{1}{\sqrt{n}}
\end{equation}
for the $1$st integral on the r.h.s. of \eqref{eq:contr error h1^3 h2^3}, and its analogues for the other integrals on the r.h.s.
of \eqref{eq:contr error h1^3 h2^3} also follow (the $4$th is symmetric, whereas the $2$nd and the $3$rd are easier, with no need to deal with the denominator).
Substituting \eqref{eq:error 1st} and its analogues for the other three integrals in \eqref{eq:contr error h1^3 h2^3}
into \eqref{eq:contr error h1^3 h2^3}, we finally obtain a bound for the contribution of the error term
\begin{equation}
\label{eq:contr error << 1/sqrt(n)}
\int\limits_{Q} |s_{\eta}s_{\mu}| \frac{E_{\eta,\mu}(x_{1})}{s_{\eta}^{2}h_{\eta}^{2} +
s_{\mu}^{2}h^{2}_{\mu}} dx_{1}dx_{2} \ll \frac{1}{\sqrt{n}}.
\end{equation}
This, together with \eqref{eq:int Sij main <<1/sqrt(n)}, and \eqref{eq:int Sij<<main + error}, implies the statement
\eqref{eq:int Sij<<1/sqrt(n)} of Lemma \ref{lem:bnd lose sum} in this non-degenerate case \eqref{eq:d(lam,mu)>0}.\\

Finally, we treat the degenerate case $x_{1}^{\eta}=x_{1}^{\mu}$, or, equivalently,
\begin{equation}
\label{eq:mu1h1=lam1h2}
\eta_{1}h_{\mu}=\mu_{1}h_{\eta}.
\end{equation}
In this case the situation becomes easier to analyse, and the main terms of the expansion \eqref{eq:num exp dist err} vanishes
via \eqref{eq:mu1h1=lam1h2}, so that here \eqref{eq:num exp dist err} reads
\begin{equation}
\label{eq:degenerate all=error}
\left|\mu_{1}\sin(\eta_{1}x_{1})\cos(\mu_{1}x_{1})-\eta_{1}
\sin(\mu_{1}x_{1})\cos(\eta_{1}x_{1})\right| =E_{\eta,\mu}(x_{1}),
\end{equation}
with the same bound \eqref{eq:Elammu bnd} for the error term. The argument leading to \eqref{eq:int Sij error diag} above works unimpaired, with
no need to bound the extra term \eqref{eq:intSij error dist} here, as the precise identity \eqref{eq:mu1h1=lam1h2} reduces bounding the integral
\eqref{eq:int(Ii) relation} (after integrating w.r.t. $x_{2}$ on $I_{j}$)
to bounding \eqref{eq:int Sij error diag} with no remainder term \eqref{eq:intSij error dist}.
This shows that in this degenerate case, \eqref{eq:contr error << 1/sqrt(n)} holds, and, by \eqref{eq:degenerate all=error},
it also shows that the statement \eqref{eq:int Sij<<1/sqrt(n)} of Lemma \ref{lem:bnd lose sum} holds here.

\end{proof}

\section{Proof of Proposition \ref{prop:K1 asymp nonsing}: perturbative analysis on the non-singular set}
\label{sec:pert anal proof}


\begin{proof}
The proof of Proposition \ref{prop:K1 asymp nonsing} rests on a precise Taylor analysis for the density function $K_1$. We exploit the fact that the Gaussian expectation \eqref{eq:K1 fn density def} is an analytic function with respect to the parameters of the corresponding covariance matrix outside its singularities. It is then possible to Taylor expand $K_1$ explicitly, in the domain $\Qc \setminus \Qc_s$, where both $s_n(x)$ and the all the entries of $\Gamma_n(x)$ are small.

We first expand the factor
\begin{align}
\label{14:29Lon}
\frac{1 }{\sqrt{\var(f_{n}(x)) }}&= \frac{1}{\sqrt{1-s_n(x)}}=1+ \frac{1}{2} s_n(x) + \frac 3 8 s^2_n(x) +O(s^3_n(x)).
\end{align}
that appear in \eqref{eq:K1 fn density def}.
Next, we consider the Gaussian integral
\begin{align} \label{20:04 Par}
\Ic(\Gamma_n(x))=  \iint\limits_{\mathbb{R}^2} |z|  \exp \left\{ -\frac 1 2 z (I_2 + \Gamma_n(x))^{-1} z^t \right\} d z.
\end{align}
Observing that
\begin{align*}
 (I_2 + \Gamma_n(x))^{-1}&=I_2 - \Gamma_n(x) + \Gamma^2_n(x) +O(\Gamma^3_n(x)),
\end{align*}
we expand the exponential in \eqref{20:04 Par} as:
\begin{align*}
 & \exp \left\{ -\frac 1 2 z (I_2 + \Gamma_n(x))^{-1} z^t \right\}=   \exp \left\{ -\frac{z z^t}{2} \right\} \cdot \exp \left\{ \frac 1 2 z \left( \Gamma_n(x) - \Gamma^2_n(x) +O(\Gamma^3_n(x)) \right)  z^t \right\} \\
&\;\; = \exp \left\{ -\frac{z z^t}{2}\right\}  \left[ 1+ \frac 1 2 z \left( \Gamma_n(x) - \Gamma^2_n(x) +O(\Gamma^3_n(x)) \right) \right.  z^t \\
&\hspace{0.5cm} +\left. \frac{1}{2 \cdot 4} \left(  z ( \Gamma_n(x) - \Gamma^2_n(x) +O(\Gamma^3_n(x)))  z^t  \right)^2 + O\left(  z ( \Gamma_n(x) - \Gamma^2_n(x) +O(\Gamma^3_n(x)))  z^t  \right)^3  \right],
\end{align*}
so that the Gaussian integral \eqref{20:04 Par} is such that
\begin{align*}
\Ic(\Gamma_n(x))&=  \iint\limits_{\mathbb{R}^2} |z| \exp \left\{ -\frac{z z^t}{2} \right\}    \left[ 1+ \frac 1 2 z \Gamma_n(x)  z^t  -  \frac 1 2 z \Gamma^2_n(x)  z^t + \frac{1}{8} \left(  z  \Gamma_n(x)   z^t  \right)^2    \right] d z\\&+ O(\Gamma^3_n(x)).
\end{align*}
We introduce the following notation:
\begin{equation}
\label{eq:Ic0 def}
\Ic_0(\Gamma_n(x))=  \iint\limits_{\mathbb{R}^2} |z| \exp \left\{ -\frac{z z^t}{2} \right\},
\end{equation}
\begin{equation*}
\Ic_1(\Gamma_n(x))=  \frac 1 2  \iint\limits_{\mathbb{R}^2} |z| \exp \left\{ -\frac{z z^t}{2} \right\}  z \Gamma_n(x)  z^t d z,
\end{equation*}
\begin{equation*}
\Ic_2(\Gamma_n(x))=-  \frac 1 2  \iint\limits_{\mathbb{R}^2} |z| \exp \left\{ -\frac{z z^t}{2} \right\}      z \Gamma^2_n(x)  z^t  d z,
\end{equation*}
and
\begin{equation}
\label{eq:Ic3 def}
\Ic_3(\Gamma_n(x))=\frac{1}{8}
\iint\limits_{\mathbb{R}^2} |z| \exp \left\{ -\frac{z z^t}{2} \right\}   \left(  z  \Gamma_n(x)   z^t  \right)^2 d z.
\end{equation}
In Lemma \ref{lem:auxiliary 1}, postponed at the end of this section, we evaluate the terms $\Ic_i(\Gamma_n(x))$; we use Lemma \ref{lem:auxiliary 1} to write
\begin{align} \label{14:28Lon}
\Ic(\Gamma_n(x))&= \Ic_0(\Gamma_n(x)) + \Ic_1(\Gamma_n(x))+\Ic_2(\Gamma_n(x))+\Ic_3(\Gamma_n(x)) +  O(\Gamma^3_n(x)) \nonumber \\
&=\Ic_0(\Gamma_n(x))  + \Ic_1(\Gamma_n(x))+ \beta_1 \tr(\Gamma_n^2(x)) +\beta_2 [\tr(\Gamma_n(x))  ]^2   +  O(\Gamma^3_n(x))
\end{align}
where
\begin{align*}
 \beta_1=-\frac{3}{2^{3/2}}\pi^{3/2}+\frac{1}{8} \frac{15 \sqrt 2 }{4} \pi^{3/2}=- \frac{9 }{16\sqrt{2}}\pi^{3/2}, \hspace{0.5cm} \beta_2=  \frac{15 \sqrt 2}{64}  \pi^{3/2}.
\end{align*}
Finally we expand the factor
\begin{align*}
\frac{1}{\sqrt{  \det (I_2+\Gamma_n(x))}}.
\end{align*}
We know that $\Gamma_n(x)$ is symmetric and hence diagonalizable, we denote with $g_{1;n}(x)$ and $g_{2;n}(x)$ its eigenvalues. The eigenvalues of $I_2+\Gamma_n(x)$ are then $1+g_{1;n}(x)$ and $1+g_{2;n}(x)$, and we have
\begin{align*}
\det(I_2+\Gamma_n(x))&=[1+g_{1;n}(x)] [1+g_{2;n}(x)]
=1+ \tr(\Gamma_n(x)) + \det(\Gamma_n(x)).
\end{align*}
This implies that we can write
\begin{align*}
\frac{1}{\sqrt{  \det (I_2+\Gamma_n(x))}}&=1 - \frac{1}{2} \left[  \tr(\Gamma_n(x)) + \det(\Gamma_n(x)) \right] + \frac 3 8  \left[  \tr(\Gamma_n(x)) + \det(\Gamma_n(x)) \right]^2 + O(\Gamma^3_n(x)).
\end{align*}
We rewrite the third term on the right hand side of the last equation as follows
\begin{align*}
&\left[  \tr(\Gamma_n(x)) + \det(\Gamma_n(x)) \right]^2 \\
&= [g_{1;n}(x)+g_{2;n}(x)+ g_{1;n}(x) g_{2;n}(x)]^2\\
&= g^2_{1;n}(x)+ g^2_{2;n}(x) + g^2_{1;n}(x) g^2_{2;n}(x) + 2 g^2_{1;n}(x) g_{2;n}(x) + 2 g_{1;n}(x)g^2_{2;n}(x) +2 g_{1;n}(x)g_{2;n}(x)\\
&= g^2_{1;n}(x)+ g^2_{2;n}(x) +2 g_{1;n}(x) g_{2;n}(x) + O(\Gamma^3_n(x))\\
&= \left[  \tr(\Gamma_n(x)) \right]^2+ O(\Gamma^3_n(x))
\end{align*}
and we have that
\begin{align*}
\frac{1}{\sqrt{  \det (I_2+\Gamma_n(x))}}&= 1 - \frac 1 2  \tr(\Gamma_n(x)) - \frac 1 2  \det(\Gamma_n(x)) + \frac 3 8  [\tr(\Gamma_n(x))]^2 +  O(\Gamma^3_n(x)).
\end{align*}
Observing that  $g^2_{1;n}(x)$ and  $g^2_{2;n}(x)$ are eigenvalues of $\Gamma^2_n(x)$, we rewrite $\det (\Gamma_n(x))$ as follows
\begin{align*}
 \det (\Gamma_n(x)) &= g_{1;n}(x) g_{2;n}(x) \\
 &= \frac{1}{2} \left\{  [g_{1;n}(x) + g_{2;n}(x)]^2 - [g^2_{1;n}(x)+ g^2_{1;n}(x)] \right\} \\
 &= \frac 1 2 \left\{ \left[ \tr(\Gamma_n(x)) \right]^2 - \tr(\Gamma^2_n(x))  \right\},
\end{align*}
and we arrive at the following expansion
\begin{align} \label{14:27Lon}
\frac{1}{\sqrt{  \det (\Gamma_n(x)+I_2)}}&= 1 - \frac 1 2  \tr(\Gamma_n(x)) + \frac 1 8   [\tr(\Gamma_n(x))]^2 +\frac 1 4 \tr (\Gamma^2_n(x))+  O(\Gamma^3_n(x)).
\end{align}
In conclusion, in view of \eqref{14:29Lon}, \eqref{14:28Lon} and \eqref{14:27Lon}, we obtain the Taylor expansion of the density function $K_1$:
\begin{align*}
&K_{1,n}(x)= \frac{\sqrt n}{2^2 \sqrt{ \pi} } \left[  1+ \frac{1}{2} s_n(x) + \frac 3 8 s^2_n(x)  \right]\times \\
&\;\; \times \left[  \sqrt{2} \pi^{3/2} +\frac{3}{2^{3/2}} \pi^{3/2}  \tr (\Gamma_n(x)) - \frac{9 }{8 \cdot 2^{3/2}}\pi^{3/2} \tr(\Gamma_n^2(x)) + \frac{15 \sqrt 2}{8^2}  \pi^{3/2} [\tr(\Gamma_n(x))  ]^2  \right] \\
& \;\; \times \left[ 1 - \frac 1 2  \tr(\Gamma_n(x)) + \frac 1 8   [\tr(\Gamma_n(x))]^2 +\frac 1 4 \tr (\Gamma^2_n(x)) \right] + O( \sqrt n  \cdot  s^3_n(x))+O( \sqrt n  \cdot  \Gamma^3_n(x)) \\
&= \frac{\sqrt n \pi^{3/2}}{2^2 \sqrt{ \pi} } \left[  \sqrt 2 + \frac{s_n(x)}{\sqrt 2} + \frac{1}{2 \sqrt 2} \tr(\Gamma_n(x)) + \frac{3}{4 \sqrt 2} s^2_n(x) + \frac{1}{4 \sqrt 2} s_n(x) \tr(\Gamma_n(x)) \right.\\
&\;\; \left.- \frac{1}{16 \sqrt 2}\tr (\Gamma^2_n(x)) - \frac{1}{32 \sqrt 2 }    [\tr(\Gamma_n(x))]^2 \right] + O( \sqrt n \cdot s^3_n(x))+O( \sqrt n  \cdot  \Gamma^3_n(x)).
\end{align*}
\end{proof}
We state now Lemma \ref{lem:auxiliary 1}; the proof of this lemma is relegated to Appendix \ref{sec:auxiliary results}.
\begin{lemma} We have
\label{lem:auxiliary 1}
\begin{align*}
&\Ic_0(\Gamma_n(x))= \sqrt{2} \pi^{3/2},\\
&\Ic_1(\Gamma_n(x))= \frac{3}{2^{3/2}} \pi^{3/2}  \tr (\Gamma_n(x)),\\
&\Ic_2(\Gamma_n(x))=  -  \frac{ 3 }{2^{3/2}}  \pi^{3/2} \tr (\Gamma^2_n(x)),\\
&\Ic_3(\Gamma_n(x))= \frac{15 \sqrt 2}{64}  \pi^{3/2}  \left\{ 2 \, \tr(\Gamma_n^2(x)) + [\tr(\Gamma_n(x)) ]^2  \right\}.
\end{align*}
\end{lemma}

\section{Proof of Lemma \ref{lem:L1 int, Ups bnd}: Boundary effect and error term}
\label{sec:bnd eff err sec}

We first prove  Lemma \ref{lem:L1 int, Ups bnd}\ref{it:int L(x) val}. By the definition \eqref{eq:Ln asympt K1 def} of $L_{n}(x)$, we have
\begin{equation} 
\label{eq:int Ln ind}
\begin{split}
\int\limits_{\Qc} L_{n}(x) d x &=  \frac{\sqrt n \pi}{4 \sqrt 2  } \int\limits_{\Qc} \left[    s_n(x) + \frac{1}{2 }\tr(\Gamma_n(x)) + \frac{3}{4 } s^2_n(x) + \frac{1}{4 } s_n(x)\tr(\Gamma_n(x))  \right]  d x  \\
&\;\;+ \frac{\sqrt n \pi}{4 \sqrt 2 }
\int\limits_{\Qc}  \left[ - \frac{1}{16 }\tr(\Gamma^2_n(x)) - \frac{1}{32  }    [\tr(\Gamma_n(x))]^2 \right] d x.
\end{split}
\end{equation}

The following technical lemma evaluates the individual integrals encountered within \eqref{eq:int Ln ind}.

\begin{lemma}
\label{21:33 Lon}

The integrals of the individual terms on the r.h.s. of \eqref{eq:int Ln ind} admit the following asymptotics.

\begin{enumerate}[a.]

\item  $$\int\limits_{\Qc}  s_n(x) d x=0.$$

\item $$\int\limits_Q\tr(\Gamma_n(x)) d x =   -  \frac{6}{ N_n } + O \left(\frac{1}{N^2_n} \right) + O\left( N_{n}^{-l-1}  |\Mcc_{l}(n)|\right).$$

\item $$\int\limits_{\Qc}  s^2_n(x) d x=\frac{5}{N_n}.$$

\item $$\int\limits_Q s_n(x)
\tr(\Gamma_n(x)) d x =   \frac{2}{ N_n}  +  O \left(\frac{1}{N^2_n} \right) + O\left( N_{n}^{-l-1}  |\Mcc_{l}(n)|\right).$$

\item $$\int\limits_{\Qc} \tr(\Gamma^2_n(x)) dx =\frac{2^2}{N_n} \left[1+2^5 M_4(n) \right] +O\left(\frac{1}{N^2_n} \right)+ O\left(n^{-1-1/2} N_n^{-l-3}  |\Mcc_{l}(n)|\right).$$

\item $$\int\limits_{\Qc} [\tr(\Gamma_n(x))]^2 d x=\frac{2^2}{N_n} \left[2^6 M_4(n)-3 \right]+O\left(\frac{1}{N^2_n} \right)+  O\left( N_{n}^{-l-1}  |\Mcc_{l}(n)|\right).$$

\item \label{s^3} $$\int\limits_{\Qc} s^3_n(x) d x=O \left(\frac{1}{N^2_n} \right).$$

\item \label{trace_cubed} $$\int\limits_{\Qc} [ \tr( \Gamma_n(x)) ]^3 d x=O \left(\frac{1}{N^2_n} \right).$$

\end{enumerate}

\end{lemma}

The proof of Lemma \ref{21:33 Lon} is postponed till Section \ref{sec: proof of lemma lem:L1 int, Ups bnd auxiliary results}.

\begin{proof}[Proof of Lemma \ref{lem:L1 int, Ups bnd} assuming Lemma \ref{21:33 Lon}]

We substitute the various asymptotic statements of Lemma \ref{21:33 Lon} into \eqref{eq:int Ln ind} to obtain
\begin{align*} 
&\int\limits_{\Qc} L_{n}(x) d x \\
&=  \frac{\sqrt n \pi}{4 \sqrt 2  } \left\{- \frac 1 2  \frac{6}{ N_n } + \frac 3 4 \frac{5}{N_n}+ \frac 1 4 \frac{2}{ N_n} - \frac{1}{16} \frac{2^2}{N_n} \left[1+2^5 M_4(n) \right] - \frac{1}{32} \frac{2^2}{N_n} \left[2^6 M_4(n)-3 \right] \right\} \\
&\;\;+O\left(\frac{\sqrt n }{N^2_n} \right)+ O\left(n^{-1/2} N_{n}^{-l-1}  |\Mcc_{l}(n)|\right)\\
&=  \frac{\sqrt n \pi}{4 \sqrt 2  }\frac{1}{N_n} \left[ \frac{11}{2^3} - 2^4 M_4(n) \right] +O\left(\frac{\sqrt n }{N^2_n} \right)+ O\left(n^{1/2} N_{n}^{-l-1}  |\Mcc_{l}(n)|\right).
\end{align*}
The latter formula, expressed in terms of the sequence $\{\widehat{\nu_n}(4)\}$ of the $4$th Fourier coefficients of $\nu_n$ gives the statement of Lemma \ref{lem:L1 int, Ups bnd}\ref{it:int L(x) val}. The proof of Lemma \ref{lem:L1 int, Ups bnd}\ref{it:int Ups bnd} follows immediately from \eqref{eq:Ups error bnd} and Lemma \ref{21:33 Lon} parts \ref{s^3} and \ref{trace_cubed}.

\end{proof}

\section{Proof of Theorem \ref{thm:semi-correlations Lyosha}: length-$l$ spectral semi-correlations}
\label{sec:Lyosha proof}

We begin by proving Theorem~\ref{thm:semi-correlations Lyosha} in the case of square-free numbers. To this end, for any fixed $K\ge 1$ we introduce the set
\[ \Omega_{M,K}=\{n\le M,\ \text{rad}(n)=n,\ p\vert n\in S\Rightarrow\ p\ge K\},\]
and let $\Omega_M:=\Omega_{M,1}=S\cap [1,M].$
The following Lemma is borrowed from~\cite{BB}.
\begin{lemma}\label{bb1}
For $m\in \Omega_{M,K},$ let $m=p_1\cdotp_2\cdot\dots p_r$ be its factorization with $K< p_1<p_2\dots<p_r.$ Then as $M\to\infty$ we have $p_s>2^{s\Phi(s)}$ for $1\le s\le r$ holds for all $m\in\Omega_{M,K}\setminus \Omega_{M,K}^{(1)},$ where the exceptional set $\Omega_{M,K}^{(1)}$ has cardinality
\[|\Omega_{M,K}^{(1)}|\le \eta(K, \Phi)|\Omega_{M,K}|\]
with $\eta(K, \Phi)\to 0$ as $K\to\infty.$ If $\Phi(x)=o(\log x),$ then we can choose $\eta(K,\Phi)=K^{-1+\delta}$ for every fixed $\delta>0.$
\end{lemma}
The next proposition shows that the number of solutions of~\eqref{eq:M semi-corr def} is small for almost all $m\in \Omega_{M,K}.$
\begin{proposition}\label{squarefree}
Let $\delta>0$ be fixed. If $K\ge K(\delta)$ and $M\to\infty,$ then for all but $K^{-1+\delta}|\Omega_{M,K}|$ elements $m\in\Omega_{M,K}$ the equation~\eqref{eq:M semi-corr def} has $O(N_m^k)$ solutions.
\end{proposition}
\begin{proof} Let $\tilde{S}\in S$ be the subset for which~\eqref{eq:M semi-corr def} has a nontrivial solution. For any prime $p$ we write $p=\pi\cdot\bar{\pi}$ where $\pi$ is the corresponding Gaussian prime with $\text{arg}(\pi)\in [0,\pi/2].$ For any integer $s\ge 1$ we introduce the set
\[\mathcal{F}_s=\left\{n\in\Omega_{M,K},\ \omega(n)=s,\ n\in \tilde{S};\ \forall d\ne n, d\vert n\Rightarrow d\in S\setminus \tilde{S}\right\}.\]
Fix $s\ge 1$ and consider $n\in\mathcal{F}_s$ with a given factorization $n=p_1\cdot p_2\dots p_s,$ $K<p_1<p_2<\dots<p_s.$ We have that there exist  integer points $\{\xi\}_{j=1}^{2k}$ with $\|\xi_j\|=\sqrt{n}$ and $\varepsilon_j\in \{-1,0,1\},$ $1\le j\le 2k$ with vanishing linear combination
\[\sum_{j=1}^{2k}\varepsilon_i\text{Re}(\xi_j)=0.\]
Each point $\xi_r$ can be uniquely written as a product $\xi_r=i^{k_{\xi_r}}\prod_{j\le s }\pi_{j,r}^*$ where  each $\pi_{j,r}^*\in \{\pi_j,\bar{\pi}_j\}$ and $k_{\xi_r}\in\{0,1,2,3\}.$
We now regroup the terms in the last expression by collecting $\pi_s$ and $\bar{\pi_s}$ into different  summands to end up with an equivalent form
\begin{equation}\label{iterate1}
\text{Re}(\pi_sA_{s-1})+\text{Re}(\bar{\pi_s}B_{s-1})=0,
\end{equation}
where each $A_{s-1},B_{s-1}$ consists of the sum of at most $2k-1$ terms composed of first $(s-1)$ Gaussian primes. Setting $\pi_s=|\pi_s|e^{i\phi_s},$ $A_{s-1}=|A_{s-1}|e^{ia_{s-1}}$ and $B_{s-1}=|B_{s-1}|e^{ib_{s-1}},$ the equation \eqref{iterate1} can be rewritten as
\[|A_{s-1}|\cos (\phi_s+a_{s-1})+|B_{s-1}|\cos (b_{s-1}-\phi_s)=0.\] Using elementary trigonometric identities we get
\[\cos(\phi_s)(|A_{s-1}|\cos (a_{s-1})+|B_{s-1}|\cos (b_{s-1}))=\sin(\phi_s)(|A_{s-1}|\sin (a_{s-1})-|B_{s-1}|\sin (b_{s-1})).\]
Consequently,
\[\tan (\phi_s)=\frac{|A_{s-1}|\cos (a_{s-1})+|B_{s-1}|\cos (b_{s-1})}{|A_{s-1}|\sin (a_{s-1})-|B_{s-1}|\sin (b_{s-1})}\]
and so $\tan (\phi_s)$ is determined uniquely unless $$|A_{s-1}|\cos (a_{s-1})+|B_{s-1}|\cos (b_{s-1})=|A_{s-1}|\sin (a_{s-1})-|B_{s-1}|\sin (b_{s-1})=0.$$ In the latter case we must have
\begin{equation}\label{iterate}
{|A_{s-1}|\cos (a_{s-1})+|B_{s-1}|\cos (b_{s-1})}=0.
\end{equation}
Since $n\in \mathcal{F}_s$ we have  $p_1p_2\dots p_{s-1}=\tilde{n}\vert n$ and so by definition $\tilde{n}\in S\setminus\tilde{S}.$ This implies that~\eqref{iterate} and consequently~\eqref{iterate1} must be trivial with $A_{s-1}=B_{s-1}=0.$ This contradicts the definition of $n\in\mathcal{F}_s.$\\
Hence $\tan(\phi_s)$ is determined uniquely and so is the corresponding prime $p_s.$ Indeed, if $\pi^{(1)}_s=x^2+y^2$ and $p^{(2)}_s=a^2+b^2$ are two primes corresponding to the same angle $\phi_s,$ then
\[\tan (\phi_s)=\frac{a}{b}=\frac{x}{y}\cdot\]
Since $(a,b)=(x,y)=1,$ we have that $|a|=|x|$ and $|b|=|y|$ and therefore $a^2+b^2=x^2+y^2=p^{(1)}_s=p^{(2)}_s:=p_s.$\\ We are now ready to estimate the number of $m\in\Omega_{M,K}$ which give rise to a nontrivial solution of~\eqref{eq:M semi-corr def}.
By Lemma~\ref{bb1}, we can restrict ourselves to $m=p_1p_2\dots p_r\in\Omega_{M,K},$ with $K<p_1<p_2<\dots<p_r$ and $p_j\ge 2^{j\Phi(j)}$ for any $1\le j\le r$ and some slowly growing function $\Phi(x)$ to be determined later.
Clearly for each such $m,$ there exists unique $1\le s\le r$ such that the product $p_1p_2\dots p_s\in \mathcal{F}_s.$
Given $K<p_1<p_2<\dots <p_{s-1}$ we can form at most $2^{2k(s-1)}$ sums $A_{s-1}$ and $B_{s-1}$ and thus produce at most $2^{2k(s-1)}$ distinct $n=p_1p_2\dots p_{s-1}p_s\in\Omega_{M,K}.$ By Lemma~\ref{bb1}, $p_s\ge \max\{2^{s\Phi(s)},K\}$ and therefore the total number of elements in $\tilde{S}\cap [1,M]$ induced by the elements in $\mathcal{F}_s$ is at most
\[\ll 2^{2ks}\left|\left\{m\le \frac{M}{\max{\{K,2^{s\Phi(s})\}}},\ m\in \Omega_{M,K}\right\}\right|,\]
where the last bound comes from ``conditioning" on the value of $p_s$ and noting that $m/p_s\in\Omega_{M,K}.$
Summing this over different ranges for $s$ and choosing $\Phi(x)=o(\log x)$ in the same way as in \cite{BB} yields the desired conclusion.
\end{proof}
We are now ready to handle the general case.
\begin{proposition}\label{key2}
For all but $o\left(\frac{M}{\sqrt{\log M}}\right)$ elements $m\in S\cap [1,M],$ the equation~\eqref{eq:M semi-corr def} has $O(N_m^k)$ solutions.
\end{proposition}
\begin{proof}
Fix large $K>0$ and consider $\mathcal{P}_K=\prod_{p\le K}p.$ We decompose
$n=n_{K}n_b$ where $(n_b,\mathcal{P}_K)=1$ and $\text{rad}(n_K)\vert \mathcal{P}_K.$ For any fixed part $n_K=\xi\bar{\xi}$ we count the number of $n\in \tilde{S}\cap [1,M]$ with fixed $n_K|n$ and $(\frac{n}{n_K},\mathcal{P}_K)=1.$  In this way~\eqref{eq:M semi-corr def} reduces to
\[\sum_{i=1}^{2k}\text{Re}(\alpha_i\xi_i)=0\]
with $\alpha_i\vert n_K$
and $\|\xi_i\|=\sqrt{n_b}$ for $1\le i\le 2k.$
We now follow the proof of Proposition~\ref{squarefree} regarding $\alpha_i$ as fixed coefficients. Let $\Omega_{4,1}(n)$ denote the number of prime divisors $p=1(\bmod 4)$ of $n$ counting multiplicity. Given $n_K,$ we have at most $2^{2k\Omega_{4,1}(n_K)}$ choices for the coefficients $\alpha_i.$ Note that the number of integers $n\in\Omega_M$ for which $p^2|n$ for some prime $p\ge K,$ is bounded above by
\[\sum_{p>K}|\Omega_{\frac{M}{p^2}}|\ll \sum_{p>K}\frac{M}{p^2\sqrt{\log M}}\ll \frac{M}{K\log K \sqrt{\log M}}\]
and thus give negligible contribution.
Consequently, we can restrict ourselves to the set of integers with $\text{rad}(n_b)=n_b.$ The number of $n\in \tilde{S}\cap [1,M]$ induced in this way, after appealing to Proposition~\ref{squarefree} is bounded above by
\begin{align*}
\sum_{\text{rad}(n_K)\vert \mathcal{P}_K}2^{2k\Omega_{4,1}(n_K)}|\Omega_{M/(n_K),K}|&\ll \sum_{\text{rad}(n_K)\in\mathcal{P}_K}\frac{4^{k\Omega_{4,1}(n_k)}}{n_K}\left((K^{-1+\delta}+o(1))\frac{M}{\sqrt{\log M}}\right)\\&\ll \prod_{p\le K,\ p=1\pmod 4}\left(\sum_{j\ge 0}\frac{4^{kj}}{p^{j}}\right)\left({K^{-1+\delta}}+o(1)\right)\frac{M}{\sqrt{\log M}}\\&
\ll \left(\frac{(\log K)^{2k+1}}{K^{1-\delta}}+o(1)\right)\frac{M}{\sqrt{\log M}}\cdot
\end{align*}
The result now follows by letting $K\to\infty.$
\end{proof}
Proposition~\ref{key2} now implies that $|\tilde{S}\cap [1,M]|=o\left(\frac{M}{\sqrt{\log M}}\right)$ and we may take $S'=S\setminus \tilde{S}.$
In order to prove the last part of Theorem~\ref{thm:semi-correlations Lyosha}, we require the following classical result of Kubilius.
\begin{lemma}\label{pnt}
The number of Gaussian primes $\omega$ in the sector $0\le \alpha\le\arg \omega\le \beta\le 2\pi,$ $|\omega|^2\le u$ is equal to
\[\frac{2}{\pi}(\beta-\alpha)\int_{2}^{u}\frac{1}{\log v}dv+O\left(u\exp(-b\sqrt{\log u})\right),\]
for some positive constant $b\in\mathbb{R}.$
\end{lemma}
We now choose particular ``thin" subset of $S,$ to guarantee the desired limiting behaviour of $\{\widehat{\nu_{n}}(4)\}_{n\in S'}.$
\begin{proposition}\label{extremal}
For any $s\in[-1,1],$ there exists a sequence $\{n_i\}_{i\ge 1},$ with $N_{n_i}\to\infty$ whenever $i\to\infty,$ such that $\widehat{\nu_{n_i}}(4)\to s$ and equation~\eqref{eq:M semi-corr def} has $O(N_{n_i}^k)$ solutions for any $i\ge 1.$
\end{proposition}
\begin{proof}
Fix large $m\ge 1$ and small $\varepsilon>0.$ By Lemma~\ref{pnt}, we can select prime $p_{n}=\pi_{n}\bar{\pi_n},$ $p_n=1\pmod 4$ and $|\text{arg}(\pi_n)|\le \frac{\varepsilon}{100m}.$ We further select prime $p$ such that
\begin{equation}\label{fourier1}
|\widehat{\nu_{p}}(4)-s|\le \frac{\varepsilon}{2}\cdot\end{equation} Consider the number of the form $n=p_n^{m}p.$ It is easy to see that
\begin{equation}\label{fourier2}
|\widehat{\nu_{p_n}}(4)-1|\le \frac{\epsilon}{2m}\end{equation} and $N_n> 2^m.$ Using an elementary inequality $$\left|\prod_{j\le r}x_j-\prod_{j\le r}y_j\right|\le \sum_{j\le r}|x_j-y_j|,$$
valid for $|x_j|,|y_j|\le 1$
and the fact that $\widehat{\nu_{n}}(4)=(\widehat{\nu_{p_n}}(4))^m\widehat{\nu_{p}}(4),$ we can estimate
$$|\widehat{\nu_{n}}(4)-s|\le m|\widehat{\nu}_{p_n}(4)-1|+|\widehat{\nu}_{p}(4)-s|\le m\cdot\frac{\varepsilon}{2m}+\frac{\varepsilon}{2}=\varepsilon.$$ We are left to show that equation~\eqref{eq:M semi-corr def} has only trivial solutions for appropriately chosen values of $p_n,p,$ which satisfy~\eqref{fourier1} and~\eqref{fourier2}. Let $\pi_n=r_ne^{i\phi}$ and $p=\pi\cdot\bar{\pi}$ with $\arg{\pi}=\alpha.$ Clearly each integer point on the circle of radius $\sqrt{n}$ can be written as $\xi_j=\sqrt{n}e^{i(j\phi\pm\alpha+r\frac{\pi}{2})}$ for some $|j|\le m$ and $r=\{0,1,2,3\}.$ For such defined $n,$ upon taking real parts, equaton~\eqref{eq:M semi-corr def} can be rewritten in the form
\begin{align}\label{eq1}
\sum_{j=1}^{2k}\varepsilon_j\cos \left(\l_j\phi\pm\alpha+\frac{\pi r_j}{2}\right)=0,
\end{align}
where $\varepsilon_j=\{+1,-1\}$ and $|\l_j|\le m$ for $1\le j\le 2k.$ By collecting terms with equal phases $\l_j\phi$ and using elementary trigonometric identities, we can rewrite~\eqref{eq1} in the form
\begin{equation}\label{trigonometric}
F_{r}(\phi)=\sum_{j=1}^{r}\cos (m_j\phi)(\alpha_j\cos \alpha+\beta_j\sin \alpha)+\sin (m_j\phi)(\alpha^{(1)}_j\cos \alpha+\beta_j^{(1)}\sin \alpha)=0,\end{equation}
where $1\le r\le 2k$ and $0\le m_1<m_2<\dots m_r$ with $-2k\le \alpha_j,\alpha_j^{(1)},\beta_j,\beta_j^{(1)}\le 2k.$
Since $k$ is fixed, there are only finitely many choices for the coefficients $\alpha_j,\alpha_j^{(1)},\beta_j,\beta_j^{(1)}$
and therefore we can select angle $\alpha$ for which the corresponding prime $p$ satisfies~\eqref{fourier1} and such that $a\sin{\alpha}+b\cos(\alpha)\ne 0$ for all $a,b\in\mathbb{Z}$ with $|a|+|b|\ne 0$ and $|a|,|b|\le 2k.$ Now since each $F_r(\phi)$ is a trigonometric polynomial of a total degree at most $4k,$ each non degenerate equation~\eqref{trigonometric} has at most $4k$ solutions. Since there are only finitely many of such equations, we can select $p_n$ sufficiently large which satisfies~\eqref{fourier2} and the corresponding equation~\eqref{eq1} has only trivial solutions. This concludes the proof.
 \end{proof}
 Combining Proposition~\ref{key2} and Proposition~\ref{extremal} completes the proof of Theorem~\ref{thm:semi-correlations Lyosha}.
\begin{remark}
It is possible to give a construction of $n$ in Proposition~\ref{extremal} which is square-free. The idea is to use Lemma \ref{pnt} and choose inductively sequence of primes $p_1<p_2<\dots <p_m$ such that $p_j=\pi_j\bar{\pi_j}$ and $|\text{arg}(\pi_j)|\le \frac{1}{m^2}$ with the property that for any $1\le r\le m$ we have $\cos(r\cdot\text{arg}(\pi_m))\notin \text{span}_{\mathbb{N}}\{\cos(\sum_{i\le m-1}a_i\text{arg}(\pi_i))\}$ where $-m\le a_i\le m,$ $a_i\in\mathbb{N}.$ The latter can be ensured by taking sufficiently sparse sequence of primes $p_1,p_2.\dots$ Now select $n=p\cdot p_1\dots p_m$ with $|\hat{\mu_p}(4)-s|\le \delta$ and $\delta$ sufficiently small. We leave the details to the interested reader.
\end{remark}

\appendix

\section{Proof of Lemma \ref{lem:Omega cond covar def}: evaluating $\Omega_{n}(\cdot)$, the normalised conditional covariance matrix}
\label{sec:covariance matrix}

The ultimate goal of this section is evaluating the $2\times 2$ (normalised) covariance
matrix $\Omega_{n}(\cdot)$ in \eqref{eq:Omega cond covar def}
of $\nabla f_{n}(x)$, conditioned upon $f_{n}(x)=0$. First, we will need to evaluate (\S\ref{apx:eval uncond mat}) the (unconditional) $3\times 3$
covariance matrix $\Sigma_{n}(x)$ of $(f_{n}(x),\nabla f_{n}(x)$, and then apply (\S\ref{apx:eval cond mat}) the well-known procedure
for conditioning in the Gaussian case.

\subsection{Evaluating the unconditional covariance matrix}
\label{apx:eval uncond mat}

\begin{lemma}
\label{lem:Sigma uncond eval}
Let $\Sigma_{n}(x)$ be the $3\times 3$ covariance matrix of $(f_{n}(x),\nabla f_{n}(x))$. Then $\Sigma_{n}(x)$
could be expressed as
\begin{align*}
\Sigma_n(x)=\left(
\begin{array}{cc}
v_{n}(x) & B_{n}(x) \\
B_{n}^t(x) & C_{n}(x)
\end{array}%
\right),
\end{align*}
where $v_{n}(\cdot)=\var (f_{n}(x))$ is as in \eqref{eq:vn var ident},
\begin{equation*}
B_n(x)=\mathbb{E}[f_n(x) \cdot \nabla_y f_{n}(y)] \Big|_{x=y}
\end{equation*}
and
\begin{equation*}
C_n(x)=\mathbb{E}[ \nabla_x f_n(x) \otimes  \nabla_y f_n(y)] \Big|_{x=y}.
\end{equation*}
The $2\times 1$ matrix $B^{t}_{n}(x)$ is:
\begin{equation}
\label{eq:Bt expr expl}
B^{t}_{n}(x) = \frac{ 4 \pi }{N_{n}}\left( \begin{matrix} \sum_{\mu\in\Ec_{n}/\sim}   \mu_1
 \sin ( 2 \mu_1 \pi x_1)  \big[ 1-\cos(2 \mu_2 \pi x_2) \big] \\  \sum_{\mu\in\Ec_{n}/\sim }   \mu_2
 \sin ( 2 \mu_2 \pi x_2)  \big[ 1-\cos(2 \mu_1 \pi x_1) \big]  \end{matrix}\right),
\end{equation}
and the entries of the $2\times 2$ matrix $C_{n}(x)$ is given by:
\begin{equation*}
C_{n;11}(x)= \frac{ n \pi^2}{2}  +  \frac{4 \pi^2}{N_{n}} \sum_{\mu\in\Ec_{n}/\sim} \mu_1^2 \Big[ \cos  ( 2 \mu_1 \pi x_1) - \cos(2 \mu_2 \pi x_2) - \cos(2  \mu_1 \pi x_1)  \cos(2  \mu_2 \pi x_2) \Big]
\end{equation*}
\begin{equation*}
C_{n;12}(x)=C_{n;21}(x) = \frac{ 4 \pi^2}{N_{n}} \sum_{\mu\in\Ec_{n}/\sim} \mu_1  \mu_2
 \sin (2  \mu_1 \pi x_1)  \sin (2  \mu_2 \pi x_2),
\end{equation*}
and
\begin{equation*}
C_{n;22}(x)= \frac{ n \pi^2}{2}  + \frac{ 4\pi^2}{N_{n}} \sum_{\mu\in\Ec_{n}/\sim} \mu_2^2 \Big[
\cos( 2 \mu_2 \pi x_2)-  \cos( 2 \mu_1 \pi x_1) -  \cos( 2 \mu_2 \pi x_2) \cos( 2 \mu_1 \pi x_1) \Big].
\end{equation*}

\end{lemma}

\begin{proof}

First, that the $(1,1)$ entry of $\Sigma_{n}(x)$ is $v_{n}(x)$ is self-evident, being the variance of $f_{n}(x)$.
For the other elements of $\Sigma_{n}(x)$, recalling the covariance function $r_{n}(x,y)$ in \eqref{eq:rn covar def}, we have that
\begin{equation}
\label{eq:Bn rn der}
B_n(x)=\nabla_y \, r_n(x,y)  \Big|_{x=y},\\
\end{equation}
and
\begin{equation}
\label{eq:Cn rn der}
C_n(x)=\left(\nabla_x \otimes \nabla_y\right)  r_n(x,y)\Big|_{x=y}.
\end{equation}
What follows below is the, somewhat technically demanding, routine, evaluation of the derivatives on the r.h.s. \eqref{eq:Bn rn der} and
\eqref{eq:Cn rn der}.

We first compute the entries of the $1\times 2$ matrix $B_{n}(x)$. Note that, without stationarity (nor unit variance),
successive derivatives are not uncorrelated at a fixed point. We have
\begin{align*}
\frac{\partial}{\partial y_1} r_n(x,y) \Big|_{x=y} &=  \frac{16}{N_{n}} \sum_{\mu\in\Ec_{n}/\sim}
 \sin (  \mu_1 \pi x_1)  \sin (  \mu_2 \pi x_2)  \frac{\partial}{\partial y_1}  \sin (  \mu_1 \pi y_1) \sin (  \mu_2 \pi y_2) \Big|_{x=y} \\
 &= \frac{16}{N_{n}} \sum_{\mu\in\Ec_{n}/\sim}
 \sin (  \mu_1 \pi x_1)  \sin (  \mu_2 \pi x_2)   \cos (  \mu_1 \pi y_1) \sin (  \mu_2 \pi y_2)   \mu_1 \pi \Big|_{x=y} \\
 &= \frac{16  \pi }{N_{n}} \sum_{\mu\in\Ec_{n}/\sim}
 \sin (  \mu_1 \pi x_1)  \sin^2 (  \mu_2 \pi x_2)   \cos (  \mu_1 \pi x_1)    \mu_1,
\end{align*}
since $\sin(2 \theta)=2 \sin \theta \cos \theta$, we can also write
\begin{align*}
\frac{\partial}{\partial y_1} r_n(x,y) \Big|_{x=y} &= \frac{8  \pi }{N_{n}} \sum_{\mu\in\Ec_{n}/\sim}
 \sin ( 2 \mu_1 \pi x_1)  \sin^2 (  \mu_2 \pi x_2)     \mu_1.
\end{align*}
Similarly
\begin{align*}
\frac{\partial}{\partial y_2} r_n(x,y) \Big|_{x=y}
 &= \frac{8  \pi }{N_{n}} \sum_{\mu\in\Ec_{n}/\sim}  \sin ( 2 \mu_2 \pi x_2)
 \sin^2 (  \mu_1 \pi x_1) \mu_2.
\end{align*}
Further, since $2 \sin^2 \theta=1-\cos(2 \theta)$, we have
\begin{equation*}
\frac{\partial}{\partial y_1} r_n(x,y) \Big|_{x=y} = \frac{ 4 \pi }{N_{n}} \sum_{\mu\in\Ec_{n}/\sim}   \mu_1
 \sin ( 2 \mu_1 \pi x_1)  \big[ 1-\cos(2 \mu_2 \pi x_2) \big]
\end{equation*}
and
\begin{equation*}
\frac{\partial}{\partial y_2} r_n(x,y) \Big|_{x=y} = \frac{ 4 \pi }{N_{n}} \sum_{\mu\in\Ec_{n}/\sim }   \mu_2
\sin ( 2 \mu_2 \pi x_2)  \big[ 1-\cos(2 \mu_1 \pi x_1) \big],
\end{equation*}
which is \eqref{eq:Bt expr expl}.
Note that the both entries of $B_{n}(x)$ are purely oscillatory (though do not vanish identically, as in the stationary case).

\vspace{2mm}

The entries of the matrix $C_{n}(x)$ are
\begin{align*}
&\frac{\partial}{\partial y_1} \frac{\partial}{\partial x_1} r_n(x,y) \Big|_{x=y} =
\frac{16 \pi^2}{N_{n}} \sum_{\mu\in\Ec_{n}/\sim} \mu_1^2
 \cos (  \mu_1 \pi x_1)  \sin (  \mu_2 \pi x_2)  \cos (  \mu_1 \pi y_1) \sin (  \mu_2 \pi y_2) \Big|_{x=y}\\
 &= \frac{16 \pi^2}{N_{n}} \sum_{\mu\in\Ec_{n}/\sim} \mu_1^2
 \cos^2 (  \mu_1 \pi x_1)  \sin^2 (  \mu_2 \pi x_2)\\
  &= \frac{4 \pi^2}{N_{n}} \sum_{\mu\in\Ec_{n}/\sim} \mu_1^2 +  \frac{4 \pi^2}{N_{n}} \sum_{\mu \in \Ec_n  / \sim } \mu_1^2 \Big[ \cos  ( 2 \mu_1 \pi x_1) - \cos(2 \mu_2 \pi x_2) - \cos(2  \mu_1 \pi x_1)  \cos(2  \mu_2 \pi x_2) \Big]\\
  &= \frac{ n \pi^2}{2}  +  \frac{4 \pi^2}{N_{n}} \sum_{\mu\in\Ec_{n}/\sim} \mu_1^2 \Big[ \cos  ( 2 \mu_1 \pi x_1) - \cos(2 \mu_2 \pi x_2) - \cos(2  \mu_1 \pi x_1)  \cos(2  \mu_2 \pi x_2) \Big],
 \end{align*}
where we reused the identity $2 \sin^2 \theta=1-\cos(2 \theta)$ and $2 \cos^2 \theta=1+\cos(2 \theta)$, and applied the equality
\begin{align*}
\sum_{\mu\in\Ec_{n}/\sim} \mu_1^2 = \sum_{\mu \in  \Ec_n / \sim } \mu_1^2 \pm \sum_{\mu\in\Ec_{n}/\sim} \mu_2^2 = n \frac{N_{n}}{4} - \sum_{\mu\in\Ec_{n}/\sim} \mu_2^2= n \frac{N_{n}}{4} - \sum_{\mu\in\Ec_{n}/\sim} \mu_1^2.
\end{align*}
Similarly we have
\begin{align*}
&\frac{\partial}{\partial y_2} \frac{\partial}{\partial x_2} r_n(x,y) \Big|_{x=y}
=\frac{16 \pi^2}{N_{n}} \sum_{\mu\in\Ec_{n}/\sim} \mu_2^2
\sin^2 (  \mu_1 \pi x_1)  \cos^2 (  \mu_2 \pi x_2)\\
&= \frac{ n \pi^2}{2}  + \frac{ 4\pi^2}{N_{n}} \sum_{\mu\in\Ec_{n}/\sim} \mu_2^2 \Big[
\cos( 2 \mu_2 \pi x_2)-  \cos( 2 \mu_1 \pi x_1) -  \cos( 2 \mu_2 \pi x_2) \cos( 2 \mu_1 \pi x_1) \Big].
\end{align*}
The off-diagonal entries of $C_{n}(x)$ are both equal to
 \begin{align*}
&\frac{\partial}{\partial y_1} \frac{\partial}{\partial x_2} r_n(x,y) \Big|_{x=y}   =
\frac{16 \pi^2}{N_{n}} \sum_{\mu\in\Ec_{n}/\sim} \mu_1  \mu_2
 \sin (  \mu_1 \pi x_1)  \cos (  \mu_2 \pi x_2)  \cos (  \mu_1 \pi y_1) \sin (  \mu_2 \pi y_2) \Big|_{x=y}\\
 &=  \frac{16 \pi^2}{N_{n}} \sum_{\mu\in\Ec_{n}/\sim} \mu_1  \mu_2
 \sin (  \mu_1 \pi x_1)  \cos (  \mu_2 \pi x_2)  \cos (  \mu_1 \pi x_1) \sin (  \mu_2 \pi x_2) \\
 &=  \frac{ 4 \pi^2}{N_{n}} \sum_{\mu\in\Ec_{n}/\sim} \mu_1  \mu_2
 \sin (2  \mu_1 \pi x_1)  \sin (2  \mu_2 \pi x_2).
 \end{align*}

\end{proof}

\subsection{Proof of Lemma \ref{lem:Omega cond covar def}}
\label{apx:eval cond mat}

\begin{proof}

Let $\Theta_{n}(x)$ be the conditional covariance matrix of the Gaussian vector $(\nabla f_n(x) | f_n(x)=0)$,
related to $\Omega_{n}(x)$ via the normalisation
\begin{equation}
\label{eq:Omega Theta rel}
\Omega_n(x)= \frac{2}{n \pi^2} \Theta_n(x).
\end{equation}
Once the (unconditional) covariance matrix $\Sigma_{n}$ of $(f_{n}(x),\nabla f_{n}(x))$ is, following Lemma \ref{lem:Sigma uncond eval},
known, the conditional covariance matrix $\Theta_{n}(x)$ is given
by the standard Gaussian transition formula:
\begin{align*}
\Theta_n(x)=C_n(x) - \frac{1}{v_n(x)}B^t_n(x) B_n(x).
\end{align*}

We use \eqref{eq:Bt expr expl} to compute
\begin{align*}
\frac{1}{v_n(x)}B^t_n(x) B_n(x) = \frac{64 \pi^2}{N_{n}^2 v_n(x)}  \left(\begin{matrix}d_{1;n}(x)^{2} &d_{1;n}(x)\cdot d_{2;n}(x) \\
d_{1;n}(x)\cdot d_{2;n}(x) &d_{2;n}(x)^{2}\end{matrix} \right),
\end{align*}
and then, on recalling the notation \eqref{eq:b11 def}-\eqref{eq:d2 def}, we have
\begin{align*}
\Theta_n(x) &= \frac{n \pi^2}{2} I_2 + \frac{4 \pi^2}{N_{n}} \left(\begin{matrix}b_{11;n}(x) &b_{12;n}(x) \\ b_{21;n}(x) &b_{22;n}(x)\end{matrix} \right)
 -\frac{64 \pi^2}{N_{n}^2 v_n(x)}
\left(\begin{matrix}d_{1;n}(x)^{2} &d_{1;n}(x)\cdot d_{2;n}(x) \\
d_{1;n}(x)\cdot d_{2;n}(x) &d_{2;n}(x)^{2}\end{matrix} \right).
\end{align*}
All that remains to obtain an explicit expression for $\Omega_{n}(x)$ from $\Theta_{n}(x)$
in the form prescribed in \eqref{eq:Omega cond covar def} (and \eqref{eq:Gamma comp}-\eqref{eq:d2 def})
is merely to invoke the normalisation \eqref{eq:Omega Theta rel}:
\begin{align*}
\Omega_n(x)= \frac{2}{n \pi^2} \Theta_n(x) &= I_2 + \Gamma_n(x).
\end{align*}

\end{proof}

\section{Proof of Lemma \ref{lem:auxiliary 1}: Auxiliary results for Proposition \ref{prop:K1 asymp nonsing}}
\label{sec:auxiliary results}

\begin{proof}
Recall the notation \eqref{eq:Ic0 def}-\eqref{eq:Ic3 def} for whatever quantities are been evaluated in turn.
A transformation of variables shows at once that
\begin{align*}
\Ic_0(\Gamma_n(x))&=  \iint\limits_{\mathbb{R}^2} |z| \exp \left\{ -\frac 1 2 z z^t \right\}  d z= 2 \pi \int\limits_{0}^{\infty} \exp\left\{-\frac{1}{2}  \rho^2 \right\} \rho^{2}  \, d \rho  = 2 \pi \sqrt{\frac{\pi}{2}} = \pi^{3/2} \sqrt{2}.
\end{align*}
Let $\Gamma_{n}(x)=(\gamma_{ij;n}(x))_{ij=1,2}$, with $\gamma_{ij;n}(x)$ the entries of $\Gamma_n(x)$ (cf. \eqref{eq:Gamma comp}).
We have
\begin{align*}
\Ic_1(\Gamma_n(x))&= \frac 1 2  \iint\limits_{\mathbb{R}^2} |z| \exp \left\{ -\frac 1 2 z z^t \right\}  z \Gamma_n(x)  z^t d z\\
&= \frac 1 2  \iint\limits_{\mathbb{R}^2} |z| \exp \left\{ -\frac 1 2 z z^t \right\}  \left( z_1^2 \gamma_{11;n}(x) + z_2^2 \gamma_{22;n}(x) +2 z_1 z_2 \gamma_{12;n}(x)  \right)  d z \\
&= a_1 \tr (\Gamma_n(x))  +a_2 \gamma_{12;n}(x), 
\end{align*}
where
\begin{align*}
a_1&=\frac 1 2 \iint\limits_{\mathbb{R}^2} |z| \exp \left\{ -\frac 1 2 z z^t \right\}   z_1^2  d z= \frac 1 2 2 \pi \int\limits_0^{\infty} \rho \exp \left\{ -\frac{\rho^2}{2} \right\} \frac{\rho^2}{2} \rho d \rho =3\frac{ \pi^{3/2}}{2^{3/2}} ,\\
a_2&= \iint\limits_{\mathbb{R}^2} |z| \exp \left\{ -\frac 1 2 z z^t \right\}   z_1 z_2  d z=0.
\end{align*}
Further,
\begin{align*}
\Ic_2(\Gamma_n(x))&= -  \frac 1 2  \iint\limits_{\mathbb{R}^2} |z| \exp \left\{ -\frac 1 2 z z^t \right\}      z \Gamma^2_n(x)  z^t  d z \\
&= -  \frac 1 2  \iint\limits_{\mathbb{R}^2} |z| \exp \left\{ -\frac 1 2 z z^t \right\}     \left( \gamma^2_{11;n}(x) z_1^2+\gamma^2_{12;n}(x) z_1^2 +2 \gamma_{11;n}(x) \gamma_{12;n}(x) z_1 z_2 \right.\\
&\;\; \left. +2 \gamma_{12;n}(x) \gamma_{22;n}(x) z_1 z_2 +\gamma^2_{12;n}(x) z_2^2+\gamma^2_{22;n}(x) z_2^2 \right)  d z   \\
&= -  \frac 1 2  \iint\limits_{\mathbb{R}^2} |z| \exp \left\{ -\frac 1 2 z z^t \right\}     \left\{ z_1^2 [\gamma^2_{11;n}(x) +\gamma^2_{12;n}(x)]  +z_2^2[ \gamma^2_{12;n}(x) +\gamma^2_{22;n}(x)] \right\}  d z   \\
&= - a_1 [\gamma^2_{11;n}(x)+\gamma^2_{12;n}(x)] - a_1 [\gamma^2_{22;n}(x)+\gamma^2_{12;n}(x)]= - a_1  \tr (\Gamma^2_n(x)).
\end{align*}

Finally,
\begin{equation}
\label{eq:Ic sum int}
\begin{split}
\Ic_3(\Gamma_n(x))&= \frac{1}{8}
\iint\limits_{\mathbb{R}^2} |z| \exp \left\{ -\frac{z z^t}{2} \right\}   \left(  z  \Gamma_n(x)   z^t  \right)^2 d z\\
&=\frac{1}{8}  \iint\limits_{\mathbb{R}^2} |z| \exp \left\{ -\frac{z z^t}{2} \right\}   \left[  \gamma^2_{11;n}(x) z_1^4+\gamma^2_{22;n}(x) z_2^4+4 \gamma^2_{12;n}(x) z_1^2 z_2^2 \right.\\
&\;\; \left.+2 \gamma_{11;n}(x) \gamma_{22;n}(x) z_1^2 z_2^2 +4 \gamma_{11;n}(x) \gamma_{12;n}(x) z^3_1 z_2+4 \gamma_{12;n}(x) \gamma_{22;n}(x) z_1  z^3_2 \right] d z\\
&=\frac{1}{8}  \iint\limits_{\mathbb{R}^2} |z| \exp \left\{ -\frac{z z^t}{2} \right\}   \left[  \gamma^2_{11;n}(x) z_1^4+\gamma^2_{22;n}(x) z_2^4+4 \gamma^2_{12;n}(x) z_1^2 z_2^2 \right. \\
&\;\; \left. +2 \gamma_{11;n}(x) \gamma_{22;n} (x)z_1^2 z_2^2  \right] d z.
\end{split}
\end{equation}
We write
\begin{equation}
\label{eq:int|z| Gauss}
\begin{split}
&\iint\limits_{\mathbb{R}^2} |z| \exp \left\{ -\frac 1 2 z z^t \right\} (z_1^2 + z_2^2)^2 d z = 2 \pi \int\limits_0^{\infty} \rho \exp \left\{ - \frac{\rho^2}{2} \right\} \rho^4 \rho d \rho = 2 \frac{15}{\sqrt 2} \pi^{3/2},\\
&\iint\limits_{\mathbb{R}^2} |z| \exp \left\{ -\frac 1 2 z z^t \right\}   z_1^4 d z =  \frac{15}{ \sqrt 2}  \frac{3}{4} \pi^{3/2},
\end{split}
\end{equation}
and, using these, we also have
\begin{equation}
\label{eq:int|z|z1^2z2^2 Gauss}
\begin{split}
\iint\limits_{\mathbb{R}^2} |z| \exp \left\{ -\frac 1 2 z z^t \right\}   z_1^2 z_2^2 d z
=  \frac{15}{\sqrt 2} \frac 1 4  \pi^{3/2}.
\end{split}
\end{equation}
Substituting \eqref{eq:int|z| Gauss}-\eqref{eq:int|z|z1^2z2^2 Gauss} into \eqref{eq:Ic sum int} we obtain:
\begin{align*}
\Ic_3(\Gamma_n(x))&=\frac{1}{8}  \frac{15}{4 \sqrt 2}  \pi^{3/2}  \left[  3 \gamma^2_{11;n}(x)  +  3 \gamma^2_{22;n}(x)  +4 \gamma^2_{12;n}(x)   +2 \gamma_{11;n}(x) \gamma_{22;n}(x)   \right]\\
&=\frac{1}{8}  \frac{15 \sqrt 2}{8}  \pi^{3/2}   \left[  (2 \gamma^2_{11;n}(x)   +  2 \gamma^2_{22;n}(x)   +4 \gamma^2_{12;n}(x))   +(\gamma^2_{11;n}(x)   +   \gamma^2_{22;n}(x)  \right.\\
&\;\; \left. +2 \gamma_{11;n}(x)  \gamma_{22;n}(x) )   \right]=\frac{15 \sqrt 2}{64}  \pi^{3/2}  \left[ 2 \tr(\Gamma_n^2(x)) + [\tr(\Gamma_n(x)) ]^2  \right],
\end{align*}
precisely as claimed.
\end{proof}

\section{Proof of Lemma \ref{21:33 Lon}: auxiliary results for the proof of Lemma \ref{lem:L1 int, Ups bnd}}
\label{sec: proof of lemma lem:L1 int, Ups bnd auxiliary results}

\begin{proof}

We first observe that
\begin{align*}
\int\limits_{\Qc} s_n(x) d x=0, \hspace{0.7cm} \int\limits_{\Qc} s^2_n(x) d x=\frac{4^2}{N^2_n} \sum_{\mu \in \Ec_n / \sim} \frac{5}{4}=\frac{5}{4} \frac{4^2}{N^2_n} \frac{N_n}{4}= \frac{5}{N_n}.
\end{align*}
Now we turn to
\begin{align*}
\tr(\Gamma_n(x))&= \frac{8}{n N_n}  \left[ b_{11;n}(x) +  b_{22;n}(x) \right] - \frac{128}{n N^2_n v_n(x)} \left[   d^2_{1;n}(x) +  d^2_{2;n}(x)\right].
\end{align*}
We observe that
\begin{align*}
\int\limits_{\Qc} b_{11;n}(x) d x =0, \hspace{0.7cm}  \int\limits_{\Qc} b_{22;n}(x) d x =0.
\end{align*}
Moreover, recalling that \cite[Lemma 2.3]{RW08}
$$\sum_{\mu \in \Ec_n / \sim} \mu_i \mu_j = \frac{n}{2} \frac{N_n}{4} \delta_{i,j},$$
we have
\begin{align*}
\int\limits_{\Qc} d^2_{2;n}(x) d x=\int\limits_{\Qc} d^2_{1;n}(x) d x = \frac{3}{16} \sum_{\mu \in \Ec_n / \sim} \mu_1^2= \frac{3}{16}  \frac{n}{2} \frac{N_n}{4}=\frac{3}{2^7} n N_n,
\end{align*}
and also
\begin{align*}
\int\limits_{\Qc} d^2_{1;n}(x)  s_n(x) d x = \int\limits_{\Qc} d^2_{2;n}(x)  s_n(x) d x=O(n).
\end{align*}

Consolidating all the above estimates, we write
\begin{align*}
&\int\limits_Q\tr(\Gamma_n(x)) d x =  - \frac{2^7}{n N^2_n }   \int\limits_{\Qc}\frac{1}{v_n(x)}  \left[    d^2_{1;n}(x) +  d^2_{2;n}(x)  \right] d x\\
&=  - \frac{2^7}{n N^2_n}   \int\limits_{\Qc}  \left[    d^2_{1;n}(x) +  d^2_{2;n}(x)  \right]  \left[1+  O(s_n(x)) \right] d x + O\left(N_{n}^{2-l}  |\Mcc_{l}(n)|\right)\\
 &= - \frac{2^7}{n N^2_n } \left( 2 \frac{3}{2^7} n N_n + O(n) \right)+ O\left(N_{n}^{-l-1}  |\Mcc_{l}(n)|\right) \\
 &=   -  \frac{6}{ N_n } + O \left(\frac{1}{N^2_n} \right) + O\left(N_{n}^{-l-1}  |\Mcc_{l}(n)|\right).
\end{align*}
And, since
\begin{align*}
 \int\limits_{\Qc} s_n(x) b_{11;n}(x) d x =  \int\limits_{\Qc} s_n(x) b_{22;n}(x) d x=\frac{4}{N_n} \sum_{\mu \in \Ec_n / \sim} \frac{\mu^2_1}{4}=\frac{4}{N_n} \frac{1}{4}  \frac{n N_n}{8}=\frac{n}{8},
\end{align*}
we have
\begin{align*}
&\int\limits_Q s_n(x) \tr(\Gamma_n(x)) d x \\
&=  \frac{8}{n N_n}  \int\limits_{\Qc} s_n(x) \left[ b_{11;n}(x) +  b_{22;n}(x) \right]  d x - \frac{2^7}{n N^2_n }   \int\limits_{\Qc}\frac{s_n(x)}{v_n(x)}  \left[    d^2_{1;n}(x) +  d^2_{2;n}(x)  \right] d x\\
 &=   \frac{8}{n N_n}  \int\limits_{\Qc} s_n(x) \left[ b_{11;n}(x) +  b_{22;n}(x) \right]  d x + O \left(\frac{1}{N^2_n} \right) + O\left(N_{n}^{-l-1}  |\Mcc_{l}(n)|\right)\\
 &=   \frac{8}{n N_n}  2 \frac{n}{8} +  O \left(\frac{1}{N^2_n} \right) + O\left(N_{n}^{-l-1}  |\Mcc_{l}(n)|\right)
 =   \frac{2}{ N_n}  +  O \left(\frac{1}{N^2_n} \right) + O\left(N_{n}^{-l-1}  |\Mcc_{l}(n)|\right).
\end{align*}
Moreover,
\begin{align*}
[\tr(\Gamma_n(x))]^2&= \frac{8^2}{n^2 N^2_n}  \left[ b^2_{11;n}(x) +  b^2_{22;n}(x)  + 2  b_{11;n}(x)  b_{22;n}(x) \right] \\
&\;\;+ \frac{128^2}{n^2 N^4_n v^2_n(x)} \left[   d^4_{1;n}(x) +  d^4_{2;n}(x) + 2 d^2_{1;n}(x)   d^2_{2;n}(x) \right]\\
&\;\; - 2 \frac{8}{n N_n} \frac{128}{n N^2_n v_n(x)}  \left[ b_{11;n}(x) +  b_{22;n}(x) \right] \left[   d^2_{1;n}(x) +  d^2_{2;n}(x)\right].
\end{align*}
and
\begin{align*}
\int\limits_{\Qc} [\tr(\Gamma_n(x))]^2 d x&=    \frac{8^2}{n^2 N^2_n} \int\limits_{\Qc} \left[ b^2_{11;n}(x) +  b^2_{22;n}(x)  + 2  b_{11;n}(x)  b_{22;n}(x) \right] dx \\
&\;\;+ \frac{128^2}{n^2 N^4_n}  \int\limits_{\Qc} \left[   d^4_{1;n}(x) +  d^4_{2;n}(x) + 2 d^2_{1;n}(x)   d^2_{2;n}(x) \right] [1+O(s_n(x))] dx\\
&\;\; - 2 \frac{8}{n N_n} \frac{128}{n N^2_n }  \int\limits_{\Qc} \left[ b_{11;n}(x) +  b_{22;n}(x) \right] \left[   d^2_{1;n}(x) +  d^2_{2;n}(x)\right] [1+O(s_n(x))]dx \\
&\;\;+  O\left(N_{n}^{-l-1}  |\Mcc_{l}(n)|\right),
\end{align*}
where
\begin{align*}
\int\limits_{\Qc} b^2_{11;n}(x) d x =\int\limits_{\Qc} b^2_{22;n}(x) d x = \frac{5}{4} \sum_{\mu \in \Ec / \sim } \mu_1^4 =: \frac{5}{4} n^2 N_n  M_4(n),
\end{align*}
and, since
\begin{align*}
\sum_{\mu \in  \Ec_n / \sim } \mu_1^2 \mu_2^2= \sum_{\mu \in  \Ec_n / \sim } \mu_1^2 ( \mu_2^2 \pm \mu_1^2) =  \sum_{\mu \in  \Ec_n / \sim  } \mu_1^2 ( n - \mu_1^2) = \frac{n^2  N_n}{8} - n^2 N_n  M_4(n),
\end{align*}
we have
\begin{align*}
\int\limits_{\Qc} b_{11;n}(x) b_{22;n}(x) d x = - \frac{3}{4} \sum_{\mu \in \Ec_n / \sim } \mu_1^2 \mu_2^2= - \frac{3}{4} \left( \frac{n^2  N_n}{8} - n^2 N_n  M_4(n)\right).
\end{align*}

Further, since for $i,j=1,2$,
\begin{align*}
 \sum_{\substack{ \eta, \mu \in \Ec_n / \sim \\ \eta \ne \mu} } \mu_j^2  \eta^2_i   =\sum_{\mu \in \Ec_n / \sim  }  \mu_j^2 \left( \sum_{\eta \in \Ec_n / \sim  }    \eta^2_i -  \mu_j^2 \right)=\frac{n^2 N^2_n}{8^2}- n^2 N_n M_4(n) ,
\end{align*}
 we have
\begin{align*}
\int\limits_{\Qc} d^4_{1;n}(x) d x&=\int\limits_{\Qc} d^4_{2;n}(x) d x = \frac{105}{1024} \sum_{\mu \in \Ec_n / \sim } \mu_1^4 + \frac{9}{256} \sum_{\substack{ \eta, \mu \in \Ec_n / \sim \\ \eta \ne \mu} } \eta^2_1 \mu_1^2 = O(n^2 N_n^2),
\end{align*}
\begin{align*}
\int\limits_{\Qc} d^2_{1;n}(x) d^2_{2;n}(x) d x& = \frac{25}{1024} \sum_{\mu \in \Ec_n / \sim } \mu_1^2 \mu_2^2 + \frac{9}{256} \sum_{ \substack{ \eta, \mu \in \Ec_n / \sim \\ \eta \ne \mu}  } \eta^2_2 \mu_1^2 = O(n^2 N_n^2),
\end{align*}
and similarly, for $i,j=1,2$,
\begin{align*}
\int\limits_{\Qc} b_{ii;n}(x) d^2_{j;n}(x) d x& = O(n^2 N_n^2),
\end{align*}
that implies
\begin{align*}
\int\limits_{\Qc} [\tr(\Gamma_n(x))]^2 d x&=    \frac{8^2}{n^2 N^2_n} \left[2 \frac{5}{4} n^2 N_n  M_4(n) - 2 \frac{3}{4} \left( \frac{n^2  N_n}{8} - n^2 N_n  M_4(n)\right)   \right] \\
&\;\;+O\left(\frac{1}{N^2_n} \right)+  O\left(N_{n}^{-l-1}  |\Mcc_{l}(n)|\right)\\
& =    \frac{8^2}{n^2 N^2_n} \left[2 \frac{5}{4} n^2 N_n  M_4(n) - 2 \frac{3}{4} \left( \frac{n^2  N_n}{8} - n^2 N_n  M_4(n)\right)   \right] \\
&=\frac{2^2}{N_n} \left(2^6 M_4(n)-3 \right)+O\left(\frac{1}{N^2_n} \right)+  O\left( N_{n}^{-l-1}  |\Mcc_{l}(n)|\right).
\end{align*}
\noindent Finally, since for a symmetric matrix $$A= \left( \begin{matrix} a&b \\ b&c \end{matrix}\right),$$
$\tr(A^2)=a^2+2b^2+c^2$, one has
\begin{align*}
\tr(\Gamma^2_n(x))& = \left[ \frac{8}{n N_n}   b_{11;n}(x) - \frac{128}{n N^2_n v_n(x)}   d^2_{1;n}(x) \right]^2 +  \left[ \frac{8}{n N_n}   b_{22;n}(x) - \frac{128}{n N^2_n v_n(x)}   d^2_{2;n}(x) \right]^2 \\
&\;\;+ 2 \left[ \frac{8}{n N_n}   b_{12;n}(x) - \frac{128}{n N^2_n v_n(x)}   d_{1;n}(x) d_{2;n}(x) \right]^2
\end{align*}
and, observing that
\begin{align*}
\int\limits_{\Qc}    b^2_{12;n}(x) d x = \frac{1}{4}  \sum_{\mu \in \Ec_n / \sim }  \mu_1^2 \mu_2^2 =  \frac{1}{4} \left( \frac{n^2 N_n}{8}-n^2 N_n M_4(n)\right),
\end{align*}
\begin{align*}
\int\limits_{\Qc}    b_{12;n}(x) d_{1;n}(x) d_{2;n}(x) d x =O(n^2 N_n),
\end{align*}
we have
\begin{align*}
&\int\limits_{\Qc} \tr(\Gamma^2_n(x)) dx = 2 \frac{8^2}{n^2 N^2_n} \left[ \int\limits_{\Qc} b^2_{11;n}(x) dx+ \int\limits_{\Qc} b^2_{12;n}(x) dx\right]+O\left(\frac{1}{N^2_n} \right)+ O\left(n^{-2} N_n^{-l-2}  |\Mcc_{l}(n)|\right)\\
&= 2 \frac{8^2}{n^2 N^2_n} \left[ \frac{5}{4} n^2 N_n M_4(n) + \frac{1}{4} \left( \frac{n^2 N_n}{8}-n^2 N_n M_4(n)\right) \right]
 +O\left(\frac{1}{N^2_n} \right)+ O\left(n^{-2} N_n^{-l-2}  |\Mcc_{l}(n)|\right) \\
&=\frac{2^2}{N_n} \left[1+2^5 M_4(n) \right] +O\left(\frac{1}{N^2_n} \right)+ O\left(n^{-1-1/2} N_n^{-l-3}  |\Mcc_{l}(n)|\right).
\end{align*}

To prove parts \ref{s^3} and \ref{trace_cubed} of Lemma \ref{21:33 Lon} we evaluate
\begin{align*}
\int_{\Qc} s^3_n(x) dx= \frac{4^3}{N_n^3} \sum_{\mu \in \Ec_n / \sim } \left( - \frac 3 2\right) = O\left( \frac{1}{N^2_n}\right),
\end{align*}
and
\begin{align*}
\int\limits_{\Qc} [ \tr( \Gamma_n(x)) ]^3 d x = O\left( \frac{1}{n^3 N^3_n} \sum_{\mu \in \Ec_n / \sim } \mu_1^6\right) =  O\left( \frac{1}{N^2_n}\right).
\end{align*}

\end{proof}

\end{document}